\documentclass[18pt,reqno]{amsart}
\usepackage{amsmath,amssymb,amsthm,hyperref}
\usepackage{stmaryrd}
\usepackage{enumerate}
\usepackage{graphicx}
\usepackage{float}
\usepackage{tikz}
\usepackage{url}
\usepackage{mathrsfs}
\usepackage{mathtools}
\usepackage{subcaption}

 \newtheorem{theorem}{Theorem}[section]
 \newtheorem{assumptions}[theorem]{Assumptions}
 \newtheorem{corollary}[theorem]{Corollary}
 
 \newtheorem{lemma}[theorem]{Lemma}
 \newtheorem{proposition}[theorem]{Proposition}

 \theoremstyle{definition}
 \newtheorem{definition}[theorem]{Definition}
 \newtheorem{remark}[theorem]{Remark}

\newcommand{\N}{\mathbb{N}}
\newcommand{\Z}{\mathbb{Z}}
\newcommand{\R}{\mathbb{R}}
\newcommand{\D}{\mathbb{D}}
\newcommand{\Q}{\operatorname{Q}}

\newcommand{\FF}{\mathcal{F}}

\newcommand{\E}{\mathbb{E}}

\newcommand{\dint}{{\rm d}}
\newcommand{\dTV}{\operatorname{d}_{\rm TV}}
\newcommand{\dKOL}{\operatorname{d}_{\rm Kol}}
\newcommand{\dWass}{\operatorname{d}_{\rm W^1}}
\newcommand{\HH}{\mathcal{H}}

\newcommand{\I}{\mathbb{I}}
\renewcommand{\P}{\mathbb{P}}
\newcommand{\dHH}{\operatorname{d}_{\HH}}

\newcommand{\CC}{\mathcal{C}}
\newcommand{\WW}{\mathcal{W}}

\renewcommand{\dint}{{\rm d}}

\renewcommand{\L}{\operatorname{L}}
\newcommand{\Id}{\operatorname{Id}}
\newcommand{\RR}{\operatorname{R}}
\newcommand{\Pt}{\operatorname{P}}

\newcommand{\Ga}{\Gamma}
\newcommand{\Var}{\operatorname{Var}}

\newcommand{\Tj}{\operatorname{T}}
\renewcommand{\S}{\operatorname{S}}

\newcommand{\LL}{\mathcal{L}}
\newcommand{\Gamarkov}{\boldsymbol{\Gamma}}
\newcommand{\Lmarkov}{\boldsymbol{\operatorname{L}}}
\newcommand{\Idmarkov}{\boldsymbol{\operatorname{Id}}}
\newcommand{\Kermarkov}{\boldsymbol{\operatorname{Ker}}}
\newcommand{\Cpol}[1]{\mathscr{C}_{\text{pol}}^{#1}(\mathbb{R})}
\newcommand{\projgeq}[1]{\operatorname{Proj}_{\geq #1}}
\renewcommand{\j}{\mathbf{j}}
\newcommand{\J}{\operatorname{J}}
\newcommand{\EE}[1]{\mathbf{E}_{\leq #1}}
\DeclareMathOperator{\prob}{\mathbf{P}}
\DeclareMathOperator{\gaus}{\mathbf{G}}
\DeclareMathOperator{\Dom}{\operatorname{Dom}}

 \DeclarePairedDelimiter{\norm}{\lVert}{\rVert}
 \numberwithin{equation}{section}
\title{Edgeworth expansion on Wiener chaos}

\author{Paul Mansanarez}
\address{
Paul Mansanarez, Nantes Universit\'e/Universit\'e libre de Bruxelles, France/Belgium. E-mail:
paul.mansanarez@ulb.be}
\author{Guillaume Poly}
\address{Guillaume Poly, Nantes Universit\'e, France. E-mail:guillaume.poly@univ-nantes.fr}
\author{Yvik Swan}
\address{
Yvik Swan, Universit\'e libre de Bruxelles/Vrije Universiteit  Brussel, Belgium. E-mail:
yvik.swan@ulb.be}

\begin{document}
\begin{abstract} Consider $F$ an element of the $p$-th Wiener chaos $\WW_p$, and denote by $\prob_F$ its law. 
For a positive integer $m$, let $\boldsymbol{\gamma}_{F,m}$ be the Radon measure with density  
\[
x \mapsto \frac{e^{-x^2/2}}{\sqrt{2\pi}} \left(1 + \sum_{k=3}^{4m-1} \frac{\E[H_k(F)]}{k!}\, H_k(x)\right),
\]  
where $H_k$ is the $k$-th Hermite polynomial. 
The main goal of this article is to prove that the total variation distance between $\prob_F$ and $\boldsymbol{\gamma}_{F,m}$ is of order $\Var(\Gamma(F,F))^{({m+1})/{2}}$, 
where $\Gamma(F,F)$ denotes the carr\'e-du-champ operator of $F$. 
The variance of $\Gamma(F,F)$ is known to govern Gaussian fluctuations and can be bounded from above by $\kappa_4(F)$, the fourth cumulant of $F$, as established in the seminal work \cite{NP2009a}. 
Our  result thus provides a genuine Edgeworth expansion in the setting of central convergence on Wiener chaoses. In this context, the quantity $\Var(\Gamma(F,F))$ plays the role of the small parameter that governs the accuracy of the approximation, in the same way that $1/\sqrt{n}$ does in the classical central limit theorem.

To the best of our knowledge, our work is the first to establish Edgeworth expansions for Wiener chaoses in full generality and at arbitrary order, 
together with explicit remainder bounds that systematically improve with the order of the expansion---exactly as one would expect from an Edgeworth approximation. 
Our results apply verbatim to every situation where a central limit theorem is available for chaos elements, 
since no structural assumption is required beyond belonging to a fixed Wiener chaos. 
As a byproduct, we recover the celebrated optimal fourth moment theorem from \cite{NP2015} by combining the expansions at the first and second orders, with sharper quantitative bounds. 
Previous works on Edgeworth expansions for Wiener chaoses were essentially restricted to the first order.  
Moreover, our proofs rely on the integration-by-parts formalism and the spectral properties of the Ornstein--Uhlenbeck operator, 
and they extend naturally to the broader framework of Markov chaoses.
\end{abstract}

\maketitle


\section{Introduction}

In the last decades, higher-order refinements of the central limit theorem have been intensively studied and are now relatively well understood, e.g.\ for sums of independent random variables or vectors. A central result in this regard is the celebrated \emph{Edgeworth expansion} which represents the distribution of normalized sums as the Gaussian law corrected by polynomial perturbations, with coefficients expressed in terms of the cumulants of the underlying distribution, and with remainders that become progressively smaller as the order of the expansion increases. Such expansions provide a refined description of how higher-order cumulants influence deviations from Gaussianity. They are particularly useful because they deliver explicit error rates and higher-order corrections to Gaussian approximations, making them central tools in applications such as bootstrap methods, hypothesis testing, and bias reduction for estimators (see, e.g., \cite{Hall1992} for a classical entry point in this huge literature).

Edgeworth expansions for sums of random variables  can be stated  as follows.  For $X_1, \ldots, X_n$ i.i.d. centred random variables with $\E[X_1^2]=1$, denoting $S_n \coloneq \sum_{k=1}^n X_k/\sqrt{n}$ the normalized sum, under mild assumptions on the law of $X_1$, one has
\begin{align}\label{edgeworth_Bhatta}
\left | \int_\R h(x) \, \P_{S_n}(\dint x) - \int_\R h(x) \left ( 1 + \sum_{k=1}^j n^{-k/2}Q_k(x) \right )\, \gamma(\dint x) \right |  \leq   {C_j}{n^{-\frac{j+1}{2}}}
\end{align}
for every function $h:\R \rightarrow \R$ in a specific class of functions (see \cite[Theorem 20.1]{BR1976} for a precise statement), where $\P_{S_n}$ is the law of $S_n$, $\gamma$ is the standard Gaussian measure (see Subsection \ref{ss:notations}), and $Q_l$ is a polynomial depending on the cumulants of $X_1$, up to order $l$. As an example, denoting $\kappa_l$ the $l$-th cumulant of $X_1$,
\begin{align*}
    Q_1(x) = -\frac{\kappa_3}{6}H_3(x), \quad \quad Q_2(x) = -\frac{\kappa_4}{24} H_4(x)-\frac{\kappa_3^2}{72}H_6(x),
\end{align*}
where $H_3, H_4, H_6$ are Hermite polynomials (see \eqref{hermite:rodriguez} below for a definition).
We refer to \cite{Cra1946,BR1976,Hall1992} and the references therein for a more detailed exposition, and to \cite{BC2016,AP2017} for more recent results.

Of course, Gaussian proximity is not restricted to partial sums of i.i.d.\ sequences, 
and one may expect an Edgeworth-type expansion whenever a quantitative gaussian approximation is available. 
We adopt the following general definition.

\begin{definition}[Edgeworth expansion, general form]
Let $\mathcal F$ be a family of real-valued random variables. 
Fix a class $\mathcal H$ of test functions on $\R$ (e.g.\ bounded Lipschitz, or $C_b^q$ for suitable $q$).
We say that $\mathcal F$ admits an \emph{Edgeworth expansion} (in the class $\mathcal H$) if there exists a gauge $v:\mathcal F\to(0,1]$ with $v(F)\to0$ along any sequence in $\mathcal F$ converging to $\gamma$ such that for each integer $j\ge1$, there exist polynomials 
$Q_{1,F},\dots,Q_{j,F}$  such that
\begin{equation}\label{edgeworth_gen}
\Big| \int_\R h(x)\,\P_F(\dint x) 
 - \int_\R h(x)\Big(1 + \sum_{k=1}^j Q_{k,F}(x)\Big)\, \gamma(\dint x) \Big|
 \;\le\; C_j\, v(F)^{\,j+1},
\end{equation}
for all $h\in\mathcal H$, where $C_j>0$ depends only on $j$ and on the choice of $\mathcal H$ (not on $F$); sums over empty sets are 0.
\end{definition}

Returning to our first example, for normalized sums $F = S_n=\sum_{k=1}^n X_k/\sqrt n$ with i.i.d.\ $X_i$, one can take \(v(S_n)=n^{-1/2}\), and the \(Q_{k,S_n}\) are, as stated above,  combinations of  Hermite polynomials with coefficients expressed via cumulants.  
The picture is far less complete and understood  outside the  setting of partial sums, and much remains to be discovered particularly  for U-statistics or $p$-multilinear forms evaluated in independent random variables, that is, random variables of the form
\begin{align}\label{eq:pmultilinearforms}
    \sum_{1\leq i_1 < \cdots < i_p \leq M} \, a(i_1, \ldots, i_p) \, X_{i_1} \cdots X_{i_p},
\end{align}
where the $(X_i)_i$ are independent.
Indeed, in such non-linear cases, one cannot rely anymore on  Fourier techniques which are at the very heart of many proofs of Edgeworth expansions. Attempts to extend Edgeworth-type expansions to these settings are relatively scarce and, as far as we can tell,  always restricted to first-order (i.e.\ $j = 1$ in \eqref{edgeworth_gen}), see \cite{DH2023,BG2022,TY2019,PY2016,NP2009b,JW2003,BGZ1997,BGZ1986,GH1978,Hipp1977}. In those instances, the proofs are context-dependent and lean on assumptions on the law of the random variables.

In this paper we aim to propose an Edgeworth-type expansion of the form~\eqref{edgeworth_gen} in the important context of  elements of the $p$-th Wiener  $\WW_p$. 
Here, a natural  gauge is 
\[
v(F)\;:= \sqrt{\kappa_4(F)}
\]
the (square root of the) fourth cumulant of $F$  which, from the celebrated Fourth Moment Theorem (see below)  is known  to control Gaussian proximity: for $N\sim\mathcal N(0,1)$ and any admissible test function $h$,
\[
\big|\E[h(F)]-\E[h(N)]\big|\;\le\;C\,\sqrt{\kappa_4(F)},
\]
with a constant $C$ depending only on the test class. 
In other words, in this context,  $\sqrt{\kappa_4(F)}$ plays the same role here as $n^{-1/2}$ in the classical i.i.d.\ CLT.
Before delving into the details of our results, let us first describe in more detail the framework we are dealing with in this article.

\subsection{History of the Fourth Moment Theorem, and the Malliavin-Stein method}
A central line of research in stochastic analysis concerns the study of normal approximations of functionals of Gaussian fields. In this direction, the \emph{Fourth Moment Theorem} of ~\cite{NP2005} has played a decisive role. It states that a sequence of normalized random variables $(F_n)_n$ living in a fixed Wiener chaos converges in distribution to the standard Gaussian if and only if its fourth cumulant $\kappa_4(F_n)$ tends to zero. This remarkable characterization reduces a complex convergence-in-law question to a simple moment condition.

Soon after, quantitative versions of this theorem were obtained through the combination of Malliavin calculus and Stein’s method, now commonly referred to as the \emph{Malliavin--Stein method} (see e.g. ~\cite{NP2009a, NP2012}). This approach provides sharp bounds in metrics such as Wasserstein or total variation, expressed in terms of 
\begin{equation}
    v(F) = \Var  \Gamma(F,F) \label{vF}
\end{equation}
($\Gamma$ is the carr\'e-du-champ operator), a  quantity which known to be    equivalent to $\kappa_4(F)$ (see, e.g., \cite[Lemma 5.2.4]{NP2012}). These developments established a deep and fruitful connection between analysis on Wiener space, probabilistic approximations, and Stein’s method, and have since found numerous applications in fields as diverse as mathematical finance, random matrix theory, and geometric probability (see the dedicated website\footnote{\url{https://sites.google.com/site/malliavinstein/home}} for a detailed overview of this line of research).

While the Fourth Moment Theorem and the Malliavin--Stein method address whether and how fast convergence to normality occurs, they do not provide information on the \emph{structure of the error term}. Still, the question of Edgeworth expansion for functionals of a Gaussian field has been studied (see e.g. \cite{NP2009b,NP2015,PY2016,KP2018}), but only for a one-term correction, the main reference being \cite{NP2009b}. Furthermore, it needs several additional conditions on the sequence $(F_n)_n$.


\subsection{Main result}

In this article,  we construct, for a normalized random variable 
\(F \in \WW_p\), an explicit signed measure \(\boldsymbol{\gamma}_{F,m}\) whose density is the Gaussian density corrected by Hermite polynomials up to order \(4m-1\). Our main theorem shows that the total variation distance between \(\prob_F\) and \(\boldsymbol{\gamma}_{F,m}\) is controlled by 
$v(F)^{\frac{m+1}{2}}$, where \(v(F)\) is given in \eqref{vF}. More precisely, we will prove the following.  
\begin{theorem}\label{maintheorem}
Let $m$ be a positive integer. Then there exists $C_{p,m}>0$ such that for all $F\in \WW_p$ with $\E[F^2]=1$,
    \begin{align}\label{eq:maintheorem}
    \dTV(\prob_F,\boldsymbol{\gamma}_{F,m}) & \leq C_{p,m} \bigl (\Var  \Ga(F,F) \bigr )^{\frac{m+1}{2}}
\end{align}
where $\boldsymbol{\gamma}_{F,m}$ is the signed measure on $\R$ with density
\begin{align*}
    x \longmapsto \frac{e^{-x^2/2}}{\sqrt{2\pi}} \left ( 1 + \underset{k=3}{\overset{4m-1}{\sum}} \, \frac{\E[H_k(F)]}{k!}H_k(x) \right ).
\end{align*}
\end{theorem}
 This yields a genuine analogue of the Edgeworth expansion in Wiener chaos, valid to any order \(m\), with remainders expressed in terms of the natural small parameter of the theory, namely the fourth cumulant, since one has $\Var \Ga(F,F) \leq \kappa_4(F)/3$ (see, e.g., \cite[Lemma 5.2.4]{NP2012}).

 As a direct consequence of Theorem \ref{maintheorem}, we reprove the optimal version of the fourth moment theorem:
\begin{proposition}\label{prop:optimalm1}
    There exists $c_p >0$ such that for every $F\in \WW_p$ with $\E[F^2]=1$,
    \begin{align}\label{eq:optimalm1}
        c_p\Var\Ga(F,F) & \leq \dTV(\prob_{F},\boldsymbol{\gamma}_{F,1}).
    \end{align}
    
\end{proposition}
Furthermore, since the proof relies only on properties of the Ornstein-Uhlenbeck operator in dimension one, following the insightful observation from \cite{Led2012}, we are able to provide a version of Theorem \ref{maintheorem} in the setting of diffusive symmetric Markov operators. We refer to Section \ref{s:markovoperator} for more details.

Finally, also as a byproduct of Theorem \ref{maintheorem}, following the works \cite{NPR2010,MOO2010,HMP2025}, we provide an Edgeworth-type expansion for random variables of the form \eqref{eq:pmultilinearforms}, assuming a condition of matching moments. We refer to Appendix \ref{ss:matching} for more details.

\subsection{Comparison with literature}
To the best of our knowledge, no analogue of higher-order expansions such as \eqref{eq:maintheorem} has been proved in the context of Wiener chaoses. Let us describe synthetically the main references dealing with Edgeworth expansions for Wiener chaos related sequences.
\begin{itemize}
\item In \cite{NP2009b}, a one-term Edgeworth expansion is obtained in \cite[Theorem 3.1, Proposition 3.3]{NP2009b} for a larger class of functionals ($F_n \in \D^{1,2}$). However, the additional law assumption (iii) of \cite[Theorem 3.1]{NP2009b} does not hold in general, even for Wiener chaoses: there exist normalized sequences $(F_n)_n$ lying in a fixed Wiener chaos that converge to the standard Gaussian and where this stated condition is not verified. For instance, in $\WW_2$, by a diagonalizing argument, one can always write $F_n = \sum_{i=1}^{+\infty}\lambda_{i,n}(G_i^2-1)$ where $(G_i)_i$ are i.i.d. random variables with standard Gaussian distribution and $\sum_{i=1}^{+\infty}\lambda_{i,n}^2 = 1/2$. In addition, we can assume without loss of generality that $|\lambda_{1,n}| > |\lambda_{2,n}|> \cdots $. The CLT for $(F_n)_n$ is then equivalent to the condition $\lambda_{1,n} \rightarrow 0$. However, this does not imply the fact that $\lambda_{1,n}^4 = o \left ( \sum_{i=1}^{+\infty} \lambda_{i,n}^4 \right )$, which means that $\Gamma(F_n,-\L^{-1}F_n)-1$ does not have Gaussian fluctuations when properly rescaled. Hence (iii) fails to be true in general.

\item In \cite{BBNP2012,NP2015}, the authors establish optimal forms of the Fourth Moment Theorem: first in a smooth test-function distance \cite[Theorem~1.9]{BBNP2012}, and then in total variation \cite[Theorem~1.2]{NP2015}. 
A key ingredient is \cite[Proposition~3.11]{BBNP2012}, which provides an $m$-th order asymptotic expansion for smooth functionals $F\in\D^\infty$, with a remainder controlled by $\E[\Gamma_m(F)]$, where $\Gamma_m$ denotes the $m$-fold iterated carr\'e-du-champ (see \cite[Definition~3.6, Proposition~3.7]{BBNP2012}). 
Moreover, \cite[Proposition~4.3]{BBNP2012} yields explicit bounds on $\E[\Gamma_m(F)]$ for $m\in\{2,3,4\}$ in terms of powers of the fourth cumulant $\kappa_4(F)$. 
Combining these estimates, they derive an optimal fourth moment theorem by analyzing what amounts to a first-order Edgeworth correction. 
We note that the proof of \cite[Proposition~3.11]{BBNP2012} is predominantly combinatorial and relies on contraction operators (see \cite[Appendix~B.4]{NP2012}), a route we entirely avoid in our approach.

\item Building on \cite{Yos2013}, \cite{PY2016} establishes a first-order Edgeworth expansion for a broad class of non-linear functionals of diffusion processes. 
Their result remains restricted to the first-order correction and is obtained in a setting that differs substantially from ours.

\item By analyzing characteristic functions, \cite{TY2019} derives Edgeworth expansions to arbitrary finite order \cite[Theorem~4.2]{TY2019}, under fairly stringent assumptions on the remainder (e.g., Assumption~A2)—precisely the delicate point in general, as already illustrated by \cite[Proposition~4.3]{BBNP2012} and as further emphasized in the present work. 
They also apply their result to the quadratic variation of a Gaussian process solving the one-dimensional stochastic wave equation driven by space–time white noise. 
In particular, this example lies in the second Wiener chaos and can be represented as a linear combination of independent chi-square variables, thus fitting naturally into the framework of linear statistics of independent random variables.

\item In \cite{KP2018}, the authors also derive Edgeworth-type expansions to arbitrary finite order for more general functionals 
($F\in \D^{J+2,2^{J+2}}$, $J\in\N$). However, they control only the first-order remainder.

\item 
Our method of  proof relies on a variant of Stein’s method. 
While Stein-type arguments have already been used to derive Edgeworth expansions see, e.g., \cite{Bar1986,RR2003,Rot2005,Fat2021,FL2025}, these results pertain to normalized sums. 
To the best of our knowledge, the present work is the first to employ the Malliavin–Stein method to obtain Edgeworth expansions on Wiener chaoses.
\end{itemize}

\subsection{Simulation study}
To illustrate the applicability of our expansion, we present numerical simulations from two distinct models, both involving the central convergence of normalized sequences in the same Wiener chaos. In order to highlight the gains achieved through the Edgeworth expansion, we compare the probability density functions (PDFs) of the corresponding variables. Following \cite{HNTX2015,HMP2024}, it is known that in this setting a superconvergence phenomenon takes place. We illustrate this effect by examining not only the PDFs but also their first and second derivatives. The resulting plots clearly show that as additional terms are incorporated into the Edgeworth expansion, the approximating densities (and their derivatives) increasingly capture the fine fluctuations of the underlying random variables.

\subsubsection{Hermite variations of discrete-time fractional Brownian motion}
\label{sec:hermvar}
For a fixed integer $p \geq 2$, we consider the following model:
\begin{align}
\begin{cases}\label{model:hermitevar}
X_k & = B_{k+1}^H - B_k^H, \\ 
V_n & = \displaystyle \frac{1}{\sqrt{n}} \sum_{k=0}^{n-1} H_p(X_k),
\end{cases}
\end{align}
where $B^H$ denotes a fractional Brownian motion with Hurst index $H \in (0,1)$, that is, the centered Gaussian process on $\R$ with covariance function
\begin{align*}
   \forall (t,s) \in \R^2, \quad  
   \E[B_t^H B_s^H] 
   = \frac{1}{2} \Bigl( |t|^{2H} + |s|^{2H} - |t-s|^{2H} \Bigr).
\end{align*}
The normalized sequence
\[
S_n \coloneq \frac{V_n}{\sqrt{\Var(V_n)}}
\]
satisfies $\E[S_n^2]=1$ and belongs to the $p$-th Wiener chaos $\WW_p$.  
The celebrated Breuer--Major theorem (see, e.g., \cite{BM1983}) provides a necessary and sufficient condition on the parameter $H$ ensuring the central convergence of $(S_n)_n$ towards the standard Gaussian distribution.  
For further details and references, we refer to \cite[Chapter~7]{NP2012} and \cite{NP2015}.

Since the model in~\eqref{model:hermitevar} lends itself to numerical investigation, we conducted simulations for a range of parameter configurations, and the corresponding results are displayed in Figure~\ref{fig:fbm_all}. 

The first panel illustrates the critical regime $H = 1 - \tfrac{1}{2p}$, where a large sample size ($n=1000$) is required due to the logarithmic rate of convergence; in this case, we restrict to $m=1$ as numerical instabilities prevent higher-order computations.

The second panel corresponds to a non-critical setting, where much smaller values of $n$ suffice and stable behaviour is already visible at order $m=2$. 

The third panel displays the case of small $H$ with $p=3$, while the fourth panel corresponds to a larger value of $H$ with $p=4$.

\begin{figure}[H]
  \centering
\includegraphics[width=\textwidth]{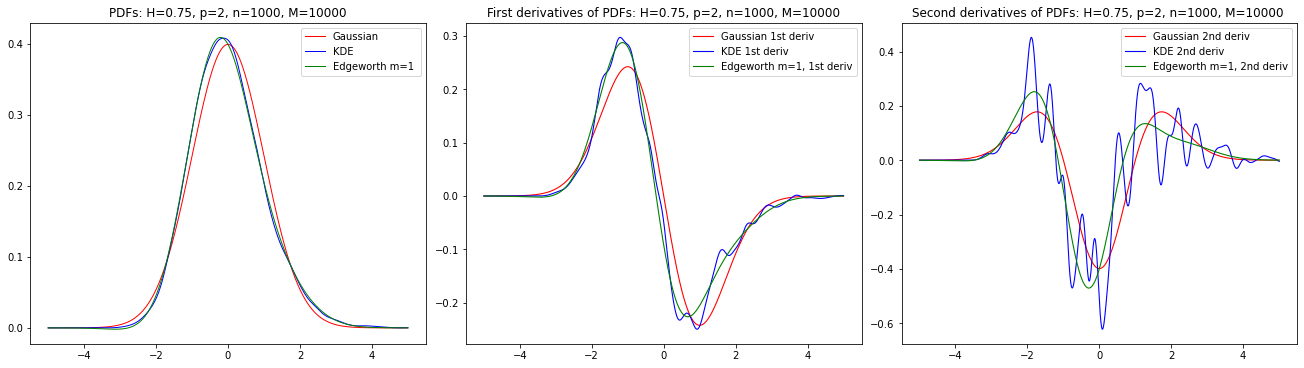}\\[0.3cm]
\includegraphics[width=\textwidth]{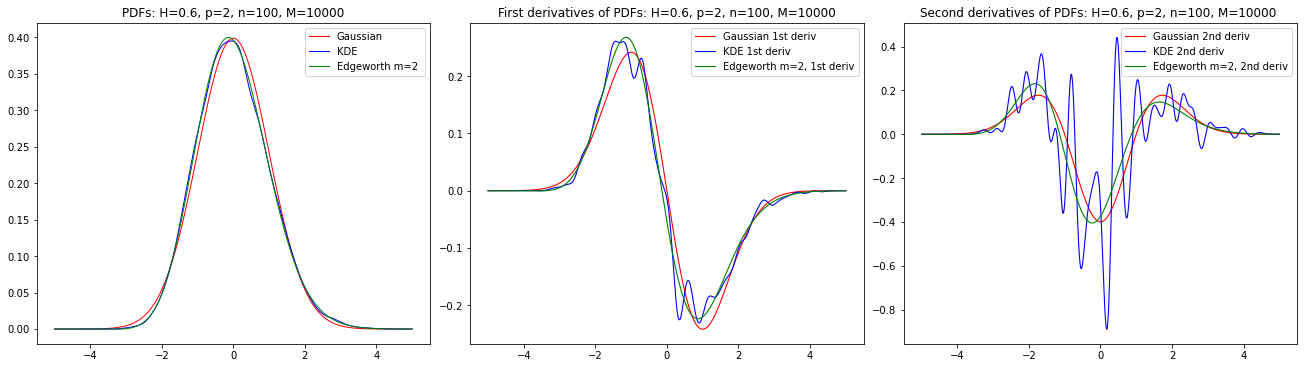}\\[0.3cm]
\end{figure}
\begin{figure}[H]
\includegraphics[width=\textwidth]{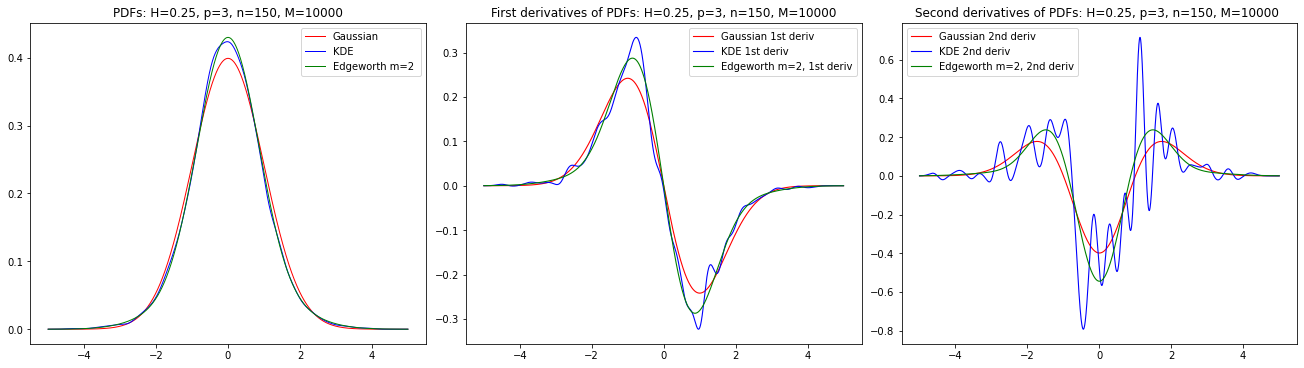}\\[0.3cm]
\includegraphics[width=\textwidth]{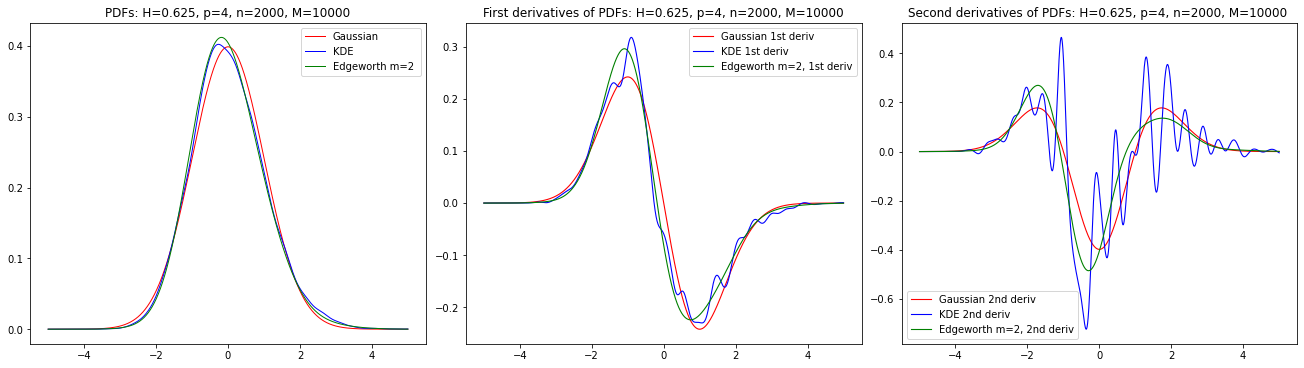}
  \caption{Simulated densities (first column) and their first and second derivatives (second and third columns) for the Gaussian distribution (red), the KDE estimate (blue), and the Edgeworth expansion from Theorem~\ref{maintheorem} (green), in the case where $F$ is a fractional Brownian motion as described in Section~\ref{sec:hermvar}.
 }
  \label{fig:fbm_all}
\end{figure}

\subsubsection{Traces of GOE}
\label{sec:goe}
We now turn to the Gaussian Orthogonal Ensemble (GOE). Recall that an element $A^{(n)}$ of the standard $n$-GOE is a symmetric random matrix 
\[
A^{(n)} = (A^{(n)}_{ij})_{1 \leq i,j \leq n}
\]
of size $n \times n$ defined by the following properties:
\begin{itemize}
    \item for all $i<j$, the entries satisfy $A^{(n)}_{ij} = A^{(n)}_{ji}$;
    \item for all $i \neq j$, the off-diagonal entries are independent and distributed as $\mathcal{N}(0,1)$;
    \item for all $i$, the diagonal entries are independent and distributed as $\mathcal{N}(0,2)$.
\end{itemize}
We define the rescaled matrix 
\[
A_n \coloneq \frac{1}{\sqrt{n}} A^{(n)} .
\]

It is well known (see, e.g., \cite{AZ2006}) that for any fixed integer $m \geq 1$, the centered and normalized trace statistics of powers of $A_n$ satisfy central limit theorems. In particular, for $m=3$, one has
\begin{align}\label{eq:CLT_trace3}
    \operatorname{Tr}(A_n^3) \;\;\xrightarrow[n \to \infty]{\;\;d\;\;}\;\; \mathcal{N}(0,3).
\end{align}
In the more recent work \cite{NP2010}, devoted to universality results for the convergence of trace statistics of general non-Hermitian matrices, the authors show in particular that the convergence in~\eqref{eq:CLT_trace3} is dictated by the central convergence of the projection of $\operatorname{Tr}(A_n^3)$ onto the third Wiener chaos $\WW_3$. This is the sequence of random variables we simulate in this subsection.

\begin{figure}[H]
  \centering
  \includegraphics[width=\textwidth]{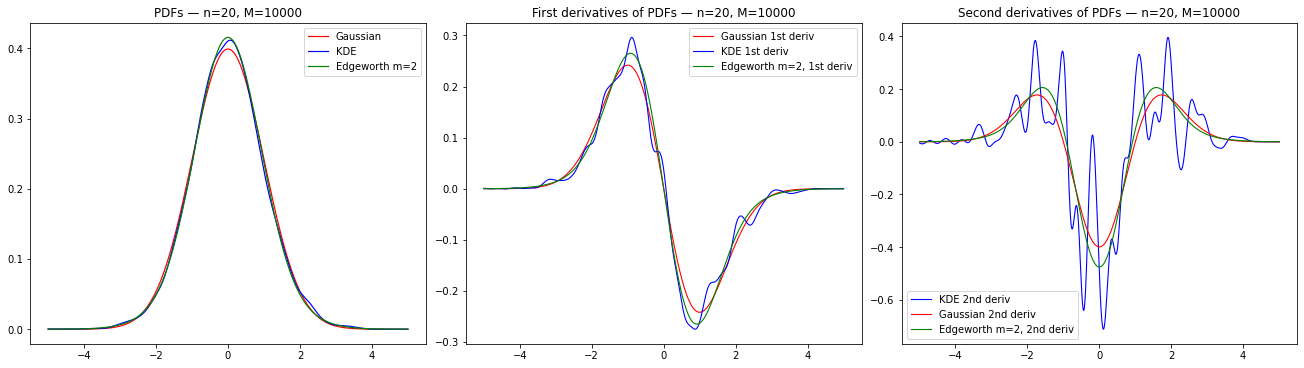}\\[0.3cm]
  \includegraphics[width=\textwidth]{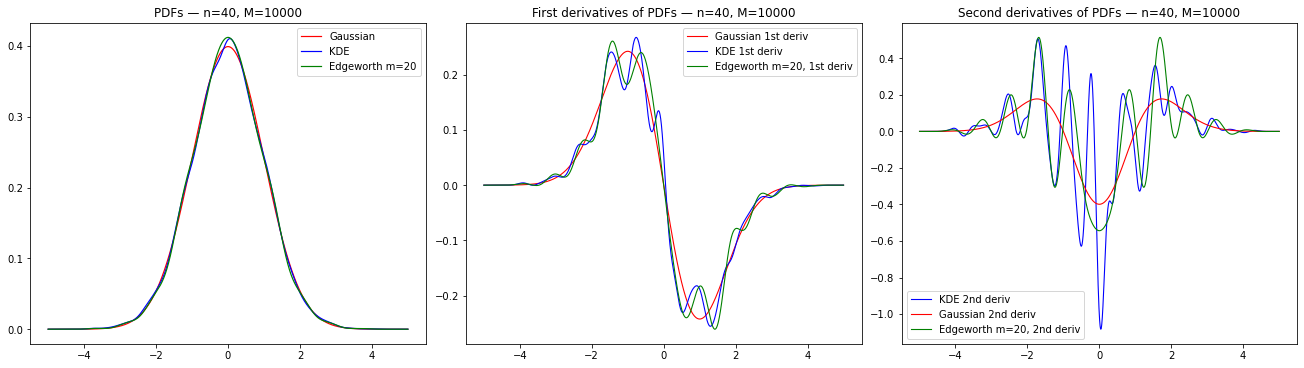}\\[0.3cm]
  \includegraphics[width=\textwidth]{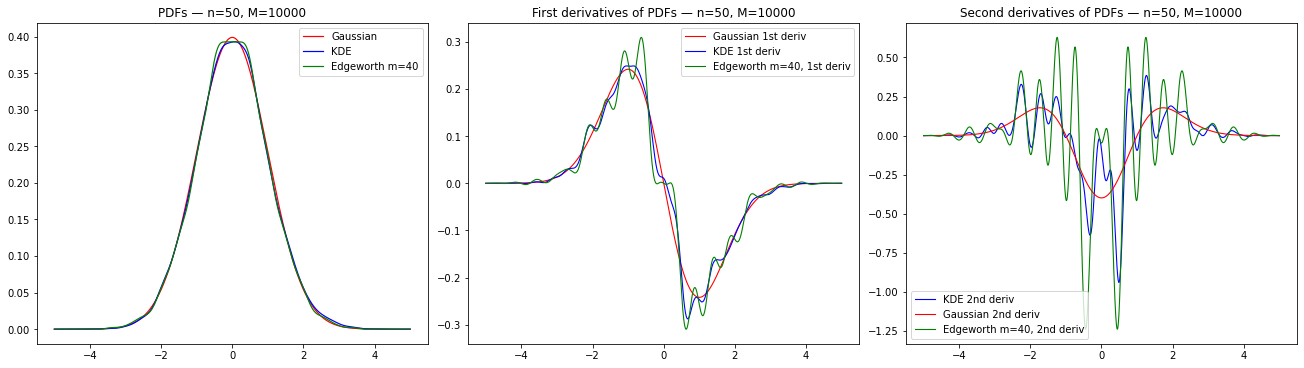}
  \end{figure}
  \begin{figure}[H]
  \includegraphics[width=\textwidth]{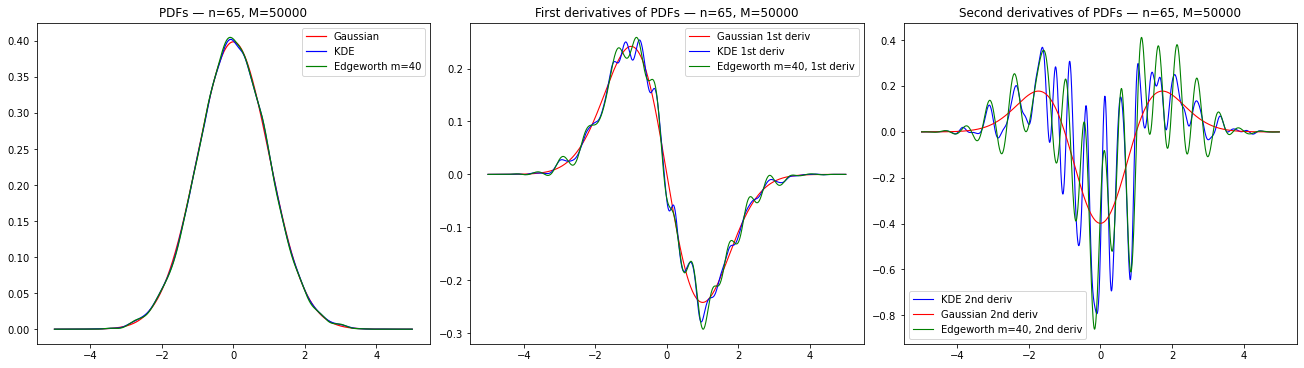}
  \caption{Simulated densities (first column) and their first and second derivatives (second and third columns) for the Gaussian distribution (red), the KDE estimate (blue), and the Edgeworth expansion from Theorem~\ref{maintheorem} (green), with $F$ as specified in Section~\ref{sec:goe}.
}
  \label{fig:GOE_all}
\end{figure}

We briefly comment on the illustrations in Figure~\ref{fig:GOE_all}. In contrast to the FBM case discussed in the previous section, numerical stability here allows us to consider much larger values of $m$, and the Edgeworth expansion provides accurate approximations even for the second derivative. In the last two rows, we observe that for $m=40$ the quality of the approximation improves as the sample size increases (from $n=50$ in the second to last line to $n=65$ last line).

\section{Structure of the paper}

The paper is structured as follows. In Section \ref{s:prolegomena}, we provide notations and remind useful results that will be used throughout the article. In Section \ref{s:overview}, we explain the main  strategy of the proof of Theorem \ref{maintheorem}. In Section \ref{s:m=1}, we prove Theorem \ref{maintheorem} in the case of $m=1$ and put in place the building blocks for the general case. In Section \ref{s:higherorder}, we prove Theorem \ref{maintheorem} in full generality. In Section \ref{s:markovoperator}, we remind materials on chaoses of a diffusive Markov operator and extend Theorem \ref{maintheorem} in this broader framework. In the appendix, we provide a proof for Proposition \ref{prop:optimalm1} and a result for Edgeworth-type expansions in the case of homogeneous sums of fixed degree.

\section{Preliminaries}\label{s:prolegomena}

In this section, we provide the necessary definitions, notations, and preliminary lemmas required for the proof of the main theorem.

\subsection{Notations}\label{ss:notations}

\begin{itemize}

\item The set of nonnegative integers is denoted by $\N$, and the set of positive integers by $\N^*$. For $a,b$ integers such that $a\leq b$, we will denote by $\llbracket a, b \rrbracket$ the set $\{a, a+1, \ldots,b\}$. For a real number $\alpha$, we will also denote by $\lceil \alpha \rceil$ (resp. $\lfloor \alpha \rfloor$) the ceiling (resp. floor) function evaluated at $\alpha$ that is, the smallest (resp. biggest) integer $k$ such that $k-1 < x \leq k$ (resp. $k\leq x <k+1$).

\item We denote by $\mathscr{C}_b^0(\R)$ the set of continuous functions $h:\R \rightarrow \R$ that are bounded.

\item For a non-negative integer $m$, we denote by $\Cpol{m}$ the set of $m$-times differentiable functions $f$ such that $f^{(m)}$ is continuous, and for all $k \in \{0,\ldots,m\}$, $f^{(i)}$ has polynomial growth, that is, there exist $C_i>0$ and $d_i \in \N$ verifying $\forall x \in \R, \; |f^{(i)}(x)| \leq C_i(1+|x|^{d_i})$. The space $\Cpol{m}$ is naturally endowed with a structure of algebra over $\R$. Finally, we define $\Cpol{\infty} \coloneq \underset{m\in\N}{\bigcap} \Cpol{m}$.

\item For every $m$ positive integer, we write $f^{(m)}$ the $m$-th derivative of $f$. For $m=1,2$, we will also use the classical notations $f', f''$ for the first and second derivatives. If $f$ has two variables $(t,x)$, we will still denote by $\partial_xf$ and $\partial_x^2 f$ the first and second partial derivative with respect to the (space) variable $x$, and $\partial_t f$ the first partial derivative with respect to the (time) variable $t$.

\item The identity operator on a vector space $E$ is denoted by $\Id_E$.

\item If $J$ is a linear operator from a vector space $E$ into itself, for $m \in \N$, we will denote by $J^m$ the composition $J\circ \cdots \circ J$ where $J$ appears $m$ times. By convention $J^0$ is the identity operator $\Id_E$. Furthermore, if $P = \prod_{k=1}^m (X-\alpha_k)$ is a polynomial, with $\alpha_1, \cdots, \alpha_m$ real numbers, and for all $k\in \{1,\ldots,m\}$, $(J-\alpha_k\Id_E)$ is invertible, we will denote by $P(J)^{-1}$ the composition $ (J-\alpha_1\Id_E)^{-1} \circ \cdots \circ (J-\alpha_m\Id_E)^{-1}$. Using the classical linear algebra notation, we will write $\alpha \Id_E = \alpha$, for all scalar $\alpha$ and we will also sometimes denote multiplicatively $J(u) = Ju $, for all $u \in E$.

\item Given a probability space $(X,\mathcal{A},\nu)$ and $q\geq1$, we will denote by $L^q(\nu)$ the vector space of $\mathcal{A}$-measurable functions $f:X \rightarrow \R$ such that 
\[\norm{f}_q^q \coloneq \int_X |f(x)|^q \, \nu(\dint x) <+\infty.\]

\item We will denote by $\gamma$ the standard Gaussian distribution on $\R$, that is, the measure with density $x \mapsto e^{-x^2/2}/\sqrt{2\pi}$ with respect to the Lebesgue measure on $\R$. Additionally, we will denote by $N$ a random variable with distribution $\gamma$. Since $\gamma$ admits moments of all orders, we have the inclusion $\Cpol{m} \subseteq \underset{q\geq 1}{\bigcap} L^q(\gamma)$, for all $m \in \N$.

\end{itemize}

\subsection{Ornstein-Uhlenbeck semigroup analysis on $L^2(\gamma)$}\label{s:wiener-chaos1}

\subsubsection{The one-dimensional Wiener space}
Consider the probability space $(\R, \mathscr{B}(\R), \gamma)$ where $\gamma$ is the standard Gaussian distribution on $\R$. The \emph{Hermite polynomials}, defined as
\begin{equation}\label{hermite:rodriguez}
  H_{k}(x) \coloneq (-1)^{k} \mathrm{e}^{x^{2}/2} \frac{\mathrm{d}^{k}}{\mathrm{d} x^{k}} \mathrm{e}^{-x^{2}/2}, \qquad k \in \mathbb{N},
\end{equation}
form an orthogonal basis of $L^{2}(\gamma)$, the space of square integrable real-valued functions with respect to the measure $\gamma$, and verify $ \norm{H_k}_2 = \sqrt{k!}$, for all $k\in \N$. These polynomials are eigenfunctions of the classical Ornstein-Uhlenbeck operator $\LL$, defined on twice differentiable functions $f:\R \rightarrow \R$ by \begin{equation}\label{eq:LOU_1D}
    \LL(f)(x) \coloneq f''(x) - xf'(x).
\end{equation} This is a self-adjoint operator on $L^2(\gamma)$, with invariant reversible measure $\gamma$. Namely, for every $n \in \N$, $H_n$ has eigenvalue $-n$, that is \begin{align}\label{hermite:eigenfunction}
     \forall x \in \R, \quad H_n''(x)-xH'_n(x) = -nH_n(x).
\end{align}
From the definition \eqref{hermite:rodriguez}, one can derive the two elementary properties listed in the following lemma.

\begin{lemma}\label{prop:gaussian_ipp}
Let $N \sim \mathcal{N}(0,1)$ and $\phi \in L^2(\gamma)$ such that $\phi'\in L^2(\gamma)$. Then,
    \begin{enumerate}[(i)]
        \item for every $n\in\N$, $ H_{n+1}'=(n+1)H_n$ ;
        \item for all $ \phi \in \Cpol{1}$ and $\psi \in \Cpol{2}$,
        \begin{align}\label{eq:gaussian_ipp}
        \E[\phi'(N) \psi'(N)] &= - \E[\phi(N) (\LL \psi)(N)].
        \end{align}
    \end{enumerate}
\end{lemma}
As a direct consequence of~\eqref{eq:gaussian_ipp}, one obtains the relation
\begin{align}\label{eq:gaussian_ipp_hermite}
    \forall n \in \N, \quad \E \left[ H_{n+1}(N)\,\phi(N) \right] 
    = \E \left[ H_{n}(N)\,\phi'(N) \right],
\end{align}
valid for all test functions $\phi \in \Cpol{1}$.

Since the sequence $(H_n/\sqrt{n!})_{n \geq 0}$ forms an orthonormal basis of $L^2(\gamma)$, every function $\phi \in L^2(\gamma)$ can be expanded as
\begin{align}
    \phi = \sum_{n=0}^{+\infty} c_n(\phi) H_n,
\end{align}
where 
\[
c_n(\phi) := \frac{1}{n!}\,\E \left[ H_n(N)\,\phi(N) \right].
\]
In particular, if $\phi \in \Cpol{1}$, relation~\eqref{eq:gaussian_ipp_hermite} implies the following identity between the coefficients of $\phi$ and those of its derivative:
\begin{align}\label{eq:rel_hermite_coef_derivative}
    c_{n+1}(\phi) = \frac{1}{n+1}\,c_n(\phi').
\end{align}

In order to invert polynomials of the operator $\LL$, we recall the following classical notion.

\begin{definition}\label{def:hermite_rank}
The \emph{Hermite rank} of $\phi \in L^2(\gamma)$ is defined as
\begin{align}\label{eq:hermite_rank}
    \operatorname{rank}(\phi) \coloneq 
    \min \bigl\{ k \in \N \;:\; c_k(\phi) \neq 0 \bigr\}.
\end{align}
\end{definition}

\begin{lemma}\label{prop:hermite_rank}
Let $\phi \in L^2(\gamma)$ with Hermite rank $r > 0$. Then for every polynomial $Q$ of degree at most $r-1$, one has
\[
\E[Q(N)\phi(N)] = 0.
\]
\end{lemma}

\begin{proof}
By Definition~\ref{def:hermite_rank}, one has 
\[
\forall k \in \{0,\dots,r-1\}, \quad \E[H_k(N)\phi(N)] = k!\,c_k(\phi) = 0.
\]
Since $(H_0,\ldots,H_{r-1})$ spans the space of polynomials of degree at most $r-1$, the result follows by linearity of the expectation.
\end{proof}

We have the following property of conservation of Hermite rank.
\begin{lemma}\label{prop:conv_rank_LOU}
    Let $\varphi \in L^2(\gamma)$ have Hermite rank $r$. Then, for every $\alpha \neq -r$, $(\LL-\alpha)\varphi$ has Hermite rank $r$.
\end{lemma}
\begin{proof}
    Let $k \in \N$. Then, by symmetry of $\LL$, one has
    \begin{align*}
        k!\,c_k\bigl ( (\LL-\alpha)\varphi \bigr ) & =\E[H_k(N) (\LL-\alpha) \psi(N)] =\E[(\LL-\alpha) H_k(N) \psi(N)] \\
        &= -(k+\alpha)k!\, c_k(\varphi).
    \end{align*}
    The result follows.
\end{proof}
For a positive integer $m$, we define the projection $\projgeq{m}(\phi)$ by  
\begin{align}\label{def:proj}
     \projgeq{m}(\phi) := \phi - \sum_{k=0}^{m-1} c_k(\phi) H_k.
\end{align}

The operator $\projgeq{m}$ enjoys the following continuity and stability property.  

\begin{lemma}\label{Projcontinuite}
Let $m \geq 1$ be fixed. Consider a sequence $(\phi_n)_{n \geq 1}$ of functions in $\Cpol{0}$ converging pointwise to a function $\phi$, and assume that there exist a constant $C>0$ and an integer $d \geq 0$ such that
\[
\forall n \in \N, \ \forall x \in \R, \quad |\phi_n(x)| \leq C(1+|x|^d).
\]
Then the sequence $\bigl( \projgeq{m}(\phi_n) \bigr)_n$ converges pointwise to $\projgeq{m}(\phi)$.
\end{lemma}

\begin{proof}
Fix $x \in \R$. By~\eqref{def:proj}, one has
\[
\projgeq{m}(\phi_n)(x) = \phi_n(x) - \sum_{k=0}^{m-1} c_k(\phi_n) H_k(x).
\]
By assumption, for every $k \in \{0,\ldots,m-1\}$,
\begin{align}\label{aux:Projcontinuite}
    \forall n \in \N, \quad 
    |\phi_n(N) H_k(N)| \leq C (1+|N|^d) |H_k(N)| \quad \text{a.s.},
\end{align}
and the right-hand side of~\eqref{aux:Projcontinuite} is integrable since $N$ has finite moments of all orders.  
Moreover, $(\phi_n(N))_n$ converges almost surely to $\phi(N)$, so by the dominated convergence theorem,
\[
c_k(\phi_n) = \frac{1}{k!}\E[\phi_n(N)H_k(N)] \;\xrightarrow[n\to\infty]{}\; \frac{1}{k!}\E[\phi(N)H_k(N)] = c_k(\phi).
\]
Therefore, for every $x \in \R$, 
\[
\projgeq{m}(\phi_n)(x) \;\xrightarrow[n\to\infty]{}\; \projgeq{m}(\phi)(x),
\]
which completes the proof.
\end{proof}

\begin{lemma}\label{Projpol}
    Let $m$ be a positive integer. Let $\phi$ be an element of $\Cpol{0}$ verifying
    \begin{align*}
        \forall x \in \R, \quad |\phi(x)| \leq A(1+|x|^d)
    \end{align*}
    for some $A>0$ and some nonnegative integer $d$. Then there exists $B_{m,d}>0$ such that
    \begin{align*}
        \forall x \in \R, \quad |\projgeq{m}(\phi)(x)| \leq AB_{m,d}(1+|x|^{\max(d,m-1)})
    \end{align*}
    where $B_{m,d}$ does not depend on $\phi$. Furthermore, for all $k \in \N$, if $\phi \in \Cpol{k}$ then $\projgeq{m}(\phi)\in \Cpol{k}$. 
\end{lemma}
\begin{proof}
    Let $x$ be a real number. Again, using \eqref{def:proj}, one has
    \begin{align*}
        \projgeq{m}(\phi)(x) := \phi(x) - \underset{k=0}{\overset{m-1}{\sum}} \, c_k(\phi_n)H_k(x).
    \end{align*}
    In particular, for all $k\in \N$, if $\phi \in \Cpol{k}$ then $\projgeq{m}(\phi)\in \Cpol{k}$, because $\projgeq{m}(\phi)$ is a sum of $\phi$ with a polynomial.
    \begin{align*}
        \bigl |\projgeq{m}(\phi)(x) \bigr | & \leq |\phi(x)| + \underset{k=0}{\overset{m-1}{\sum}} \, |c_k(\phi_n)||H_k(x)|\\
        & \leq A(1+|x|^d) + \underset{k=0}{\overset{m-1}{\sum}} \, \frac{A}{k!}\,\E[(1+|N|^d)|H_k(N)|]|H_k(x)|,
    \end{align*}
    which yields the result, since $H_k$ is a polynomial of degree $k$ and that, for positive integers $n \leq l$, one has $\forall x \in \R, \quad |x|^n \leq 1+|x|^l$.
\end{proof}
\begin{lemma}\label{prop:conv_and_hermiterank}
    Let $(\phi_n)_n$ be a sequence of elements of $\Cpol{0}$ of Hermite rank $r$, converging pointwise to a function $\phi\in L^2(\gamma)$, and verifying
    \begin{align*}
        \forall n \in \N, \;\forall x \in \R, \quad |\phi_n(x)| \leq A(1+|x|^d)
    \end{align*}
    for some $A>0$ and some nonnegative integer $d$ which do not depend on $n$. Then $\phi$ has Hermite rank greater or equal to $r$.
\end{lemma}
\begin{proof}
    Let $k \in \N$. One has 
    \begin{align}\label{aux1:conv_and_hermiterank}
        c_k(\phi_n) \underset{n\rightarrow +\infty}{\longrightarrow} c_k(\phi).
    \end{align}
    Indeed, by assumptions, one has $\phi_n(N) \underset{n\rightarrow +\infty}{\longrightarrow} \phi(N) \; \text{ a.e.}$, and 
    \begin{align}\label{aux2:conv_and_hermiterank}
        \forall n \in \N, \quad |H_k(N)\phi_n(N)| \leq A(1+|N|^d)|H_k(N)|,
    \end{align}
    where the right-hand side of \eqref{aux2:conv_and_hermiterank} is integrable, since $N$ has moments of all order and $H_k$ is a polynomial. By Lebesgue's dominated convergence theorem, one gets 
    \begin{align*}
        \E[H_k(N)\phi_n(N)] \underset{n\rightarrow +\infty}{\longrightarrow} \E[H_k(N)\phi(N)].
    \end{align*}
    that is \eqref{aux1:conv_and_hermiterank}. To conclude, one has to notice that, in light of \eqref{aux1:conv_and_hermiterank}, $c_k(\phi)$ is zero when $c_k(\phi_n)$ is.
\end{proof}
\subsubsection{The Ornstein-Uhlenbeck semigroup} 
The operator $\LL$ is the infinitesimal generator of a semigroup $(\Pt_t)_{t \geq 0}$ on $L^2(\gamma)$.  

\begin{definition}
The \emph{Ornstein--Uhlenbeck semigroup} $(\Pt_t)_{t \geq 0}$ is defined as follows: for $\phi \in \Cpol{0}$ and $x \in \R$,  
\begin{align}\label{eq:mehler}
    \Pt_t \phi(x) 
    & \coloneq \int_\R \phi \left(xe^{-t} + \sqrt{1-e^{-2t}}\,y\right) \, \gamma(\dint y) \\
    & = \E \left[ \phi \left(xe^{-t} + \sqrt{1-e^{-2t}}\,N \right) \right], \nonumber
\end{align}
where $N \sim \mathcal{N}(0,1)$ and $t \geq 0$.
\end{definition}

The following propositions summarize some basic properties of the Ornstein--Uhlenbeck semigroup; we refer to \cite{BGL2013} for detailed proofs and further background.

\begin{proposition}
The family $(\Pt_t)_{t\geq0}$ is a strongly continuous semigroup on $L^2(\gamma)$ with infinitesimal generator $\LL$. For every $\phi \in \Cpol{0}$, one has:
\begin{enumerate}
    \item $\displaystyle \lim_{t \to 0} \Pt_t \phi = \phi$ ;
    \item $\displaystyle \lim_{t \to +\infty} \Pt_t \phi = \int_\R \phi(x) \, \gamma(\dint x)$ ;
    \item if $\phi$ is bounded, then for all $t \geq 0$, $\Pt_t\phi$ is bounded and 
    \[
    \norm{\Pt_t\phi}_\infty \leq \norm{\phi}_\infty ;
    \]
    \item for all $x \in \R$ and $t>0$, 
    \begin{align}\label{eq:kolmo}
        \partial_t \Pt_t \phi(x) = \LL(\Pt_t \phi)(x) 
        = \partial_x^2 \Pt_t \phi(x) - x \,\partial_x \Pt_t \phi(x) ;
    \end{align}
    \item for all $t \geq 0$, 
    \[
    \Pt_t \phi = \sum_{k=0}^{+\infty} e^{-kt} c_k(\phi) H_k,
    \]
    where the series converges in $L^2(\gamma)$.
\end{enumerate}
\end{proposition}

\begin{proposition}
Let $\phi \in \Cpol{1}$. The following commutation relations hold:
\begin{enumerate}
    \item for all $t \geq 0$, $\Pt_t \phi$ is differentiable and
    \begin{align}\label{eq:commutationPt}
        \partial_x \Pt_t \phi = e^{-t} \Pt_t \phi' ;
    \end{align}
    \item if $\phi$ is three times differentiable, then $\LL \phi$ is differentiable and 
    \begin{align}\label{eq:commutationLOU}
        \forall x \in \R, \quad (\LL \phi)'(x) = (\LL - 1)\phi'(x).
    \end{align}
\end{enumerate}
\end{proposition}

Furthermore, we have the following representation of the derivatives.
\begin{proposition}\label{form:DPt}
    Let $\phi\in \Cpol{0}$. Then for all $t>0$,  $\Pt_t \phi \in \mathscr{C}^\infty(\R)$, and 
\begin{align}\label{eq:DPt}
    \partial_x^k\Pt_t \phi(x) = \frac{e^{-kt}}{(1-e^{-2t})^{k/2}} \int_\R H_k(y) \phi \bigl ( xe^{-t}+(1-e^{-2t})^{1/2} y \bigr ) \, \gamma(\dint  y)
\end{align}
for every $x \in \R$. In particular, $t \longmapsto \partial_x \Pt_t \phi(x)$ is integrable on $]0,+\infty[$.
\end{proposition}
\begin{proof}
    Fix $t>0$, $x \in\R$ and set $\sigma_t \coloneq \sqrt{1-e^{-2t}}$. Then using Mehler's formula \eqref{eq:mehler}, and the change of variable $u=\sigma_t y$, one has
\begin{align*}
\Pt_t \phi(x) &= \int_\R \phi \bigl (xe^{-t}+(1-e^{-2t})^{1/2}y \bigr ) \, \frac{e^{-\frac{y^2}{2}} }{\sqrt{2\pi}} \,\dint y 
 = \int_\R \phi  (xe^{-t}+u ) \, \frac{e^{-\frac{u^2}{2\sigma_t^2}} }{\sqrt{2\pi}\sigma_t} \,\dint u.
\end{align*}
Then, writing $g(v) := \phi(-v)$, one gets
\begin{align*}
\Pt_t \phi(x) &= \int_\R \phi  (xe^{-t}+u) \, \frac{e^{-\frac{u^2}{2\sigma_t^2}} }{\sqrt{2\pi}\sigma_t} \,\dint u = \int_\R g  (-xe^{-t}-u  ) \, \frac{e^{-\frac{u^2}{2\sigma_t^2}} }{\sqrt{2\pi}\sigma_t} \,\dint u \\
& =  g \ast \rho_t (-xe^{-t}),
\end{align*}
where $\rho_t$ is the function $u \longmapsto e^{-\frac{u^2}{2\sigma_t^2}}/\sqrt{2\pi}\sigma_t$. Since $\rho_t$ lies in $\mathscr{C}^\infty(\R)$, by property of the convolution product, we get that $\Pt_t \phi$ is infinitely differentiable. Furthermore, for every $k \in \N$, one has 
\begin{align}\label{c}
     \partial_x^k \Pt_t \phi(x)  & = (-1)^ke^{-kt} \, g \ast (\partial_x^k \rho_t) (-xe^{-t}) \nonumber \\
  & = \frac{e^{-kt}}{\sigma_t^k} \int_\R g  (-xe^{-t}-u  ) H_k \left ( \frac{u}{\sigma_t} \right ) \rho_t(u) \,\dint u.
\end{align}
Using the change of variables $y=u/\sigma_t$ again in \eqref{c} yields the result. 
\end{proof}

\begin{lemma}\label{prop:Pt_Cpol}
    Let $\phi$ be a continuous function such that $\forall x \in \R, \; |
    \phi(x)| \leq C(1+|x|^d)$, with $C>0$ and $d \in \N$. Then, for all $k\in \N$, there exists $A_{k,d}>0$ such that
    \begin{align*}
        \forall t \geq 1, \forall x \in \R, \quad |\partial_x^k \Pt_t \phi(x) | \leq e^{-kt} CA_{k,d}(1+|x|^d).
    \end{align*}
\end{lemma}

\begin{proof}
    Let $t \geq 1$  and $x \in \R$. Denote by $\sigma_t$ the quantity $\sqrt{1-e^{-2t}}$. Using Proposition \ref{form:DPt} and the triangular inequality, one has
    \begin{align*}
         |\partial_x^k\Pt_t \phi(x)| & \leq  \frac{e^{-kt}}{\sigma_t^k} \int_\R |H_k(y) | |\phi  ( xe^{-t}+\sigma_t y )  |\, \gamma(\dint  y) \\
         & \leq C\frac{e^{-kt}}{\sigma_t^k} \int_\R |H_k(y) | (1+|xe^{-t}+\sigma_t y|^d)\, \gamma(\dint  y) \\
        & \leq C\frac{e^{-kt}}{\sigma_t^k} \int_\R |H_k(y) | (1+2^{d-1}|x|^de^{-dt}+2^{d-1}\sigma_t^d |y|^d)\, \gamma(\dint  y) \\
        & \leq e^{-kt}CA_{k,d}(1+|x|^d),
    \end{align*}
    where $A_{k,d} \coloneq 2^{d-1}\sigma_1^{d-k} \E \bigl [ |N|^d|H_k(N)| \bigr ]$.
\end{proof}

\begin{proposition}\label{prop:controle_semigroup}
    Let $r \in \N$ and let $\phi$ be a continuous function such there exist $C>0$ and $d\in\N$, $\forall x \in \R, \quad |\phi(x)| \leq C(1+|x|^d)$. Then, there exist $M_r>0$, such that
    \begin{align}\label{eq:controle_semigroup}
        \forall t \geq 1, \forall x \in \R, \quad \left|\Pt_t\phi(x) - \sum_{k=0}^{r-1} c_k( \Pt_t\phi)H_k(x) \right | \leq e^{-rt}CM_{r,d}(1+|x|^{d+r}).
    \end{align}
\end{proposition}
\begin{proof}
    Fix $t\geq 1$ and denote by $g$ the function $x \mapsto \Pt_t\phi(x) - \sum_{k=0}^{r-1} c_k( \Pt_t\phi)H_k(x)$. The function $g$ has Hermite rank greater than $r-1$, and $\forall x \in \R, \; g^{(r)}(x) = \partial_x^r \Pt_t \phi(x)$. By Lemma \ref{prop:Pt_Cpol}, there exists $A_{r,d}>0$ such that
    \begin{align}\label{auxa}
        |g^{(r)}(x)| \leq e^{-rt} CA_{r,d}(1+|x|^d).
    \end{align}
    Fix $x \in \R$ and let $y \neq x$ (suppose for instance, $y>x$). Let us apply the mean value theorem to $g^{(r-1)}$ between $x$ and $y$: there exists $c_{x,y} \in ]x,y[$ such that
    \begin{align*}
        |g^{(r-1)}(x) - g^{(r-1)}(y)| & = |x-y||g^{(r)}(c_{x,y})| \leq |x-y|e^{-rt} CA_{r,d}(1+|c_{x,y}|^d) \\
        & \leq (|x|+|y|)e^{-rt} CA_{r,d}(1+(|x|+|y|)^d) \\
        & \leq (|x|+|y|)e^{-rt} CA_{r,d}(1+2^{d-1}|x|^d+2^{d-1}|y|^d).
    \end{align*}
    Since $g$ has Hermite rank greater than $r-1$, integrating over $R$ with respect to $y$ and the Gaussian measure gives $\int_R g^{(r-1)}(y)\, \gamma( \dint y) = 0$. With the triangular inequality and \eqref{auxa}, this reaps
    \begin{align*}
        |g^{(r-1)}(x)| & \leq e^{-rt} CM_{r,d}(1+|x|^{d+1})
    \end{align*}
    where the constant $M_{r,d}>0$ can be made explicit. The argument can be iterated.
\end{proof}

\begin{remark}
    In particular, when $\phi$ has Hermite rank $r$ above in \eqref{eq:controle_semigroup}, one has 
    \begin{align}\label{eq:controle_semigroup_hermite}
         \forall t \geq 1, \forall x \in \R, \quad \left|\Pt_t\phi(x)  \right | \leq e^{-rt}CM_{r,d}(1+|x|^{d+r}).
    \end{align}
\end{remark}

\subsubsection{Resolvent operator}

The operator $\LL$ has spectrum equal to the set of nonpositive integers $\Z_-$. Its resolvent is defined by 
\[
\RR(\alpha) := (\LL-\alpha \Id_\R)^{-1},
\]
which is a bounded operator on $L^2(\gamma)$ for every $\alpha \in \R \setminus \Z_-$. Moreover, when $\alpha \in \Z_-$, the operator $\LL$ can still be inverted on suitable subspaces of $L^2(\gamma)$. For further background on resolvent operators of $\LL$, we refer to \cite{BGL2013} and \cite{EN1999}. The next proposition collects the key properties of the inversion of $\LL$; since these are not entirely classical, we provide a proof.  

\begin{proposition}\label{prop:resolvent}
Let $\alpha \in \R$ and let $\phi \in \Cpol{0}$ with Hermite rank $r > -\alpha$. Then:
\begin{enumerate}
    \item the function $\RR(\alpha)\phi$ is well-defined, continuous, and satisfies the spectral representation
    \begin{align}\label{eq:resolvent_spectral}
        \forall x \in \R, \quad
        \RR(\alpha)\phi(x) = -\sum_{k=r}^{+\infty} \frac{c_k(\phi)}{k+\alpha}\,H_k(x),
    \end{align}
    where the series converges normally on compact subsets of $\R$;
    \item one also has the following integral representation:
    \begin{align}\label{eq:resolvent_integrale}
        \forall x \in \R, \quad 
        \RR(\alpha)\phi(x) = -\int_0^{+\infty} e^{-\alpha t} \Pt_t \phi(x)\,\dint t.
    \end{align}
\end{enumerate}
\end{proposition}

\begin{proof}
Since $r > -\alpha$, the series 
\[
g = \sum_{k=r}^{+\infty}\frac{c_k(\phi)}{k+\alpha}H_k
\]
is well-defined and converges in $L^2(\gamma)$. Clearly, $(\LL - \alpha)g = \phi$.  

\medskip
\noindent \textit{Step 1. Normal convergence on compacts.}  
Since $\sum_{k=r}^{+\infty} c_k(\phi)^2 k! < \infty$, one has $|c_k(\phi)| < 1/\sqrt{k!}$ for large enough $k$. We now recall the following classical result of Szeg\H{o}.  

\begin{theorem}[Plancherel--Rotach asymptotics, Theorem~8.22.9 in \cite{Sze1975}]\label{thm:plancherel-rotach}
For $x = \sqrt{2n+1}\cos \theta$, with $0 < \theta < \pi$, one has
\begin{align*}
    e^{-x^2/2} H_n(\sqrt{2}x)
    = \frac{2^{1/4}\sqrt{n!}}{(\pi n)^{1/4}\sin(\theta)^{1/2}}
    \Biggl( 
        \sin \Bigl( \bigl(\tfrac{n}{2}+\tfrac{1}{4}\bigr)(\sin(2\theta)-2\theta) + \tfrac{3\pi}{4} \Bigr) 
        + O \left(\tfrac{1}{n}\right)
    \Biggr),
\end{align*}
where the error term $O(n^{-1})$ is uniform in $x$.
\end{theorem}

Fix $M>0$ and $x \in [-M,M]$. Choose $k_0$ such that $\sqrt{2k_0+1} > M$. For $k \geq k_0$, one can write $x = \sqrt{2k+1}\cos \theta_k$ with $0<\theta_k<\pi$. By Theorem~\ref{thm:plancherel-rotach}, we obtain
\begin{align*}
    |H_k(\sqrt{2}x)| 
    \leq \frac{e^{M^2/2} \,2^{1/4}\sqrt{k!}}{(\pi k)^{1/4}\sin(\theta_k)^{1/2}}
    \biggl( \bigl|\sin(\cdot)\bigr| + O \left(\tfrac{1}{k}\right) \biggr).
\end{align*}
By bounding $\theta_k$ away from $0$ and $\pi$ uniformly in $k$, one deduces the existence of a constant $K_M>0$ such that
\begin{align}\label{hermite sup bound}
    \sup_{|x|\leq M} |H_k(x)| \leq K_M \frac{\sqrt{k!}}{k^{1/4}}, \quad \forall k \geq k_0.
\end{align}
Therefore, for $k \geq \max(k_0,r)$ and $|x|\leq M$,
\[
\frac{|c_k(\phi)|}{|k+\alpha|}\,|H_k(x)| 
\leq \frac{K_M}{(k+\alpha)k^{1/4}} \sim \frac{K_M}{k^{5/4}}.
\]
Thus the series in~\eqref{eq:resolvent_spectral} converges uniformly on compact sets, showing that $\RR(\alpha)\phi$ is continuous.

\medskip
\noindent \textit{Step 2. Integral representation.}  
For each $k \geq r$, observe that
\[
\frac{1}{k+\alpha} = \int_0^{+\infty} e^{-(k+\alpha)t}\,\dint t.
\]
Hence
\begin{align*}
    \sum_{k=r}^{+\infty} \frac{c_k(\phi)}{k+\alpha}H_k(x) 
    &= \sum_{k=r}^{+\infty}\int_0^{+\infty} c_k(\phi)e^{-(k+\alpha)t}H_k(x)\,\dint t \\
    &= \int_0^{+\infty} e^{-\alpha t} \sum_{k=r}^{+\infty} c_k(\phi)e^{-kt}H_k(x)\,\dint t \\
    &= \int_0^{+\infty} e^{-\alpha t} \Pt_t \phi(x)\,\dint t,
\end{align*}
where the interchange of sum and integral is justified by the uniform convergence on compacts established in Step~1. This completes the proof.
\end{proof}

A consequence of the latter properties is the next lemma about the conservation of Hermite rank.

\begin{lemma}\label{prop:cons_rank}
  Let $\phi \in L^2(\gamma)$ with Hermite rank $d$. Then $\RR(\alpha)\phi$ has Hermite rank $d$, for every $\alpha \in \bigl (\R \setminus \Z \bigr ) \, \cup \, ]-d,0]$.
\end{lemma}
\begin{proof}
    Let $\alpha\in \bigl (\R \setminus \Z \bigr ) \, \cup ]-d,0]$. Then, using the equality \eqref{eq:resolvent_spectral}, for every $k \in \{0,\ldots,d-1\}$, $c_k\bigl ( \RR(\alpha) \phi \bigr ) =0$,
    because $c_k(\phi)=0$ for $k<d$, and $c_d\bigl ( \RR(\alpha) \phi \bigr ) \neq 0$ since $c_d(\phi)\neq 0$.
\end{proof}

\begin{corollary}\label{prop:fractionLOU_hermite_rank}
    Let $\alpha_1, \ldots, \alpha_l$ be real numbers, and denote by $Q$ the rational fraction $\frac{1}{(X-\alpha_1)\cdots (X-\alpha_l)}$. For all $\phi \in \Cpol{0}$ with Hermite rank $r> -\underset{1\leq i \leq r}{\min} \alpha_i$, $Q(\LL)\phi$ has the same Hermite rank as $\phi$.
\end{corollary}
\begin{proof}
    By definition, $Q(\LL) = \RR(\alpha_1) \cdots \RR(\alpha_r)$. By Lemma \ref{prop:cons_rank}, each $\RR(\alpha_i)$ conserves Hermite rank, hence so do $Q(\LL)$.
\end{proof}
The resolvent operator $\RR(\alpha)$ also verifies a commutation relation.
\begin{proposition}\label{prop:resolvent_commutation}
    Let $\alpha$ be a real number. For every $\phi$ in $\Cpol{1}$ with Hermite rank $r > -\alpha$, $\RR(\alpha)\phi$ is differentiable and
\begin{align}\label{eq:resolvent_commutation}
     \bigl ( \RR(\alpha) \phi \bigr )'& =\RR(\alpha+1) \phi'.
\end{align}
\end{proposition}
\begin{proof}
    Let $\phi$ be an element of $\Cpol{1}$ with Hermite rank $r>-\alpha$. Then $\phi'$ has Hermite rank $r-1$ and 
    \begin{align}
        \RR(\alpha+1)\phi' = \sum_{k=r-1}^{+\infty} \frac{c_k(\phi')}{k+1+\alpha} H_k = \sum_{k=r}^{+\infty} \frac{c_{k-1}(\phi')}{k+\alpha} H_{k-1} = \sum_{k=r}^{+\infty} \frac{c_{k}(\phi)}{k+\alpha} H'_{k}.
    \end{align}
    By Proposition \ref{prop:resolvent}, the series $\sum_{k=r}^{+\infty} \frac{c_{k}(\phi)}{k+\alpha}H_{k}$ and  $\sum_{k=r}^{+\infty} \frac{c_{k}(\phi)}{k+\alpha} H'_{k} = \sum_{k=r-1}^{+\infty} \frac{c_{k}(\phi')}{k+1+\alpha} H_{k}$ converge normally on every compact set of $\R$. By theorem, we get that $\RR(\alpha) \phi$ is differentiable and 
    \begin{align}
        \bigl ( \RR(\alpha) \phi \bigr )' = \sum_{k=r}^{+\infty} \frac{c_{k}(\phi)}{k+\alpha} H'_{k} = \RR(\alpha+1) \phi'.
    \end{align}
\end{proof}

We end the section with two lemmas: an estimate and a continuity property for  resolvent operators of $\LL$.
\begin{lemma}\label{Rpol}
Let $\alpha \in \R$. Let $\phi$ be a continuous function with Hermite rank $r>-\alpha$, verifying  $\forall x \in \R, \;|\phi(x)| \leq C(+|x|^d)$, for some $C>0$. Then, there exists $V_{r,d,\alpha} >0$ such that
\[\forall x \in \R, \quad
|\RR(\alpha) \phi(x) | \leq  CV_{r,d,\alpha}(1+|x|^{d+r}).
\]
\end{lemma}
\begin{proof}
Let $x \in \R$. Using  \eqref{eq:resolvent_integrale}, one has
\begin{align*}
 \RR(\alpha)\phi(x) &= - \int_{0}^{+\infty} e^{-\alpha t} \Pt_{t} \phi(x) \, \dint t =-\int_{0}^{1} e^{-\alpha t} \Pt_{t} \phi(x) \, \dint t -\int_{1}^{+\infty} e^{-\alpha t} \Pt_{t} \phi(x) \, \dint t.
\end{align*}
Hence,  using Lemma \ref{prop:Pt_Cpol} and Proposition \ref{eq:mehler}, there exist $A_{0,d}>0$ and $M_{r,d}>0$ such that
\begin{align*}
| \RR(\alpha) \phi(x)|  &\leq C \left ( A_{0,d}\int_0^1 e^{-\alpha t} \, \dint t +M_{r,d}\int_{1}^{+\infty}  e^{-(\alpha+r) t}\, \dint t \right )(1+|x|^{d+r}) \\
& \leq  CV_{r,d,\alpha} (1+|x|^{d+r}),
\end{align*} 
where $V_{r,d,\alpha} \coloneq    A_{0,d}\int_0^1 e^{-\alpha t} \, \dint t +\frac{e^{-\alpha-r}M_{r,d}}{\alpha+r}$.
\end{proof}

\begin{remark}
    In particular, from Propositions \ref{prop:resolvent_commutation} and \ref{Rpol}, one deduces that for every $\alpha \in \R$, every integer $d>0$ and every $\phi \in \Cpol{d}$ with Hermite rank $r>-\alpha$, $\RR(\alpha) \phi$ belongs to $\Cpol{d}$.
\end{remark}

\begin{lemma}\label{Rcontinuite}
    Let $\alpha$ be a real number and consider a sequence $(\phi_n)$ in $\Cpol{0}$ of Hermite rank greater or equal to $r>-\alpha$, converging pointwise to a function $\phi$. Suppose that there exist a constant $C$ and a nonnegative integer $d$ such that
    \[\forall n \in \N, \forall x \in \R, \quad |\phi_n(x)| \leq C(1+|x|^d).\]
    Then the sequence $(\RR(\alpha)\phi_n)_n$ converges pointwise to $\RR(\alpha)\phi$.
\end{lemma}
\begin{proof}
Fix $x$ in $\R$. Using \eqref{eq:resolvent_integrale} and then \eqref{eq:mehler}, for all $n$ in $\N$, we have
\begin{align*}
    \RR(\alpha) \phi_n(x) & = \int_0^{+\infty} e^{-\alpha t} \Pt_t \phi_n(x) \, \dint t =  \int_0^{+\infty} e^{-\alpha t} \, \E \bigl [\phi_n(xe^{-t}+\sqrt{1-e^{-2t}}N) \bigr ] \, \dint t
\end{align*}
where $N \sim \mathcal{N}(0,1)$. By assumption, one has $\forall n \in \N, \; |\phi_n(xe^{-t}+\sqrt{1-e^{-2t}}N)| \leq C(1+|xe^{-t}+\sqrt{1-e^{-2t}}N|^d)$. Since $N$ admits a moment of order $d$, Lebesgue's dominated convergence theorem yields
\begin{align*}
    \E \bigl [\phi_n(xe^{-t}+\sqrt{1-e^{-2t}}N) \bigr ] \underset{n\rightarrow +\infty}{\longrightarrow} 
    \E \bigl [\phi(xe^{-t}+\sqrt{1-e^{-2t}}N) \bigr ] = \Pt_t \phi(x).
\end{align*}
Then, by \eqref{eq:controle_semigroup_hermite}, one has the domination 
\begin{align*}
    |e^{-\alpha t} \Pt_t \phi_n(x)| & \leq e^{-(\alpha+r)t}CM_{r,d}(1+|x|^{d+r}),
\end{align*}
where $M_{r,d}$ is positive. Since $\alpha+r >0$, the dominated convergence theorem leads to the result, by taking the limit inside the integral.
\end{proof}

\subsection{Wiener chaoses}\label{s:wiener-chaos}
We work on the following countable product of probability spaces which we call a \emph{Wiener space}:
\begin{equation}\label{def:wiener-space}
  (\Omega,\mathscr{F},\prob) \coloneq \left(\mathbb{R},\mathscr{B}(\mathbb{R}),\gamma\right)^{\mathbb{N}}.
\end{equation}
We define the \emph{coordinate functions}
\begin{equation*}
N_i \coloneq
\begin{cases}
  \Omega & \longrightarrow  \mathbb{R},
  \\ (x_{0}, x_{1},\cdots) & \longmapsto  x_i.
\end{cases}
\end{equation*}
By construction, the $N_i$'s are independent random variables on $(\Omega,\mathscr{F},\prob)$, with common law the standard Gaussian distribution.

  For $m \in \mathbb{N}$, the \emph{$m$-th Wiener chaos} $\mathcal{W}_{m}$ is defined as the $L^{2}(\prob)$-closure of the linear span of functions of the form $\prod_{i=0}^\infty H_{k_i}(x_i)$ where the $k_{i}$'s satisfy $\sum_{i=0}^\infty k_i=m$.
The above product is finite since $H_0(x)=1$ and only finitely many integers $(k_i)_{i\ge 0}$ are non zero.
For $m = 0$, we find that $\mathcal{W}_{0}$ is the linear space of constant functions.
Importantly, Wiener chaoses provide the following  decomposition.
\begin{theorem}[It\^o-Wiener decomposition]
    Let $m\neq n$ be two nonnegative integers. The linear subspaces $\WW_m$ and $\WW_n$ are orthogonal. Furthermore, we have the direct sum
    \begin{align*}
L^2(\prob)&=\bigoplus_{m=0}^\infty \mathcal{W}_m.
    \end{align*}
\end{theorem}
We sometimes work in a finite sum of Wiener chaoses.
Accordingly, let us define
\begin{equation*}
  \mathcal{W}_{\leq m} \coloneq \bigoplus_{k=0}^m \mathcal{W}_k.
\end{equation*}

Note that this exposition of the Wiener chaoses is only one among several. Another important way of introducing them is through the so-called Wiener-It\^o multiple integrals. We refer to \cite{Nua2006, NP2012} for that approach.

  We often use that on $\mathcal{W}_{\leq m}$ all the $L^{p}(\prob)$-norms are equivalent.
Namely, we have the following result.
\begin{proposition}\label{prop:hypercontractivity}
    Let $m \in \mathbb{N}$, and $1 < p < q < \infty$. Then, for all $F \in \WW_{\leq m}$ such that $\E[F]=0$, one has 
\begin{align}\label{eq:hypercontractivity}
  \norm{F}_{p} & \leq \norm{F}_{q} \leq \left ( \frac{q-1}{p-1} \right )^{m/2} \norm{F}_{p}
\end{align}
\end{proposition}

This fact is well-known as a consequence of \emph{hypercontractivity} estimates, for instance \cite[Corollary\ 2.8.14]{NP2012}.

\subsection{Operators from Malliavin calculus}

We now introduce the operators from Malliavin calculus that will be used throughout this article.  
For a broader introduction and detailed proofs of the properties below, we refer to the textbooks \cite{Nua2006,NP2012,BH1991}.

\begin{definition}
The \emph{Ornstein--Uhlenbeck operator} $\L$ is defined by  
\begin{align}\label{eq:def_L}
    \L F \coloneq -\sum_{p=0}^{+\infty} p \, J_p(F),
\end{align}
on the domain
\[
\operatorname{Dom} \L \coloneq \Bigl\{ F \in L^2(\prob) \;\big|\; \sum_{p=0}^{+\infty} \,p^2 \, \|J_p(F)\|_2^2 < +\infty \Bigr\},
\]
where $J_p(F)$ denotes the orthogonal projection of $F$ onto the $p$-th Wiener chaos $\WW_p$.
\end{definition} 

\begin{definition}
The \emph{carr\'e-du-champ operator} (or \emph{square field operator}) $\Ga$ associated with $\L$ is defined, for $\phi,\psi \in \mathscr{C}_{pol}^2(\R^m)$, $m \in \N^*$, by  
\begin{align}\label{eq:def_Gamma}
    \Gamma(\phi,\psi) \coloneq \sum_{k=1}^{m} \partial_{x_k}\phi \;\partial_{x_k}\psi.
\end{align}
It extends by closability to the domain
\[
\operatorname{Dom} \Ga \coloneq \Bigl\{ H \in L^2(\prob) \;\big|\; \sum_{p=0}^{+\infty} p \,\|J_p(H)\|_2^2 < +\infty \Bigr\}.
\]
In particular, for any $m \in \N$, one has $\WW_{\leq m} \subset \operatorname{Dom} \L \subseteq \operatorname{Dom} \Ga$.
\end{definition}

The following result records some classical properties.

\begin{proposition}\label{prop:LGonsmooth}
\begin{enumerate}[(i)]
    \item For all $p \in \N$, \quad $\WW_p = \operatorname{Ker}(\L+p\Id)$.
    \item For all $m \in \N^*$ and all $\phi \in \mathscr{C}_{pol}^2(\R^m)$,  
    \[
    \L \phi = \Delta \phi - \sum_{k=1}^m x_k \,\partial_{x_k}\phi.
    \]
\end{enumerate}
\end{proposition}

The next proposition summarizes some fundamental identities involving $\L$ and $\Ga$, which will be heavily used in the sequel.

\begin{proposition}\label{prop:LGamma}
Let $m \in \N$. For all $F,G \in \WW_{\leq m}$, the following hold:
\begin{enumerate}[(a)]
    \item \emph{Integration by parts:}  
    \begin{align}\label{eq:LIPP}
        \E[\Ga(F,G)] = -\E[G \L F] = -\E[F \L G].
    \end{align}
    \item \emph{Diffusion property:} for any $\varphi \in \Cpol{2}$,  
    \begin{align}\label{eq:diffusion}
        \L \varphi(F) = \varphi''(F)\,\Ga(F,F) + \varphi'(F)\,\L F.
    \end{align}
    Equivalently, $\Ga$ acts as a derivation in the sense that  
    \[
    \Gamma(\varphi(F),F) = \varphi'(F)\,\Gamma(F,F).
    \]
    \item \emph{Spectral stability:}  
    \begin{align}\label{eq:spectral_stability}
        FG \in \WW_{\leq 2m}.
    \end{align}
    In particular, if $m \geq 1$, then
    \begin{align}\label{eq:gamma_stability}
        \Ga(F,F) \in \WW_{2m-2}.
    \end{align}
\end{enumerate}
\end{proposition}

For brevity, we will often write $\Ga(F) \coloneq \Ga(F,F)$.  
Proofs of Propositions~\ref{prop:LGonsmooth} and~\ref{prop:LGamma} can be found in \cite{BH1991}.
\subsection{Stein's method on Wiener chaos}\label{ss:steinsmethod}

The classical Stein's method starts from the following fundamental observation.

\begin{lemma}[Stein's lemma, 1972]
Let $\FF$ be the set of absolutely continuous functions $f \in L^2(\gamma)$ such that $f' \in L^1(\gamma)$, and let $X$ be a random variable. Then 
\begin{align}\label{eq:stein_charac}
    X \sim \mathcal{N}(0,1) 
    \quad \Longleftrightarrow \quad 
    \forall f \in \FF, \quad \E[f'(X)-Xf(X)]=0.
\end{align}
\end{lemma}

This characterization allows one to quantify the proximity of a random variable $X$ to the standard normal distribution by controlling the quantity $\E[f'(X)-Xf(X)]$.  
More precisely, let $\HH$ be a suitable class of measurable functions $h:\R \to \R$ (see \cite[Definition~C.1.2]{NP2012}) such that 
\begin{align}\label{eq:def_dHH}
    \dHH(F,G) \coloneq \sup_{h \in \HH} \, \bigl|\E[h(F)-h(G)]\bigr|
\end{align}
defines a distance on the space of Borel probability measures on $\R$.  
Important examples include:
\begin{enumerate}[(a)]
    \item $\HH = \{\I_A : A \in \mathscr{B}(\R)\}$, yielding the \emph{total variation} distance $\dHH = \dTV$;
    \item $\HH = \{\I_{(-\infty,z]} : z \in \R\}$, yielding the \emph{Kolmogorov} distance $\dHH = \dKOL$;
    \item $\HH = \operatorname{Lip}_1(\R)$, the set of $1$-Lipschitz functions, yielding the \emph{Wasserstein-1} distance $\dHH = \dWass$ on measures with finite first moment.
\end{enumerate}
The distance~\eqref{eq:def_dHH} can also be extended to bounded signed measures (see \cite{Dud2002}).  

\medskip

The core of Stein's method proceeds as follows.  
Fix $h \in \HH$. The \emph{Stein equation} associated with $h$ is the differential equation
\begin{align}\label{eq:stein_equation}
    \phi'(x) - x \phi(x) = h(x) - \E[h(N)], \quad x \in \R,
\end{align}
where $N \sim \mathcal{N}(0,1)$.  
A particular solution $\phi_h$ is given by
\begin{align} \label{eq:stein_solll}
    \phi_h(x) \coloneq \sqrt{2\pi}\, e^{x^2/2} \int_{-\infty}^x \bigl(h(y)-\E[h(N)]\bigr)\,\gamma(\dint y), \quad x \in \R.
\end{align}
For any random variable $F$, plugging $x=F$ into \eqref{eq:stein_equation} and taking expectations yields
\begin{align}\label{eq:auxh}
    \E[h(F)-h(N)] = \E[\phi'_h(F)-F\phi_h(F)].
\end{align}
Therefore,
\begin{align}\label{eq:boundingbystein}
    \dHH(F,N) 
    = \sup_{h \in \HH} \bigl|\E[h(F)-h(N)]\bigr|
    = \sup_{h \in \HH} \bigl|\E[\phi'_h(F)-F\phi_h(F)]\bigr|.
\end{align}

The essence of Stein's method is  to control the right-hand side of~\eqref{eq:boundingbystein}.  
To this end, bounds on the solution $\phi_h$ from \eqref{eq:stein_solll} to the Stein equation are crucial.  
In our setting, we are primarily interested in the total variation distance, i.e., $\dHH$ with $\HH$ given by the set of indicator functions of measurable subsets of $\R$.  
Since such functions are not regular, this creates technical difficulties.  
However, as shown in \cite{NP2015} and \cite[Proposition~C.3.5]{NP2012}, one may use a density argument to obtain
\begin{align}\label{eq:dTV_continuous}
    \dTV(F,G) = \tfrac{1}{2} \sup_{\substack{h \in \mathscr{C}_b^0(\R)\\ \|h\|_\infty \leq 1}} \bigl|\E[h(F)-h(N)]\bigr|
\end{align}
for all random variables $F,G$.  
This allows one to improve the regularity of the associated solution $\phi_h$.  

The next proposition summarizes some useful properties of $\phi_h$ in this context (see \cite{NP2012} for proofs).

\begin{proposition}\label{prop:proprietes_solstein}
Let $h:\R \to \R$ be continuous and bounded, let $p > 0$, and let $F \in \WW_p$. Then:
\begin{enumerate}[(i)]
    \item $\phi'_h$ is continuous;
    \item $\|\phi_h\|_\infty \leq \sqrt{\tfrac{\pi}{2}}\,\|h-\E[h(N)]\|_\infty$,  
    \quad $\|\phi'_h\|_\infty \leq 2\,\|h-\E[h(N)]\|_\infty$;
    \item $\E[h(F)-h(N)] = \tfrac{1}{p}\,\E[\phi'_h(F)\,(p-\Ga(F))]$.
\end{enumerate}
\end{proposition}

\section{Overview of the proof}\label{s:overview}

In this section, we fix $p>1$ and $F \in \WW_p$ such that $\E[F^2]=1$. Note that, since $p>0$, one has $\E[F]=0$.

The starting point for proving Theorem~\ref{maintheorem} is to split the right-hand side of~\eqref{eq:dTV_continuous} into two parts: the first yields the Edgeworth expansion of $F$ around the standard Gaussian distribution, and the second will be bounded by a power of $\Var\Ga(F)$.

Fix a continuous, bounded function $h:\R\rightarrow\R$. By Proposition~\ref{prop:proprietes_solstein},
\begin{align}\label{eq:expansion1}
    \E\big[h(F)-h(N)\big]
    & =
    \frac{1}{p}\,\E \big[\phi'_h(F)\,\big(p-\Ga(F)\big)\big].
\end{align}
The standard approach in the literature is then to apply integration by parts once more to the right-hand side of~\eqref{eq:expansion1}, thereby introducing iterated gradients of $F$ (see, e.g.,~\cite{NP2015}). However, it is not clear how to control these iterated gradients by a power of the variance of the carr\'e-du-champ operator of $F$.

Therefore, we take a different route. By Proposition~\ref{prop:proprietes_solstein}, the function $\phi_h'$ is bounded and hence belongs to $L^2(\gamma)$. We can thus expand $\phi_h'$ in the (complete) Hermite basis of $L^2(\gamma)$. Namely,
\begin{align}\label{eq:expansion2}
   \forall x \in \R, \quad \phi'_h(x) &= \projgeq{M}(\phi'_h)(x)+\underset{k=0}{\overset{M-1}{\sum}} \, c_k(\phi_h')H_k(x).
\end{align}
Plugging \eqref{eq:expansion2} back into \eqref{eq:expansion1}, one gets 
\begin{eqnarray}\label{eq:edgeworth_stein1}
&& \E \bigl [ h(N)-h(F) \bigr ]  = \frac{1}{p} \E \bigl [ \phi'_h(F) (\Ga(F)-p) \bigr ] \\
&& = \underset{k=0}{\overset{M-1}{\sum}} \, \frac{c_k(\phi_h')}{p}\E \bigl [ H_k(F)(\Ga(F)-p) \bigr]  + \frac{1}{p}\E \bigl [ \projgeq{M}(\phi_h')(F)(\Ga(F)-p) \bigr]\nonumber.
\end{eqnarray}
Putting some of the terms of \eqref{eq:edgeworth_stein1} on the left side of the equation leads to the next expression.
\begin{lemma}\label{prop:writing_edgeworth}
    Let $M$ be a nonnegative integer. We define the polynomial function $g_{F,M}$ by
    \begin{align*}
       \forall x \in \R, \quad g_{F,M}(x) \coloneq  1+\underset{k=3}{\overset{M+1}{\sum}} \,\frac{\E [ H_{k}(F) ]}{k!}H_{k}(x).
    \end{align*}
    Then, for all $h:\R \rightarrow \R$ measurable and bounded, one has
    \begin{align}\label{auxauxaux}
        \E \bigl [ h(F)-h(N)g_{F,M}(N) \bigr ]
&=  - \frac{1}{p} \,\E \bigl [ \projgeq{M}(\phi_h')(F)(\Ga(F)-p) \bigr]
    \end{align}
\end{lemma}

\begin{proof}
Fix $h:\R \rightarrow \R$ a bounded measurable function and consider equality \eqref{eq:edgeworth_stein1}. For all $k$ in $\llbracket 0,M-1 \rrbracket$, one has
\begin{align}\label{auxaux}
   \frac{1}{p} c_k(\phi_h') \E \bigl [ H_{k}(F)(\Ga(F)-p) \bigr] & = -\frac{\E [ H_{k+2}(F) ]}{(k+2)!}\E [ H_{k+2}(N)h(N)  ]
\end{align}
since, on one hand, by a classical computation (see for instance \cite{BBNP2012})
\begin{align*}
    \E \bigl [H_k(N) \phi'_h(N)\bigr ] & =  -\frac{1}{k+2} \E \bigl [H_{k+2}(N) h(N)\bigr ],
\end{align*}
and, using Lemma \ref{prop:gaussian_ipp} $(i)$, Proposition \ref{prop:LGamma} $(a)$,$(b)$ and equality \eqref{hermite:eigenfunction}, one gets
\begin{multline*}
      \E \bigl [ H_k(F) (\Ga(F)-p) \bigr ]
     = \E \bigl [ H_k(F)\Ga(F,F) \bigr ]-p \E \bigl [ H_k(F)  \bigr ] \\
     =\frac{1}{(k+1)} \E \bigl [ \Ga(H_{k+1}(F),F) \bigr ]- p\E \bigl [ H_k(F)  \bigr ] =\frac{p}{(k+1)} \E \bigl [ H_{k+1}(F)F \bigr ]- p\E \bigl [ H_k(F)  \bigr ] \\
     =\frac{p}{(k+1)} \E \bigl [ H_{k+2}(F) \bigr ]. \hspace{8.6cm}
\end{multline*}
Plugging \eqref{auxaux} in \eqref{eq:edgeworth_stein1}, and after a change of indice in the sum, we get
\begin{eqnarray*}
&&\E \bigl [ h(N)-h(F) \bigr ]
\\
&&= -\underset{k=2}{\overset{M+1}{\sum}} \,\frac{\E  [ H_{k}(F) ]}{k!}\E \bigl [ H_{k}(N)h(N)  \bigr ]  + \frac{1}{p} \, \E \bigl [ \projgeq{M}(\phi_h')(F)(\Ga(F)-p) \bigr]\\
&& = \E \bigl [\bigl (1-g_{F,M}(N) \bigr )h(N) \bigr ] + \frac{1}{p} \, \E \bigl [ \projgeq{M}(\phi_h')(F)(\Ga(F)-p) \bigr],
\end{eqnarray*}
since $\E[H_2(F)] = \E[F^2]-1 = 0$. Hence the result.
\end{proof}
Taking the supremum over all those continuous function $h$ bounded by $1$ in the left-hand side of \eqref{auxauxaux} leads not to a the distance between $F$ and a random variable, but to the distance between the law of $F$ and a signed measure that we define below.
\begin{definition}
    We define the Gaussian correcting measure of order $M$ associated to $F$, denoted by $\gaus_{F,M}$, as the signed measure with density $x \mapsto g_{F,M}(x)e^{-x^2/2}/\sqrt{2\pi}$.
\end{definition}
A consequence of Lemma \ref{prop:writing_edgeworth} is the following.
\begin{proposition}\label{prop:overview of proof}
    Let $m$ be a positive integer. Then, for every $F \in \WW_p$ such that $\E[F^2]=1$, we have
    \begin{align*}
    \dTV(\prob_F, \gaus_{F,M}) = \frac{1}{2p} \, \underset{\underset{\norm{h}_\infty \leq 1}{h\in\mathscr{C}^0_{b}(\R)}}{\sup} \, \bigl |\E \bigl [ (\projgeq{M}(\phi_h')(F)(\Ga(F)-p) \bigr] \bigr |
    \end{align*}
    where $\gaus_{F,M}$ is the signed measure defined in \ref{maintheorem}.
\end{proposition}
\begin{proof}
    Let $F\in\WW_p$ such that $\E[F^2]=1$. By definition, one has
    \begin{align*}
        \dTV(\prob_F, \gaus_{F,M}) & = \frac{1}{2} \underset{h }{\sup} \left | \int_\R h(x) \, \prob_F(\dint x) - \int_\R h(x)\, \gaus_{F,M}(\dint x)\right | \\
        & = \frac{1}{2} \underset{h }{\sup} \left | \int_\R h(x) \, \prob_F(\dint x) - \int_\R h(x)g_{F,M}(x)\frac{e^{-x^2/2}}{\sqrt{2\pi}} \,\dint x\right | \\
        & = \frac{1}{2} \underset{h}{\sup} \,\bigl | \E \bigl [ h(F)-h(N)g_{F,M}(N) \bigl ] \bigr |
    \end{align*}
    where the suprema are taken over all $h \in\mathscr{C}^0_{b}(\R)$ such that $\norm{h}_\infty \leq 1$.
    Then, substituting $\E \bigl [ h(F)-h(N)g_{F,M}(N) \bigl ]$ using \eqref{auxauxaux} from Lemma \ref{prop:writing_edgeworth} leads to the result.
\end{proof}
    
In light of Proposition \ref{prop:overview of proof}, in order to prove Theorem \eqref{maintheorem}, one has to show that, given $F \in \WW_p$ with $\E[F^2]=1$, the term
\begin{align}\label{goalofproof}
    \bigl | \E \bigl [ \projgeq{M}(\phi_h')(F)(\Ga(F)-p) \bigr] \bigr |
\end{align}
can be bounded, independently of $F$ and $h$, by a power of $\Var\Ga(F)$, for $M$ big enough. This is the most difficult part of the proof. We will see that choosing $M=4m-2$ enables us to bound \eqref{goalofproof} by the power $\bigl ( \Var\Ga(F) \bigr )^{\frac{m+1}{2}}$. Consequently, we introduce the next definition.

\begin{definition}\label{def:edgeworth_measure}
    Let $m$ be a positive integer. We called \emph{Edgeworth expansion measure of $F$ of order $m$} the signed measure $\boldsymbol{\gamma}_{F,m} \coloneq \gaus_{F,4m-2}$.
\end{definition}

To make the proof clearer and easier to understand, we divide it into two parts: case $m=1$, where we detail and explain the main technical ideas, and general case, where we generalize the ideas developed for case $m=1$ to reach the higher powers.
\section{Proof of Theorem \ref{maintheorem} in the case $m=1$}\label{s:m=1}
Fix $p \geq 1$ an integer and $F \in \WW_p$. As explained in Section \ref{s:overview} above, the proof of Theorem \ref{maintheorem} is centered around bounding \eqref{goalofproof} by a power of $\Var\Ga(F)$. To that effect, we will consider the related quantity
\begin{align}\label{eq:mainterm}
  \E \bigl [\varphi(F)(\Ga(F)-p) \bigr],
\end{align}
where $\varphi \in \Cpol{0}$, with for main idea to set $\varphi = \projgeq{M}(\phi_h')$ at the end.
\subsection{An integration by parts formula}
The first step is the following equality, which illustrates Stein's method in the context of eigenfunctions of a diffusive Markov operator (see \cite{Led2012} and Section \ref{s:markovoperator} below), and that we will use to deduce an integration by parts formula for quantities of the form \eqref{eq:mainterm}.
\begin{lemma}\label{prop:formula1}
    Let $\phi$ be an element of $\Cpol{2}$. Then
\begin{align}\label{eq:formula1}
    \L(\phi(F)) & = \phi''(F)(\Gamma(F)-p)+p(\LL \phi)(F)
\end{align}
\end{lemma}
\begin{proof}
    One uses the diffusion property \eqref{eq:diffusion} of $\L$ to write
    \begin{align*}
        \L(\phi(F)) & = \phi''(F)\Gamma(F)+\phi'(F)\L F.
    \end{align*}
    Then, using $\L F = -p F$, one gets
    \begin{align*}
        \L(\phi(F)) & = \phi''(F)\Gamma(F)-p\phi'(F) F = \phi''(F)(\Gamma(F)-p)+p\phi''(F) -p\phi'(F) F \\
        &= \phi''(F)(\Gamma(F)-p)+p(\LL \phi)(F),
    \end{align*}
    since, by \eqref{eq:LOU_1D}, $(\LL\phi)(F) = \phi''(F) -F\phi'(F)$.
\end{proof}
If $\phi \in \Cpol{2}$ then isolating $\LL\phi$ in \eqref{eq:formula1}, one gets
\begin{align}\label{eq:trick1}
  (\LL \phi)(F) & = -\frac{1}{p}\, \phi''(F)(\Gamma(F)-p)  +\frac{1}{p}\,\L\phi(F).
\end{align}
Plugging $\varphi = \LL \phi$ in \eqref{eq:mainterm}, it follows
\begin{align}\label{form:1}
   \E[(\LL\phi)(F)(\Ga(F)-p)] & = -\frac{1}{p}\,\E[\phi''(F)(\Gamma(F)-p)^2]  +\frac{1}{p}\,\E[\L\phi(F) (\Gamma(F)-p)]\nonumber \\
    & = -\frac{1}{p}\,\E[\phi''(F)(\Gamma(F)-p)^2]  +\frac{1}{p}\, \E[\phi(F) \L(\Gamma(F)-p)]
\end{align}
by symmetry of $\L$. Iterating this procedure, one gets the following lemma.

\begin{lemma}\label{trick1_rec}
     Let $\alpha$ be a positive integer and let $\phi$ be an element of $\Cpol{2\alpha}$. Then one has
\begin{align}\label{eq:trick2}
   \E[(\LL^\alpha\phi)(F)(\Ga(F)-p)] & = -\underset{j=0}{\overset{\alpha-1}{\sum}}\, \frac{1}{p^{j+1}}\E\bigl [(\LL^{\alpha-j-1}\phi)''(F)(\Gamma(F)-p)\L^j(\Gamma(F)-p) \bigr ]\nonumber \\
   & \qquad +\frac{1}{p^\alpha}\E \bigl [\phi(F) \L^\alpha(\Gamma(F)-p) \bigr ] 
\end{align}
\end{lemma}
\begin{proof}
    If $\alpha = 1$, the result is given by \eqref{form:1}. Suppose $\alpha>1$ and equality \eqref{eq:trick2} verified, for all function $\phi$ in $\Cpol{2\alpha}$. Let $\phi$ be in $\Cpol{2\alpha+2}$. Then $\LL\phi$ belongs to $\Cpol{2\alpha}$. Thus, applying \eqref{eq:trick2} to $\LL\phi$, one gets
\begin{align}\label{trick1_rec_aux}
   \E[(\LL^{\alpha+1}\phi)(F)(\Ga(F)-p)] & =  \underset{j=0}{\overset{\alpha-1}{\sum}}\, \frac{-1}{p^{j+1}}\E[(\LL^{\alpha+1-j-1}\phi)''(F)(\Gamma(F)-p)\L^j(\Gamma(F)-p)]\nonumber \\
   & \qquad +\frac{1}{p^\alpha}\E[(\LL\phi)(F) \L^\alpha(\Gamma(F)-p)].
\end{align}
Then, using \eqref{eq:trick1}, we get
\begin{eqnarray*}
  && \frac{1}{p^\alpha}\,\E[(\LL\phi)(F) \L^\alpha(\Gamma(F)-p)] \\
  && = \frac{1}{p^\alpha}\,\E\left [\left (-\frac{1}{p}\, \phi''(F)(\Gamma(F)-p)  +\frac{1}{p}\,\L  \phi(F) \right )\L^\alpha(\Gamma(F)-p)\right ] \\ 
  && =- \frac{1}{p^{\alpha+1}}\,\E \big[ \phi''(F)(\Gamma(F)-p)\,\L^\alpha(\Gamma(F)-p)\big] +\frac{1}{p^{\alpha+1}}\, \E [ \phi(F) \L^{\alpha+1}(\Gamma(F)-p) ],
\end{eqnarray*}
where we used the symmetry of $\L$. Substituting in \eqref{trick1_rec_aux}, we obtain
\begin{multline*}
   \E[(\LL^{\alpha+1}\phi)(F)(\Ga(F)-p)] =  \underset{j=0}{\overset{\alpha-1}{\sum}} \frac{-1}{p^{j+1}}\E[(\LL^{\alpha+1-j-1}\phi)''(F)(\Gamma(F)-p)\L^j(\Gamma(F)-p)]\nonumber \\
     - \frac{1}{p^{\alpha+1}}\,\E \big[ \phi''(F)(\Gamma(F)-p)\,\L^\alpha(\Gamma(F)-p)\big] +\frac{1}{p^{\alpha+1}}\, \E [ \phi(F) \L^{\alpha+1}(\Gamma(F)-p) ] \\
    = -\underset{j=0}{\overset{\alpha}{\sum}}\, \frac{1}{p^{j+1}}\,\E[(\LL^{\alpha+1-j-1}\phi)''(F)(\Gamma(F)-p)\L^j(\Gamma(F)-p)]\nonumber\\
   +\frac{1}{p^{\alpha+1}}\, \E [ \phi(F) \L^{\alpha+1}\Gamma(F) ], \hspace{5cm}
\end{multline*}
thus yielding the desired equality by induction.
\end{proof}
\begin{lemma}\label{productformula1}
    Let $P$ be the polynomial $X(X+1) \cdots (X+2p-2)$. Then, for all $G$ in $\WW_p$, $P(\L)(\Ga(G)-p) = 0$.
\end{lemma}
\begin{proof}
    Let $G$ be an  element of $\WW_p$. By the product formula, one has $\Ga(G) \in \WW_{\leq 2p-2}$ since $G\in \WW_p$. Hence, the polynomial $P = X(X+1) \cdots (X+2p-2)$ of degree $2p-1$ is a canceling polynomial of $\L$ on $\WW_{\leq 2p-2}$, therefore $P(\L)(\Ga(G)-p) = 0$.
\end{proof}
Let us denote by $\underset{\alpha = 0}{\overset{2p-1}{\sum}} a_\alpha X^\alpha$ the expansion of $P$ in the canonical basis of $\R[X]$. Note that $a_0 = P(0) = 0$.
\begin{lemma}\label{dev:pol}
    Let $\phi$ be an element of $\Cpol{4p-2}$. Then one has
\begin{align*}
   \E[P(p\LL)(\phi)(F)(\Ga(F)-p)] & = -\underset{j=0}{\overset{2p-2}{\sum}}\, \E \bigl [\bigl (P_j(p\LL)(\phi) \bigr )''(F)(\Gamma(F)-p)\L^j(\Gamma(F)-p) \bigr ],
\end{align*}
where $P_j(X) := \underset{\alpha=j+1}{\overset{2p-1}{\sum}} a_\alpha X^{\alpha-j-1}$. 
\end{lemma}
\begin{proof}
For each $\alpha$ integer such that $1\leq \alpha \leq 2p-1$, equality \eqref{eq:trick2} yields
\begin{align}\label{aux}
    \E[(\LL^\alpha\phi)(F)(\Ga(F)-p)] & = -\underset{j=0}{\overset{\alpha-1}{\sum}}\, \frac{1}{p^{j+1}}\E[(\LL^{\alpha-j-1}\phi)''(F)(\Gamma(F)-p)\L^j(\Gamma(F)-p)]\nonumber \\
   & \qquad +\frac{1}{p^\alpha}\E[\phi(F) \L^\alpha(\Gamma(F)-p)].
\end{align}
Multiplying \eqref{aux} by $a_\alpha p^\alpha$ on both sides and summing over all $\alpha \in \llbracket 1,2p-1\rrbracket$ leads then to 
\begin{eqnarray}\label{aux2}
    && \underset{\alpha=1}{\overset{2p-1}{\sum}}a_\alpha p^\alpha\E[(\LL^\alpha\phi)(F)(\Ga(F)-p)] \nonumber \\ 
    &&= -  \underset{\alpha=1}{\overset{2p-1}{\sum}}\underset{j=0}{\overset{\alpha-1}{\sum}}\, a_\alpha p^{\alpha-j-1}\E[(\LL^{\alpha-j-1}\phi)''(F)(\Gamma(F)-p)\L^j(\Gamma(F)-p)] \nonumber \\
    && \qquad +  \underset{\alpha=1}{\overset{2p-1}{\sum}}a_\alpha\E[\phi(F) \L^\alpha(\Gamma(F)-p)].
\end{eqnarray}
By linearity of the expectation, one has
\begin{align*}
     \underset{\alpha=1}{\overset{2p-1}{\sum}}a_\alpha p^\alpha\E[(\LL^\alpha\phi)(F)(\Ga(F)-p)] & = \E[P(p\LL)(\phi)(F)(\Ga(F)-p)],
\end{align*}
and 
\begin{align*}
    \underset{\alpha=1}{\overset{2p-1}{\sum}}a_\alpha\E[\phi(F) \L^\alpha(\Gamma(F)-p)] & = \E[\phi(F) P(\L)(\Gamma(F)-p)] = 0,
\end{align*}
where we used Lemma \ref{productformula1}. Hence
\begin{align*}
    \E[(P(p\LL)\phi)&(F)(\Ga(F)-p)] \\ & = -  \underset{\alpha=1}{\overset{2p-1}{\sum}}\underset{j=0}{\overset{\alpha-1}{\sum}}\, a_\alpha p^{\alpha-j-1}\E \bigl [(\LL^{\alpha-j-1}\phi)''(F)(\Gamma(F)-p)\L^j(\Gamma(F)-p) \bigr ] \\
    &= - \underset{j=0}{\overset{2p-2}{\sum}}\underset{\alpha=j+1}{\overset{2p-1}{\sum}}\, a_\alpha p^{\alpha-j-1}\E\bigl [(\LL^{\alpha-j-1}\phi)''(F)(\Gamma(F)-p)\L^j(\Gamma(F)-p)\bigr ] \\
    &=- \underset{j=0}{\overset{2p-2}{\sum}}\E\bigl [\bigl ( P_j(p\LL)\phi \bigr )''(F)(\Gamma(F)-p)\L^j(\Gamma(F)-p) \bigr ],
\end{align*}
where, in the last equality, we used the linearity of the expectation and the differentiation.
\end{proof}
\begin{proposition}\label{rel1}
    Let $\varphi$ be an element of $\Cpol{4p}$ with Hermite rank $r\geq 2$. Then one has
\begin{align*}
   \E[\varphi(F)(\Ga(F)-p)] & =\underset{j=0}{\overset{2p-2}{\sum}}\, \E \bigl [\Tj_j(\varphi)(F)(\Gamma(F)-p)\L^j(\Gamma(F)-p) \bigr ],
\end{align*}
where $\Tj_j \varphi :=- P_j\bigl (p(\LL-2)\bigr )P\bigl (p(\LL-2)\bigr )^{-1}\varphi''$.
\end{proposition}
\begin{proof}
This proof is divided into two steps.

\noindent \underline{\it Step 1}: Application of Lemma \ref{dev:pol}. One has 
$$P(pX) = p^{2p-1}X(X+1/p) \cdots (X+(2p-2)/p),$$
that is $P(pX)$ has roots the $-i/p$, where $i \in \llbracket 0, 2p-2 \rrbracket$. Let $i \in \llbracket 0, 2p-2 \rrbracket$. Since $i/p<2$, by Proposition \ref{prop:resolvent}, the operator $\LL+i/p$ is invertible on function $u \in \Cpol{0}$ with Hermite rank greater than $1$. Since $\varphi$ has Hermite rank  $r\geq 2$, by Lemma \ref{prop:cons_rank}, we deduce that $\RR(-i/p) \cdots \RR(-(2p-2)/p) \varphi$ also has Hermite rank $r$. It follows that $P(p\LL) = p^{2p-1}\LL(\LL+1/p) \cdots (\LL+(2p-2)/p)$ is invertible on $\varphi$, and 
    \begin{align}\label{eq:P(pLL)-1}
        P(p\LL)^{-1}\varphi = p^{-2p+1}\RR(0) \RR(-1/p) \cdots \RR(-(2p-2)/p)\varphi.
    \end{align}
    Given that $\varphi$ belongs to $\Cpol{4p}$, the function $\phi := P(p\LL)^{-1}\varphi$ lies also in $\Cpol{4p}$ by Proposition \ref{prop:resolvent_commutation}. Hence, applying Lemma \ref{dev:pol} to $\phi$ leads to
    \begin{align*}
         \E[\varphi(F)(\Ga(F)-p)] & = -\underset{j=0}{\overset{2p-2}{\sum}}\, \E \bigl [\bigl (P_j(p\LL)P(p\LL)^{-1}\varphi \bigr )''(F)(\Gamma(F)-p)\L^j(\Gamma(F) -p)\bigr ].
    \end{align*}
\noindent \underline{\it Step 2}: We show here that $(P_j(p\LL)P(p\LL)^{-1}\varphi \bigr )'' = P_j \bigl( p(\LL-2) \bigr )P\bigl( p(\LL-2) \bigr )^{-1}\varphi''$, using the commutation relations of $\LL$ and its resolvent with the derivatives. Using \eqref{eq:resolvent_commutation} and \eqref{eq:commutationLOU}, we proceed to rewrite  the quantity $(P_j(p\LL)P(p\LL)^{-1}\varphi \bigr )''$. First, we prove the following lemma.
    \begin{lemma}
        Let $\beta$ be a nonnegative integer. Then, for all $\psi \in \Cpol{2\beta+2}$, 
        \begin{align}\label{auxalpha}
            \bigl(\LL^\beta\psi \bigr )'' & = (\LL-2)^\beta \psi''.
        \end{align}
    \end{lemma}
    \begin{proof}
        Let $\psi$ be an element of $\Cpol{2\beta +2}$. Applying \eqref{eq:commutationLOU}, one has
        \begin{align*}
            ( \LL u  )'' & =  \bigl ( (\LL-1) u' \bigr  )' = ( \LL u'  )'-u'' = (\LL-1)u''-u'' = (\LL-2)u'',
        \end{align*}
        for all $u \in \Cpol{4}$. Then one gets
        \begin{align*}
            (\LL^\beta \psi)'' = \bigl ( \LL (\LL^{\beta-1}\psi) \bigr )'' = (\LL-2)(\LL^{\beta-1}\psi)''.
        \end{align*}
        The result follows by a straightforward iteration.
    \end{proof}
Applying \eqref{auxalpha} below, we get
\begin{eqnarray*}
&&(P_j(p\LL)P(p\LL)^{-1}\varphi \bigr )''   =\left ( \sum_{\alpha = j+1}^{2p-1} a_{\alpha}p^{\alpha-j-1} \LL^{\alpha-j-1} \bigl ( P(p\LL)^{-1}\varphi \bigr ) \right )''\\
&&=  \sum_{\alpha = j+1}^{2p-1} a_{\alpha}p^{\alpha-j-1} \Bigl  (\LL^{\alpha-j-1} \bigl ( P(p\LL)^{-1}\varphi \bigr ) \Bigr )'' \\
&& = \sum_{\alpha = j+1}^{2p-1} a_{\alpha}p^{\alpha-j-1} (\LL-2)^{\alpha-j-1} \bigl ( P(p\LL)^{-1}\varphi \bigr )'' = P_j \bigl( p(\LL-2) \bigr )\bigl ( P(p\LL)^{-1}\varphi \bigr )''.
\end{eqnarray*}
Then, an iteration of the commutation relation \eqref{eq:resolvent_commutation} two times on \eqref{eq:P(pLL)-1} yields
\begin{multline*}
     \bigl ( P(p\LL)^{-1} \varphi \bigr )''  = p^{-2p+1} \Bigl ( \RR(0+1) \RR(-1/p+1) \cdots \RR(-(2p-2)/p+1) \varphi' \Bigr )'\\
     = p^{-2p+1} \RR(2) \RR(-1/p+2) \cdots \RR(-(2p-2)/p+2) \varphi'' \\
    = \Bigl ( p^{2p-1}(\LL-2)(\LL-2+1/p) \cdots (\LL-2+(2p-2)/p) \Bigr )^{-1}\varphi'' = P \bigl ( p(\LL-2) \bigr)^{-1}\varphi'',
\end{multline*}
and so the desired equality
\begin{align*}
    (P_j(p\LL)P(p\LL)^{-1}\varphi \bigr )'' & = P_j \bigl( p(\LL-2) \bigr )\bigl ( P(p\LL)^{-1}\varphi \bigr )'' \\
    & = P_j \bigl( p(\LL-2) \bigr )P\bigl( p(\LL-2) \bigr )^{-1}\varphi''.
\end{align*}
\end{proof}
\subsection{Study of the operator $\Tj_j$}\label{ss:studyofT}
In order to investigate the properties of the operators $(\Tj_j)_j$, we proceed to rewrite their expression. 
\begin{lemma}\label{prop:rewrite_of_T}
    Let $j$ be an element of $\llbracket 0, 2p-2\rrbracket$. For all $\varphi \in \Cpol{4p}$, 
    \begin{align}\label{eq:Tj}
        \Tj_j \varphi & = -\Q_j(\LL)\S \varphi,
    \end{align}
    where $\S \varphi \coloneq (\LL-2)^{-1}\varphi''$ and 
    \[\Q_j(X) := p^{-j-1}(X-2)^{-j} - \underset{\alpha=1}{\overset{j}{\sum}} \, a_\alpha p^{\alpha-j-1} (X-2)^{\alpha-j}P \bigl ( p(X-2) \bigr )^{-1}.\]
\end{lemma}
\begin{proof}
    Observe that $X^{j+1}P_j(X) = P(X) - \sum_{\alpha=1}^{j} \, a_\alpha X^\alpha$, hence
\begin{multline}\label{az}
    P_j(p(X-2))P(p(X-2))^{-1} \\ = p^{-j-1}(X-2)^{-j-1} - \underset{\alpha=1}{\overset{j}{\sum}} \, a_\alpha p^{\alpha-j-1} (X-2)^{\alpha-j-1}P \bigl ( p(X-2) \bigr ) ^{-1}
\end{multline}
Define
\begin{align*}
   \Q_j(X) & := p^{-j-1}(X-2)^{-j} - \underset{\alpha=1}{\overset{j}{\sum}} \, a_\alpha p^{\alpha-j-1} (X-2)^{\alpha-j}P \bigl ( p(X-2) \bigr )^{-1},
\end{align*}
and, for all $\varphi \in \Cpol{4p}$, $\S \varphi := (\LL-2)^{-1} \varphi''$. Then, one has
\begin{align*}
       P_j(p(X-2))P(p(X-2))^{-1} & = \Q_j(X)(X-2)^{-1},
\end{align*}
and consequently, for all $\varphi \in \Cpol{4p}$, one has
\begin{align*}
    P_j \bigl ( p(\LL-2) \bigr ) P \bigl ( p(\LL-2) \bigr )^{-1}\varphi'' = \Q_j(\LL)(\LL-2)^{-1}\varphi'' = \Q_j(\LL)\S\varphi.
\end{align*}
\end{proof}

In order to write \eqref{rel1} for functions $\varphi$ in $\Cpol{0}$, we need to extend $\Tj_j$ on $\Cpol{0}$ and formulate a continuity property for $\Tj_j$. The first step toward that is the following.

\begin{lemma}\label{prop:Sextended}
    The operator $\S$ verifies, for all $\varphi \in \Cpol{4p}$, and $x \in \R$,
\begin{align*}
        \S(\varphi)(x) := \E \bigl [\varphi(N) \bigr ]-\varphi(x) - x \int_0^{+\infty} \frac{e^{-t}}{\sqrt{1-e^{-2t}}}\E \Bigl [ N\varphi \bigl (xe^{-t}+\sqrt{1-e^{-2t}}N \bigr )\Bigr ] \, \dint t.
\end{align*}
    Hence, the operator $\S$ can be extended to $\Cpol{0}$.
\end{lemma}
\begin{proof}
       Let $\varphi$ be an element of $\Cpol{2}$ and let $x \in \R$. Then, using \eqref{eq:kolmo}, one has
\begin{align*}
\S(\varphi)(x) & = - \int_0^{+\infty} e^{-2t} \Pt_t\varphi''(x) \, \dint t  = - \int_0^{+\infty}  \bigl (\Pt_t\varphi \bigr )''(x) \, \dint t \\
& =  - \int_0^{+\infty} \Bigl ( x\bigl ( \Pt_t\varphi \bigr )'(x)  - \partial_t \Pt_t\varphi(x) \Bigr )\, \dint t \\
& = \E \bigl [\varphi(N) \bigr ]-\varphi(x) - x \int_0^{+\infty} \frac{e^{-t}}{\sqrt{1-e^{-2t}}}\E \Bigl [ N\varphi \bigl (xe^{-t}+\sqrt{1-e^{-2t}}N \bigr )\Bigr ] \, \dint t
\end{align*}
using properties \eqref{eq:kolmo} and \eqref{eq:commutationPt}. Finally, one observes that the latter expression is still well-defined for $\varphi\in \Cpol{0}$.
\end{proof}
\begin{remark}\label{Sbound}
    If $\phi$ is a bounded continuous function then
        \begin{align*}
        \forall x \in \R, \quad |\S(\phi)(x)| & \leq \left (2+|x| \right ) \|\phi \|_{\infty}.
    \end{align*}
Indeed, one has 
\begin{align*}
    |\S(\phi)(x)| &\leq 2\norm{\phi}_\infty + |x| \,\E\bigl [|N| \bigr ] \int_0^{+\infty} \frac{e^{-t}}{\sqrt{1-e^{-2t}}} \, \dint t \norm{\phi}_\infty = \left (2+|x|\right )\norm{\phi}_\infty.
\end{align*}
\end{remark}
\begin{lemma}\label{Spol}
    Let $\phi$ be a continuous function verifying
    \begin{align*}
        \forall x \in \R, \quad |\phi(x)| \leq C(1+|x|^d)
    \end{align*}
    for some $C>0$ and some nonnegative integer $d$. Then, there exists $M$ such that
    \begin{align*}
        \forall x \in \R, \quad |\S(\phi)(x)| \leq CM(1+|x|^{d+1})
    \end{align*}
    where $M$ depends only on $d$. Furthermore, $\S(\phi)$ is continuous, that is $\S(\phi) \in \Cpol{0}$.
\end{lemma}
\begin{proof}
    For $x$ in $\R$ and $t>0$, Jensen's inequality yields
    \begin{eqnarray*}
        && \Bigl |\E \Bigl [ N\phi \bigl (xe^{-t}+\sqrt{1-e^{-2t}}N \bigr )\Bigr ] \Bigr | \leq C \, \E \Bigl [ |N|(1+ |xe^{-t}+\sqrt{1-e^{-2t}}N|^d\Bigr ] \\
        && \leq C \, \E \bigl [ |N| \bigr ] +2^{d-1}C\bigl (1-e^{-2t} \bigr )^{d/2} \E \bigl [|N|^d\bigr ] +2^{d-1}Ce^{-dt}|x|^d.
    \end{eqnarray*}
    One deduces
    \begin{align*}
        |\S(\phi)(x)| & \leq A_0+A_1|x|+A_{d}|x|^d+A_{d+1}|x|^{d+1}
    \end{align*}
    where 
    \begin{align*}
        A_0 & := 2C+C\, \E[|N|^d], \qquad A_1  := C\, \E [ |N| ]\int_0^{+\infty} \frac{e^{-t}}{\sqrt{1-e^{-2t}}} \, \dint t   \\
        & \qquad +C \, \E [ |N|^d ]\int_0^{+\infty} 2^{d-1}e^{-t}\bigl (1-e^{-2t} \bigr )^{(d-1)/2} \, \dint t, \\
        A_d & := C, \qquad
        A_{d+1}  := 2^{d-1}C\int_0^{+\infty} \frac{e^{-(d+1)t}}{\sqrt{1-e^{-2t}}} \, \dint t,
    \end{align*}
    thus concluding the proof. The continuity of $\S(\phi)$ follows from Lebesgue's dominated convergence theorem.
\end{proof}
\begin{lemma}\label{Scontinuite}
    Consider a sequence $(\varphi_n)$ of functions in $\Cpol{0}$, converging pointwise to a function $\varphi$, such that there exist a constant $C$ and a nonnegative integer $d$ such that
    \[\forall n \in \N, \forall x \in \R, \quad |\varphi_n(x)| \leq C(1+|x|^d)\]
Then $\bigl (\S(\varphi_n)\bigr )_n$ converges pointwise to $\S(\varphi)$.
\end{lemma}
\begin{proof}
    As seen previously in the proof of Lemma \ref{Spol}, Jensen's inequality gives us, for all $n \in \N$,
    \begin{align*}
        |N\varphi_n \bigl (xe^{-t}+\sqrt{1-e^{-2t}}N \bigr )|
        & \leq C \, |N| +2^{d-1}C\bigl (1-e^{-2t} \bigr )^{d/2} |N|^d\\
        & \qquad +2^{d-1}Ce^{-dt}|x|^d.
    \end{align*}
    The desired convergence follows by Lebesgue's dominated convergence theorem.
\end{proof}
Now we extend $\Tj_j$.
\begin{lemma}
    Let $j$ be an element of $\llbracket 0,2p-2 \rrbracket$. Then $\Tj_j$ can be extended on $\Cpol{0}$.
\end{lemma}
\begin{proof}
    First, Lemmas \ref{prop:Sextended} and \ref{Spol} show that $\S$ can be extended as an operator from $\Cpol{0}$ to $\Cpol{0}$. Let us then show that one can extend $\Q_j(\LL)$ on $\Cpol{0}$ and the result will follow by Lemma \ref{prop:rewrite_of_T}. By definition, one has
    \begin{align*}
        \Q_j(\LL) & = p^{-j-1}\RR(2)^{j} - \underset{\alpha=1}{\overset{j}{\sum}} \, a_\alpha p^{\alpha-j-1} \RR(2)^{j-\alpha}P \bigl ( p(\LL-2) \bigl )^{-1},
    \end{align*}
    where
    \begin{align*}
        P \bigl ( p(\LL-2) \bigl )^{-1} = p^{-2p+1}\RR(2) \RR \left ( 2-\frac{1}{ p } \right )\cdots \RR \left ( 2-\frac{2p-2}{ p } \right ).
    \end{align*}
    One deduces that $ \Q_j(\LL)$ is a sum of composition of resolvent operators of $\LL$. Hence, it can be extended on $\Cpol{0}$.
\end{proof}
We can now state the two main properties of $\Tj_j$.

\begin{proposition}\label{Tcontinuite}
    Consider a sequence $(\varphi_n)$ of functions in $\Cpol{0}$, converging pointwise to a function $\varphi$, such that there exist a constant $C$ and an integer $d$ such that
    \[\forall n \in \N, \forall x \in \R, \quad |\varphi_n(x)| \leq C(1+|x|^d)\]
 Then, for all $j$ in $\llbracket 0, 2p-2\rrbracket$, the sequence $\bigl ( \Tj_j(\varphi_n) \bigr )_n$ converges pointwise to $\Tj_j(\varphi)$.
\end{proposition}
\begin{proof}
    Since $\Tj_j$ is the composition between $\Q_j(\LL)$ and $\S$, the extension/continuity of $\Q_j(\LL)$ established above, Lemma \ref{Scontinuite} and Lemma \ref{Spol} combined together yield the result.
\end{proof}

\begin{proposition}\label{Tpol}
    Let $\phi$ be an element of $\Cpol{0}$ verifying
    \begin{align*}
        \forall x \in \R, \quad |\phi(x)| \leq C(1+|x|^d)
    \end{align*}
    for some $C>0$ and some nonnegative integer $d$. Then, for all $j$ in $\llbracket 0, 2p-2\rrbracket$, there exists $B_{j,p}>0$ such that
    \begin{align*}
        \forall x \in \R, \quad |\Tj_j(\phi)(x)| \leq CB_{j,p}(1+|x|^{d+1})
    \end{align*}
    where $B_{j,p}$ depends only on $d$, $j$ and $p$. Furthermore, $\Tj_j(\phi)$ is continuous.
 \end{proposition}
 \begin{proof}
     Since $\Tj_j$ is the composition between $\Q_j(\LL)$ and $\S$, the result is obtained by combining  the polynomial-growth bound for $\S$ in Lemma \ref{Spol} together with the boundedness of $\Q_j(\LL)$ on $\Cpol{0}$ (as a finite composition of resolvents).
 \end{proof}
\subsection{Extending the equalities}
We want to extend Proposition \ref{rel1} for functions $\phi$ in $\Cpol{0}$ with Hermite rank $r\geq 2$.
\begin{lemma}\label{Hrankapprox}
    Let $\varphi$ be an element of $\Cpol{0}$ with Hermite rank $r>0$. Suppose
    \[\forall x \in \R, \quad |\varphi(x)| \leq C(1+|x|^d),\]
    where $C>0$ and $d$ is a non-negative integer. Then there exists a sequence $(\varphi_n)$ of functions in $\Cpol{\infty}$ with Hermite rank $\ge r$ such that:
    \begin{enumerate}
        \item $(\varphi_n)$ converges pointwise to $\varphi$ ;
        \item there exists $M>0$ such that $\forall n \in \N, \forall x \in \R, \quad |\varphi_n(x)| \leq M(1+|x|^d)$.
    \end{enumerate}
\end{lemma}
\begin{proof}
    Consider a non-negative $\mathscr{C}^\infty$ compactly supported function $\rho$ such that $\operatorname{Supp} \rho \subseteq [-1,1]$ and $\int_\R \rho(x)\, \dint x = 1$. Define $(\rho_n)_n$ by 
   \[\rho_n(x) := n\,\rho  \left (nx \right )\] 
   and denote $\varphi_n := \projgeq{r}(\varphi \ast \rho_n)$. Then, for all $n$ in $\N$, the function $\varphi_n$ belongs to $\Cpol{\infty}$ and has Hermite rank greater than or equal to $r$. Indeed, let us show that $\varphi \ast \rho_n \in \Cpol{\infty}$, for all $n \in \N$. By Lemma \ref{Projpol}, we will get that $\varphi_n \in \Cpol{\infty}$. Let $x \in \R$ and $n \in \N$, one has 
   \begin{align*}
       |\varphi\ast \rho_n(x)| & \leq \int_\R C(1+|x-y|^d)\rho_n(y) \, \dint y \nonumber  \leq \int_\R C(1+(|x|+|y|)^d)\rho_n(y) \, \dint y.
   \end{align*}
Since $\operatorname{Supp} \rho \subseteq [-1,1]$, one gets
\begin{align*}
    \int_\R C(1+(|x|+|y|)^d)\rho_n(y) \, \dint y & \leq \int_\R C(1+(|x|+1)^d)\rho_n(y) \, \dint y = C(1+(|x|+1)^d) \\
    & \leq C(1+2^d+2^d|x|^d) \leq M(1+|x|^d),
\end{align*}
where $M\coloneq (2^d+1)C$. Then, let $k$ be a positive integer. By property of the convolution product, since $\rho_n$ is infinitely differentiable, so is $\varphi\ast\rho_n$, and one has
 \begin{align*}
       |\bigl(\varphi\ast \rho_n\bigr)^{(k)}(x)| & = |\varphi\ast \rho^{(k)}_n(x)| \leq\int_\R C(1+(|x|+1)^d)\,|\rho^{(k)}_n(y)| \, \dint y \\
       & \leq C(1+(1+|x|)^d) \int_\R |\rho^{(k)}_n(y)| \, \dint y \leq M_{k,n}(1+|x|^d),
\end{align*}
where $M_{k,n}\coloneq C(1+2^d) \int_\R |\rho^{(k)}_n(y)| \, \dint y >0$. Hence $\forall n \in \N, \; \varphi\ast\rho_n \in \Cpol{\infty}$ and therefore $\varphi_n=\projgeq r(\varphi\ast\rho_n)\in\Cpol{\infty}$. Finally, since $(\varphi \ast \rho_n)_n$ converges uniformly on all compact sets to $\varphi$, by Lemma \ref{Projcontinuite}, the sequence $(\varphi_n)_n$ converges pointwise to $\projgeq{r}(\varphi) = \varphi$.
\end{proof}
\begin{proposition}\label{form:Hrankcont}
    Let $\varphi$ be an element of $\Cpol{0}$ with Hermite rank $r\geq 2$. Then one has
\begin{align}\label{eq:Hrankcont}
   \E[\varphi(F)(\Ga(F)-p)] & =\underset{j=0}{\overset{2p-2}{\sum}}\, \E \bigl [\Tj_j(\varphi)(F)\,(\Ga(F)-p)\,\L^j(\Ga(F)-p)\bigr ].
\end{align}
\end{proposition}
\begin{proof}
Suppose that \[\forall x \in \R, \quad |\varphi(x)| \leq C(1+|x|^d) \]
for some $C>0$, $d \in \N$. We consider a sequence $(\varphi_n)_n$ of functions in $\Cpol{\infty}$ with Hermite rank $\ge r$ converging to $\varphi$ given by Lemma \ref{Hrankapprox}. In particular, by Proposition \ref{rel1}, one has
\begin{align*}
\E[\varphi_n(F)(\Ga(F)-p)] &= \underset{j=0}{\overset{2p-2}{\sum}} \, \E \big[\Tj_j(\varphi_n)(F)\,(\Ga(F)-p)\,\L^j(\Ga(F)-p)\big]
\end{align*}
for all $n \in \N$.
On one hand, since the sequence $(\varphi_n)_n$ converges pointwise to $\varphi$, we get
\begin{align*}
    \varphi_n(F) \underset{n\rightarrow+\infty}{\overset{a.s.}{\longrightarrow}} \varphi(F).
\end{align*}
Then, since 
\[\forall n\in \N, \quad |\varphi_n(F)| \leq C(1+|F|)^d, \] 
$F$ admits a moment of order $2d$ and $\Ga(F){\color{red}\ }$ admits a moment of order $2$, then Lebesgue's dominated convergence theorem yields
   \begin{align*}
\E\bigl [\varphi_n(F)(\Ga(F)-p)\bigr ]\underset{n\rightarrow+\infty}{\longrightarrow} \E\bigl [\varphi(F)(\Ga(F)-p)\bigr ].
   \end{align*}
   On the other hand, the use of Propositions \ref{Tcontinuite}, \ref{Tpol} and the same arguments leads to
   \begin{align*}
        \E\bigl [\Tj_j(\varphi_n)(F)\,(\Ga(F)-p)\,\L^j(\Ga(F)-p) \bigr ] \underset{n\rightarrow+\infty}{\longrightarrow} \E\bigl [\Tj_j(\varphi)(F)\,(\Ga(F)-p)\,\L^j(\Ga(F)-p) \bigr ],
   \end{align*}
   for every $j$ in $\llbracket 0, 2p-2\rrbracket$, thus concluding the proof.
   \end{proof} 
\begin{corollary}\label{formulaextended1}
    Let $\phi$ be a continuous bounded function. Then
\begin{align*}
   \E[\projgeq{2}(\phi)(F)(\Ga(F)-p)] & =\underset{j=0}{\overset{2p-2}{\sum}}\, \E \bigl [\Tj_j(\phi)(F)\,(\Ga(F)-p)\,\L^j(\Ga(F)-p) \bigr ].
\end{align*}
\end{corollary}
\begin{proof}
    The function $\projgeq{2}(\phi)$ belongs to $\Cpol{0}$ and has Hermite rank $r\geq 2$, hence by Proposition \ref{form:Hrankcont} one has
    \begin{align*}
   \E[\projgeq{2}(\phi)(F)(\Ga(F)-p)] & =\underset{j=0}{\overset{2p-2}{\sum}}\, \E \bigl [\Tj_j \bigl (\projgeq{2}(\phi) \bigr )(F)\,(\Ga(F)-p)\,\L^j(\Ga(F)-p) \bigr ].
\end{align*}
Finally, the operator of the second derivative is zero on polynomials of degree strictly less than $2$, hence by linearity we have
$\Tj_j \bigl (\projgeq{2}(\phi) \bigr ) = \Tj_j(\phi)$
which concludes the proof.
\end{proof}

 \subsection{Bounding the remainder}\label{ss:boundingthedistance}
We now prove the following result, which concludes the proof of Theorem \ref{maintheorem}.
\begin{proposition}\label{prop:bound_on_Stein_discrepancy}
There exists $\widetilde{C}_{p} >0$ such that, for all $F \in \WW_p$ with $\E[F^2]=1$, and all $h\in \mathscr{C}^0_b(\R)$ with $\|h\|_\infty \le 1$,
\begin{align}\label{ineq:bound_on_Stein_discrepancy}
    \bigl|\E \bigl[ \operatorname{Proj}_{\geq2}(\phi'_h)(F)\,(\Ga(F)-p) \bigr] \bigr|
    \;\le\; \widetilde{C}_p\, \Var \bigl(\Ga(F)\bigr).
\end{align}
\end{proposition}

\begin{proof}
Fix $F \in \WW_p$ with $\E[F^2]=1$ and $h \in \mathscr{C}_b^0(\R)$ with $\|h\|_\infty \le 1$. By Proposition \ref{prop:proprietes_solstein}, the derivative $\phi_h'$ is continuous and bounded by $4$. Proposition \ref{formulaextended1} then yields
\begin{align*}
\E \bigl[ \operatorname{Proj}_{\geq2}(\phi'_h)(F)\,(\Ga(F)-p) \bigr]
  \;=\; \sum_{j=0}^{2p-2} \E \bigl[ \Tj_j(\phi'_h)(F)\,(\Ga(F)-p)\,\L^j( \Ga(F) -p)\bigr].
\end{align*}
Using the bounds in Proposition \ref{Tpol}, there exist $B_{j,p} > 0$ such that
\begin{align}\label{est:edge-b}
& \bigl|\E \bigl[ \operatorname{Proj}_{\geq2}(\phi'_h)(F)\,(\Ga(F)-p) \bigr]\bigr| \nonumber \\
&\le \sum_{j=0}^{2p-2} B_{j,p}\, \|\phi'_h\|_{\infty}\,
   \E \Bigl[ \bigl(1+|F|\bigr) \bigl |(\Ga(F)-p)\L^j (\Ga(F)-p) \bigr | \Bigr]\nonumber\\
&\le 4 \sum_{j=0}^{2p-2} B_{j,p}\,\E \bigl[(\Ga(F)-p)^2\bigr]^{1/2}\,\E \bigl[|\L^j( \Ga(F)-p)|^2\bigr]^{1/2}\nonumber\\
&\quad + 4 \sum_{j=0}^{2p-2} B_{j,p}\,\E \bigl[F^4\bigr]^{1/4}\,\E \bigl[(\Ga(F)-p)^4\bigr]^{1/4}\,\E \bigl[|\L^j (\Ga(F)-p)|^2\bigr]^{1/2},
\end{align}
by Cauchy-–Schwarz.

By hypercontractivity (applied on the relevant chaoses),
\[
\E \bigl[(\Ga(F)-p)^4\bigr]^{1/4} \;\le\; 3^{p-1}\,\E \bigl[(\Ga(F)-p)^2\bigr]^{1/2},
\qquad
\E \bigl[F^4\bigr]^{1/4} \;\le\; 3^{p/2}\,\E[F^2]^{1/2} = 3^{p/2}.
\]
We also use:

\begin{lemma}\label{bounds on L}
Let $d\in\N^*$, $W\in \WW_{\le d}$, and $j\in\N$. Then
\[
\E \bigl[|\L^j W|^2\bigr]^{1/2} \;\le\; d^{\,j}\,\E \bigl[|W|^2\bigr]^{1/2}.
\]
\end{lemma}

\begin{proof}
Write $W=\sum_{k=0}^d G_k$ with $G_k\in\WW_k$ orthogonal in $L^2$. Since $\L^j G_k = (-k)^j G_k$,
\[
\E \bigl[|\L^j W|^2\bigr] \;=\; \sum_{k=0}^d k^{2j}\,\E \bigl[G_k^2\bigr]
\;\le\; d^{2j}\sum_{k=0}^d \E \bigl[G_k^2\bigr]
\;=\; d^{2j}\,\E \bigl[|W|^2\bigr].
\]
Taking square roots gives the claim.
\end{proof}

Since $\Ga(F)-p \in \WW_{\le 2p-2}$, Lemma \ref{bounds on L} gives
\[
\E \bigl[|\L^j (\Ga(F)-p)|^2\bigr]^{1/2}
\;\le\; (2p-2)^{\,j}\,\E \bigl[(\Ga(F)-p)^2\bigr]^{1/2},
\qquad j=0,\dots,2p-2.
\]
Plugging these estimates into \eqref{est:edge-b}, we obtain
\[
\bigl|\E \bigl[ \operatorname{Proj}_{\geq2}(\phi'_h)(F)\,(\Ga(F)-p) \bigr]\bigr|
\;\le\; \widetilde{C}_p\,\E \bigl[(\Ga(F)-p)^2\bigr]
\;=\; \widetilde{C}_p\,\Var \bigl(\Ga(F)\bigr),
\]
where
\[
\widetilde{C}_p \;:=\; 4\sum_{j=0}^{2p-2} 3^{p-1}\,B_{j,p}\,\bigl(1+3^{p/2}\bigr)\,(2p-2)^{\,j}
\]
depends only on $p$ (and the fixed constants $B_{j,p}$), not on $F$ or $h$.
\end{proof}

Finally, taking the supremum in \eqref{ineq:bound_on_Stein_discrepancy} over $h \in \mathscr{C}_b^0(\R)$ with $\|h\|_\infty \le 1$, and using Proposition \ref{prop:overview of proof}, we get
\[
\dTV(\prob_F, \boldsymbol{\gamma}_{F,1}) \;\le\; C_{p}\,\Var \Ga(F),
\qquad C_{p} \coloneq \frac{\widetilde{C}_{p}}{2p}.
\]
\section{Higher values of $m$}\label{s:higherorder}
In this section, we prove the rest of Theorem \ref{maintheorem}, that is for $m>1$. Let us fix a positive integer $m$ and an element $F$ of $\WW_p$ such that $\E[F^2]=1$.
\subsection{A more general formula} 
The idea is to apply Proposition \ref{rel1} on each term
\begin{align}\label{expansion:explanations}
\E \bigl [ \Tj_j(\varphi)(F)(\Ga(F)-p)\L^j(\Ga(F)-p) \bigr],
\end{align}
and then iterate the whole process. In this regard, we prove a more general statement, replacing $\Ga(F)-p$ in $\E[\varphi(F)(\Ga(F)-p)]$ by a random variable $W$ in a finite sum of chaoses. Since, by the product formula, $(\Ga(F)-p)\L^j(\Ga(F)-p)$ lies in $\WW_{\leq 4p-4}$, it will allow us to write an expansion for \eqref{expansion:explanations}. Keeping in mind the proof of Lemma \ref{dev:pol}, we need the following result. Set $d_m := d_{m,p} = (2p-2)m$.

\begin{lemma}\label{prop:polynome_annulateur}
Denote by $P_{d_m}$ the polynomial $X(X+1)\cdots (X+d_m)$. Then for all $W \in \WW_{\leq d_m}$, one has $P_{d_m}(\L)(W) = 0$.
\end{lemma}
In the rest of the article, we will denote $P_{d_m} = {\sum}^{d_m+1}_{\alpha=1}a_{d_m,\alpha}X^\alpha$. Now that we have a canceling polynomial for $W$, we can prove the next lemma by mildly adjusting the proof of Lemma \ref{dev:pol}.

\begin{lemma}\label{dev:polgen}
    Let $\phi$ be an element of $\Cpol{2d_m+2}$ and let $W$ be an element of $\WW_{\leq d_m}$. Then one has
\begin{align}\label{eq:dev:polgen}
   \E[P_{d_m}(p\LL)(\phi)(F)W] & = -\underset{j=0}{\overset{d_m}{\sum}}\, \E \bigl [\bigl (P_{d_m,j}(p\LL)(\phi) \bigr )''(F)(\Gamma(F)-p)\L^jW \bigr ]
\end{align}
where $P_{d_m,j} := \underset{\alpha=j+1}{\overset{d_m+1}{\sum}} a_{d_m,\alpha}X^{\alpha-j-1}$. 
\end{lemma}
\begin{proof}
    Using Lemma \ref{prop:polynome_annulateur} and replacing $\Ga(F)-p$ by $W$ and $P$ by $P_{d_m}$ in the proofs of Lemmas \ref{trick1_rec} and \ref{dev:pol}, we get \eqref{eq:dev:polgen}.
\end{proof}
Then, we proceed to invert the operator $P_{d_m}(p\LL)$ as in the proof of Proposition \ref{rel1}.
\begin{lemma}\label{prop:invertibility}
    Let $\varphi$ be an element of $\Cpol{2d_m+2}$ with Hermite rank $r\geq 2m$. Then $P_{d_m}(\LL)^{-1}\varphi$ is well-defined.
\end{lemma}
\begin{proof} 
The roots of the polynomial $P_{d_m}(pX)$ are the $-\frac{i}{p}$, where $i\in \llbracket 0, (2p-2)m\rrbracket$. All these shifts satisfy $\frac{i}{p}\le (2p-2)m/p=2m-\frac{2m}{p}<2m$. By Proposition \ref{prop:fractionLOU_hermite_rank}, each factor $(\LL+\frac{i}{p})$ is invertible on functions with Hermite rank at least $2m$, hence the product $P_{d_m}(p\LL)$ is invertible on such functions.
\end{proof}
 The following generalization of Proposition \ref{rel1} holds.

\begin{lemma}\label{form:W}
    Let $W$ be an element of $\WW_{\leq d_m}$ and let $\varphi$ be an element of $\Cpol{2d_m+2}$ with Hermite rank $r\geq 2m$. Then one has
    \begin{align}\label{eq:form:W}
   \E \bigl [\varphi(F)W \bigr ] & =  \underset{j=0}{\overset{d_m}{\sum}} \, \E \bigl [ \Tj_{m,j}(\varphi)(F)(\Ga(F)-p)\L^j W \bigr ]
\end{align}
where $\Tj_{m,j} \varphi:= -P_{d_m,j}\bigl ( p(\LL-2) \bigr )P_{d_m} \bigl ( p(\LL-2) \bigr )^{-1} \varphi''$.
\end{lemma}
\begin{remark}
    In particular, for all $j \in \llbracket 0, d_1 \rrbracket $, one has $\Tj_{1,j} = \Tj_j$. 
\end{remark}
\begin{proof}
    Consider $\phi := P_{d_m}(p\LL)^{-1} \varphi$. By Proposition \ref{prop:resolvent_commutation}, $\phi$ belongs to $\Cpol{2d_m+2}$. Hence, we can write equality \eqref{eq:dev:polgen} for $\phi$. The result, as in the proof of Proposition \ref{rel1}, follows by using the commutation relations \eqref{eq:commutationLOU} and \eqref{eq:resolvent_commutation}.
\end{proof}
As previously in Subsection \ref{ss:studyofT}, the operator $\Tj_{m,j}$ can be expressed as a composition of a fraction of $\LL$ and the operator $\S$. The proof of the next lemma is completely the same as the proof of Lemma \ref{prop:rewrite_of_T}.
\begin{lemma}\label{prop:rewrite_of_Tj}
    Let $j$ be an element of $\llbracket 0, d_m \rrbracket$. For all $\varphi \in \Cpol{2d_m+2}$ with Hermite rank $r \geq 2m$, one has 
    \begin{align}\label{eq:rewrite_of_Tj}
        \Tj_{m,j}(\varphi) & = -\Q_{m,j}(\LL) \S(\varphi)
    \end{align}
    where 
    \begin{align*}\label{eq:Qmj}
   \Q_{m,j}(X) & \coloneq p^{-j-1}(X-2)^{-j} - \underset{\alpha=1}{\overset{j}{\sum}} \, a_{d_m,\alpha} p^{\alpha-j-1} (X-2)^{\alpha-j}P_{d_m} \bigl ( p(X-2) \bigl )^{-1}.
\end{align*}
\end{lemma}

We can now extend \eqref{form:W} to continuous functions.

\begin{proposition}\label{form:Wextended}
    Let $W$ be an element of $\WW_{\leq d_m}$ and let $\varphi$ be an element of $\Cpol{0}$ with Hermite rank $r\geq 2m$. Then one has
\begin{align*}
   \E \bigl [\varphi(F)W \bigr ] & =  \underset{j=0}{\overset{d_m}{\sum}} \, \E \bigl [ \Tj_{m,j}(\varphi)(F)(\Ga(F)-p)\L^j W \bigr ].
\end{align*}
\end{proposition}

\begin{remark}
    Let us observe that, unlike the case $m=1$, the operators $\Tj_{m,j}$ cannot be extended to operators acting on $\Cpol{0}$, without Hermite rank consideration. 
\end{remark}
\subsection{Extending the equalities}
In a similar fashion as for the proof of Proposition \ref{form:Hrankcont}, we use Lemma \ref{Hrankapprox} to extend \eqref{eq:form:W} to functions $\varphi$ in $\Cpol{0}$ with Hermite rank $r\geq 2m$. First, we state the properties of continuity of the operators $\Tj_{m,j}$, whose proofs are the same as those for Proposition \ref{Tcontinuite} and \ref{Tpol}.
\begin{proposition}\label{Tjcontinuite}
    Let $k$ be a positive integer and let $j_k$ be an element of $\llbracket 0,d_k\rrbracket$. Consider a sequence $(\varphi_n)$ of continuous functions of Hermite ranks greater than or equal to $2k$, converging pointwise to a function $\varphi$, such that there exist a constant $C$ and an integer $d\geq 0$ such that
    \[\forall n \in \N, \forall x \in \R, \quad |\varphi_n(x)| \leq C(1+|x|^d).\]
 Then the sequence $\bigl ( \Tj_{k,j_k}(\varphi_n) \bigr )_n$ converges pointwise to $\Tj_{k,j_k}(\varphi)$.
\end{proposition}

\begin{lemma}
    Let $k$ be a positive integer and let $\phi$ be an element of $\Cpol{0}$, with Hermite rank $r \geq 2k$, such that $\forall x \in \R, \; |\phi(x)| \leq C(1+|x|^{d})$, for some constant $C>0$ and integer $d \geq 0$. Then, there exists $U_{k,d,p}>0$ such that 
    \begin{align*}
        \forall x \in \R, \quad \big| P_{d_k} \bigl (p(\LL-2) \bigr )^{-1}\phi(x) \big| \leq C\,U_{k,d,p}\,(1+|x|^{d+\beta_k}),
    \end{align*}
    where $\beta_{k} \coloneq 1+\sum_{i=2p+1}^{d_k} \lceil i/p-2 \rceil$.
\end{lemma}

\begin{proposition}\label{Tjpol}
    Let $k$ be a positive integer and let $j_k$ be an element of $\llbracket 0,d_k\rrbracket$. Then there exist $B_{k,j_k} >0$ such that for every $\phi$ element of $\Cpol{0}$, with Hermite rank $r \geq 2k$, verifying 
    \[\forall x \in \R, \quad |\phi(x)|\leq C(1+|x|^d),\]
    where $C>0$ and $d$ is a nonnegative integer, one has 
    \begin{align*}
        \forall x \in \R, \quad |\Tj_{k,j_k}(\phi)(x)| \leq C\,B_{k,j_k}\,(1+|x|^{d+1+\beta_k}).
    \end{align*}
    Furthermore, $\Tj_{k,j_k}(\phi)$ is continuous.
\end{proposition}

Then we can state the next proposition, with the same proof as the one for Proposition \ref{rel1}.
\begin{proposition}\label{form:Wextended}
    Let $W$ be an element of $\WW_{\leq d_m}$ and let $\varphi$ be an element of $\Cpol{0}$ with Hermite rank $r \geq 2m$. Then one has
    \begin{align}\label{eq:form:Wextended}
   \E \bigl [\varphi(F)W \bigr ] & =  \underset{j=0}{\overset{d_m}{\sum}} \, \E \bigl [ \Tj_{m,j}(\varphi)(F)(\Ga(F)-p)\L^j W \bigr ].
\end{align}
\end{proposition}

Now, we have everything to expand $\E[\varphi(F)(\Ga(F)-p)]$ further, but first we need a couple of notations to make the exposition of the expansion simpler.
Denote by $\J_m$ the set $\llbracket 0, d_1 \rrbracket \times \cdots \times  \llbracket 0, d_m \rrbracket$, and for $\j = (j_1, \ldots, j_m) \in \J_m$, denote by $W_{\j}=W_{j_1,\ldots,j_m}$ the random variable $(\Ga(F)-p) \L^{j_m}(\Ga(F)-p) \cdots \L^{j_1}(\Ga(F)-p)$.

\begin{proposition}\label{form:Wfinal}
    Let $\varphi$ be an element of $\Cpol{0}$ with Hermite rank $r \geq 4m-2$. Then one has
    \begin{align}\label{eq:form:Wfinal}
        \E[\varphi(F)(\Gamma(F)-p)] & = \underset{\mathbf{j} \in \J_m}{\sum} \, \E \bigl [ \boldsymbol{\Tj}^{(m)}_{\mathbf{j}}(\varphi)(F)W_{\mathbf{j}} \bigr ]
    \end{align}
    where $\boldsymbol{\Tj}_{\mathbf{j}}^{(m)}:=  \Tj_{m,j_m} \cdots \Tj_{1,j_1}$.
\end{proposition}
\begin{proof} 
The proof is divided into several steps.

\medskip 
\noindent \underline{Step 1}: We apply Proposition \ref{form:Wextended} to $W = \Gamma(F)-p \in \WW_{\leq d_1}$ to get
\begin{align}
     \E \bigl [\varphi(F)(\Gamma(F)-p)\bigr ] & =  \underset{j_1=0}{\overset{d_1}{\sum}} \, \E \bigl [ \Tj_{1,j_1}(\varphi)(F)(\Ga(F)-p)\L^{j_1} (\Ga(F)-p) \bigr ].
\end{align}
Then, we want to iterate and apply Proposition \ref{form:Wextended} again for each term 
\[\E \bigl [ \Tj_{1,j_1}(\varphi)(F)(\Ga(F)-p)\L^{j_1}(\Ga(F)-p) \bigr ],\] 
where $j_1 \in \llbracket 0, d_1 \rrbracket$, that is, choosing $W = (\Ga(F)-p)\L^{j_1}( \Ga(F)-p) \in \WW_{\leq d_2}$. To do so, we need to verify that the Hermite rank of $\Tj_{1,j_1}(\varphi)$ is greater than $1$. More generally, we need to show that, for every $k \in \llbracket 1, m-1 \rrbracket$ and $j_k \in \llbracket 0, d_k \rrbracket$, the Hermite rank of $\Tj_{k,j_k}\cdots \Tj_{1,j_1}(\varphi)$ is greater than $2k+1$, which will ensure that we iteratively apply Proposition \ref{form:Wextended} to get \eqref{eq:form:Wfinal}.

\medskip 
\noindent \underline{Step 2}: Fix $k \in \llbracket 1, m-1 \rrbracket$ and $j_k \in \llbracket 0, d_k \rrbracket$. Let $u$ be in $\Cpol{\infty}$ with Hermite rank $r$. Then $\Tj_{k,j_k} u$ has Hermite rank $r-2$. Indeed, 
by equality \eqref{eq:rel_hermite_coef_derivative}, $u''$ has Hermite rank greater or equal to $r-2$. Then, by Lemma \ref{prop:rewrite_of_Tj}, $\Tj_{k,j_k} u = \Q_{k,j_k}(\LL)(\LL-2)^{-1}u''$, where $\Q_{k,j_k}$ is a sum of fractions of the form like in Corollary \ref{prop:fractionLOU_hermite_rank}. Hence, by Corollary \ref{prop:fractionLOU_hermite_rank}, we get that $\Tj_{k,j_k} u$ has Hermite rank greater or equal to $r-2$.

\medskip 
\noindent \underline{Step 3}: Using Lemma \ref{Hrankapprox}, we approximate $\varphi$ pointwise by a sequence $(\varphi_n)_n$ in $\Cpol{\infty}$ with Hermite rank $r$ such that there exist $A >0$ and $d\in\N$ verifying 
\begin{align}
\forall n \in \N, \forall x \in \R,\quad |\varphi_n(x)| \leq A(1+|x|^d). 
\end{align}
Then, by Proposition \ref{Tjcontinuite} and \ref{Tjpol}, the sequence $\bigl ( \Tj_{k,j_k}(\varphi_n) \bigr )_n$ converges pointwise to $\Tj_{k,j_k}\cdots \Tj_{1,j_1}(\varphi)$ and we have
\begin{align*}
    \forall x \in \R, \forall n \in \N, \quad |\Tj_{k,j_k}(\varphi_n)(x)| \leq A\,B_{k,j_k}\,(1+|x|^{d+1+\beta_k})
\end{align*}
where $B_{k,j_k}>0$ does not depend on $n$. By Lemma \ref{prop:conv_and_hermiterank}, we get that $\Tj_{k,j_k}\cdots \Tj_{1,j_1}(\varphi)$ has Hermite rank greater than or equal to $r-2$.

\medskip 
\noindent \underline{Step 4}: Applying $k$ times Step 3, we get that the Hermite rank of $\Tj_{k,j_k}\cdots \Tj_{1,j_1}(\varphi)$ is greater or equal to $r-2k$. Since $r\geq 4m-2$ and $k \leq m-1$, we get that the Hermite rank of $\Tj_{k,j_k}\cdots \Tj_{1,j_1}(\varphi)$ is $\ge 4m-2-2(m-1)=2m$. Hence $2m \ge 2k+1$, and by Proposition \ref{form:Wextended} we get
\begin{eqnarray*}
    && \E \bigl [\Tj_{k,j_k}\cdots \Tj_{1,j_1}(\varphi)(F)(\Gamma(F)-p)\bigr ] \\ 
    && =  \underset{j_{k+1}=0}{\overset{d_{k+1}}{\sum}} \, \E \bigl [ \Tj_{k+1,j_{k+1}}\Tj_{k,j_k}\cdots \Tj_{1,j_1}(\varphi)(F)(\Ga(F)-p)\L^{j_{k+1}} W_{j_1, \ldots, j_k} \bigr ].
\end{eqnarray*}
A straightforward induction yields then \eqref{eq:form:Wfinal}.
\end{proof}

\begin{corollary}\label{form:steinbound}
    Let $\phi$ be a continuous bounded function. Then one has
    \begin{align*}
        \E[\projgeq{4m-2}(\phi)(F)(\Gamma(F)-p)] & = \underset{\mathbf{j} \in \J_m}{\sum} \,  \E \bigl [ \boldsymbol{\Tj}^{(m)}_{\mathbf{j}}\bigl (\projgeq{4m-2}(\phi) \bigr )(F) W_{\mathbf{j}} \bigr ] 
    \end{align*}
\end{corollary}
\begin{proof}
    In the same fashion as in the proof of Proposition \ref{formulaextended1}, we apply \eqref{eq:form:Wfinal} to $\varphi = \projgeq{4m-2}(\phi)$ which verifies the proper conditions. 
\end{proof}

\begin{proposition}\label{prop:bigTjpol}
        Let $\j$ be an element of $\operatorname{J}_m$. Then there exist $B_{\mathbf{j},m} >0$ such that for every $\phi$ continuous and bounded, one has 
    \begin{align*}
        \forall x \in \R, \quad \big|\boldsymbol{\Tj}_{\mathbf{j}}^{(m)}\bigl (\projgeq{4m-2}(\phi) \bigr )(x)\big| \leq B_{\mathbf{j},m}\,(1+|x|^{l_m})\,\|\phi\|_\infty,
    \end{align*}
    where $l_m \coloneq 5m-2+\sum_{k=2}^{m} \beta_k$.
\end{proposition}
\begin{proof}
    Since $\boldsymbol{\Tj}_{\j}^{(m)}$ is the composition of $m$ operators of the form $\Tj_{k,j_k}$, the result follows from applying $m$ times Proposition \ref{Tjpol}.
\end{proof}

\subsection{Final estimation of the remainder}\label{ss:laststep}
In this section, we finalize the proof of Theorem \ref{maintheorem} with the following proposition, analogous to Proposition \ref{prop:bound_on_Stein_discrepancy}.
\begin{proposition}\label{prop:laststep}
    There exists a constant $\widetilde{C}_{p,m}>0$ such that, for all $F \in \WW_p$ with $\E[F^2]=1$, and all $h \in \mathscr{C}^0_b(\R)$ with $\norm{h}_\infty \leq 1$,
    \begin{align}\label{ineq:laststep}
       \bigl | \E[\projgeq{4m-2}(\phi'_h)(F)(\Gamma(F)-p)] \bigr | &\leq \widetilde{C}_{p,m}\bigl (\Var\Ga(F) \bigr )^{\frac{m+1}{2}}.
\end{align}
\end{proposition}
\begin{proof}
Let $h$ be an element of $\mathscr{C}^0_b(\R)$, such that $\norm{h}_\infty \leq 1$. As previously stated in Subsection \ref{ss:boundingthedistance}, the function $\phi'_h$ is continuous and bounded by $4$. Corollary \ref{form:steinbound} and Proposition \ref{prop:bigTjpol} yield then
\begin{eqnarray}\label{ineq:laststep_paso1}
       && \bigl | \E[\projgeq{4m-2}(\phi'_h)(F)(\Gamma(F)-p)] \bigr | = \Bigl |\underset{\mathbf{j} \in \operatorname{J}_m}{\sum} \, \E \bigl [ \boldsymbol{\Tj}_{\mathbf{j}}^{(m)}\bigl (\projgeq{4m-2}(\phi_h') \bigr )(F)W_{\mathbf{j}} \bigr ] \Bigr | \nonumber \\ 
       && \leq \underset{\mathbf{j} \in \operatorname{J}_m}{\sum} \,\E \Bigl [ \,\bigl|\boldsymbol{\Tj}_{\mathbf{j}}^{(m)}\bigl (\projgeq{4m-2}(\phi_h') \bigr )(F)\bigr|\,|W_{\mathbf{j}}|\, \Bigr ] \nonumber\\
       && \leq \|\phi'_h\|_\infty\underset{\mathbf{j} \in \operatorname{J}_m}{\sum} B_{\mathbf{j},m}\,\E \bigl [ (1+|F|^{l_m})\, |W_{\mathbf{j}}| \bigr ] \nonumber \\
       && \leq 4\,\underset{\mathbf{j} \in \operatorname{J}_m}{\sum} B_{\mathbf{j},m}\,  \E \bigl [|W_{\mathbf{j}}|^{2} \bigr ]^{1/2}
        + 4\,\E \bigl [ |F|^{2l_m}\bigr ]^{1/2} \underset{\mathbf{j} \in \operatorname{J}_m}{\sum} B_{\mathbf{j},m}\,  \E \bigl [|W_{\mathbf{j}}|^{2} \bigr ]^{1/2},
\end{eqnarray}
where we used the Cauchy--Schwarz inequality. As before, we will now bound the terms from the right-hand side of \eqref{ineq:laststep_paso1}. By an iteration of Cauchy--Schwarz and hypercontractivity \eqref{eq:hypercontractivity} and Lemma \ref{bounds on L}, one gets the following lemma.
\begin{lemma}\label{prop:normbound_on_W}
    Let $\mathbf{j}$ be an element of $\J_m$. Then
    \begin{align}\label{ineq:normbound_on_W}
        \E \bigl [|W_{\mathbf{j}}|^{2} \bigr ]^{1/2} & \leq \underset{i=2}{\overset{m+2}{\prod}} \, \E \bigl [|\Ga(F)-p|^{2^i} \bigr ]^{1/2^i} \left(  \underset{k=1}{\overset{m}{\prod}} \, (2^{2+m-k}-1)^{d_k/2}d_k^{j_k} \right ).
    \end{align}
\end{lemma}
\begin{proof}
    Let $d$ be a nonnegative integer and let $W$ be in $\WW_{\leq d}$. Then by Cauchy--Schwarz's inequality and Lemma \ref{bounds on L},
\begin{align}\label{ineq:normbound_on_Waux}
    \E\bigl[((\Ga(F)-p)\,\L^j W)^2\bigr]^{1/2} & \leq \E\bigl[|\Ga(F)-p|^4\bigr]^{1/4} \,\E\bigl[|\L^j W|^4\bigr]^{1/4} \nonumber \\
    & \leq 3^{\,d/2}\,d^j\, \E\bigl[|\Ga(F)-p|^4\bigr]^{1/4} \, \E\bigl[|W|^4\bigr]^{1/4},
\end{align}
using hypercontractivity on $\WW_{\le d}$ to bound 
\begin{equation*}
    \|\L^j W\|_4 \le 3^{d/2}\|\L^j W\|_2 \le 3^{d/2} d^j \|W\|_2 \le 3^{d/2} d^j \|W\|_4.
\end{equation*}
Now, since $W_{\mathbf{j}} = (\Ga(F)-p) \L^{j_m}(\Ga(F)-p) \cdots \L^{j_1}\Ga(F)$, and for all $k \in \llbracket 1,m-1 \rrbracket$, $(\Ga(F)-p)\L^{j_k}(\Ga(F)-p) \cdots \L^{j_1}(\Ga(F)-p)$ belongs to $\WW_{\leq d_k}$, an iteration of inequality \eqref{ineq:normbound_on_Waux} yields
\begin{eqnarray*}
&&\E \bigl [|W_{\mathbf{j}}|^{2} \bigr ]^{1/2} \\ 
&& \leq d_m^{j_m}\, 3^{\,d_m/2}\,\E[|\Ga(F)-p|^4]^{1/4} \, \E \Bigl [ \bigl((\Gamma(F)-p)\L^{j_{m-1}}(\Ga(F)-p) \cdots \L^{j_1}(\Ga(F)-p)\bigr)^4\Bigr]^{1/4} \\
&& \leq \underset{i=2}{\overset{m+2}{\prod}} \, \E \bigl [|\Ga(F)-p|^{2^i} \bigr ]^{1/2^i} \left(  \underset{k=1}{\overset{m}{\prod}} \, (2^{2+m-k}-1)^{d_k/2}d_k^{j_k} \right ),
\end{eqnarray*}
as desired.
\end{proof}
Then, using the hypercontractivity property \eqref{eq:hypercontractivity} in $\WW_{\leq d_1}$, one has
\begin{align}\label{ineq:hypercontractivity_on_Gamma}
    \forall i \in \llbracket 2, m+2 \rrbracket, \quad \E \bigl [(\Ga(F)-p)^{2^i} \bigr ]^{1/2^i} \leq 2^{id_1 / 2} \, \E \bigl [(\Ga(F)-p)^2 \bigr ]^{1/2}.
\end{align}
Thus, combining \eqref{ineq:normbound_on_W} and \eqref{ineq:hypercontractivity_on_Gamma}, we get
\begin{align}\label{ineq:WleqVarF}
\E \bigl [|W_{\mathbf{j}}|^{2} \bigr ]^{1/2} & \leq \left(  \underset{i=2}{\overset{m+2}{\prod}} \, 2^{id_1/2} \right )\Bigl (\E \bigl [(\Ga(F)-p)^2 \bigr ]^{1/2} \Bigr )^{m+1}\left(  \underset{k=1}{\overset{m}{\prod}} \, d_k^{j_k} \right ) \nonumber \\
& \leq c_{\mathbf{j},p,m} \Var\bigl (\Ga(F) \bigr )^{\frac{m+1}{2}},
\end{align}
where 
\begin{align*}
    c_{\mathbf{j},p,m}& \coloneq \left(  \underset{i=2}{\overset{m+2}{\prod}} \, 2^{id_1 / 2} \right )\left(  \underset{k=1}{\overset{m}{\prod}} \, (2^{2+m-k}-1)^{d_k/2}d_k^{j_k} \right ) \\
    & = 2^{\frac{(m+2)(m+3)-2}{4}d_1}\left(  \underset{k=1}{\overset{m}{\prod}} \, (2^{2+m-k}-1)^{d_k/2}d_k^{j_k} \right ).
\end{align*}
Plugging \eqref{ineq:WleqVarF} back into \eqref{ineq:laststep_paso1}, we get
\begin{eqnarray*}
       && \bigl | \E[\projgeq{4m-2}(\phi'_h)(F)(\Gamma(F)-p)] \bigr | \leq 4\,\left (\underset{\mathbf{j} \in \operatorname{J}_m}{\sum} B_{\mathbf{j},m}c_{\mathbf{j},p,m} \right )\bigl (\Var\Ga(F) \bigr )^{\frac{m+1}{2}} 
       \\ 
       && \qquad + 4\,\E \bigl [ |F|^{2l_m}\bigr ]^{1/2}\left (\underset{\mathbf{j} \in \operatorname{J}_m}{\sum} B_{\mathbf{j},m}c_{\mathbf{j},p,m} \right )\bigl (\Var\Ga(F) \bigr )^{\frac{m+1}{2}} .
\end{eqnarray*}
Another use of \eqref{eq:hypercontractivity} gives 
\[\E \bigl [ |F|^{2l_m}\bigr ]^{1/2} \leq (2l_m-1)^{pl_m/2} \, \E \bigl [ F^2 \bigr ]^{l_m/2} = (2l_m-1)^{pl_m/2},\]
since $\E \bigl[F^2\bigr ] = 1$. Hence
\begin{align*}
       \bigl | \E[\projgeq{4m-2}(\phi'_h)(F)(\Gamma(F)-p)] \bigr | &\leq \widetilde{C}_{p,m}\bigl (\Var\Ga(F) \bigr )^{\frac{m+1}{2}},
\end{align*}
where $\widetilde{C}_{p,m} \coloneq 4\left (\sum_{\mathbf{j} \in \operatorname{J}_m} B_{\mathbf{j},m}c_{\mathbf{j},p,m} \right ) \left( 1+ (2l_m-1)^{pl_m/2} \right )$.
\end{proof}
Combining Propositions \ref{prop:overview of proof} and \ref{prop:laststep} leads to 
\begin{align*}
    \dTV(\prob_F, \boldsymbol{\gamma}_{F,m}) & \leq C_{p,m}\bigl (\Var\Ga(F) \bigr )^{\frac{m+1}{2}},
\end{align*}
where $C_{p,m} \coloneq \widetilde{C}_{p,m}/2p$, thus concluding the proof of Theorem \ref{maintheorem}.
\section{Broader framework: chaoses of a Markov diffusive operator}\label{s:markovoperator}
In his insightful article \cite{Led2012}, Michel Ledoux provided a new proof of the Fourth Moment Theorem in the general framework of symmetric diffusive Markov operators using a purely spectral viewpoint. This includes, in particular, the Ornstein–Uhlenbeck operator $\L$ on the Wiener space. Building on this approach, \cite{ACP2014} offered substantially simplified arguments and extended the range of structures (e.g.\ the Laguerre and Jacobi settings from \cite{BGL2013}) for which the Fourth Moment Theorem holds. Finally, \cite{AMMP2016} generalized the Nualart–Peccati criterion by showing that the convergence of \emph{any} moment of order strictly larger than four to its standard Gaussian counterpart already implies central convergence.

It is worth emphasizing that, in these proofs, the Wiener space is used only through spectral tools, namely the generator $\L$ and the carr\'e-du-champ operator $\Gamma$. This suggests extending Theorem~\ref{maintheorem} to the abstract setting of \cite{AMMP2016}. We therefore adopt this more general framework below, which we recall for completeness.

\subsection{Definition of the structure}\label{ss:def_Markov_framework}
Let $\boldsymbol{\L}$ be a symmetric Markov operator on some state space $(E,\FF)$ with invariant and reversible probability measure $\boldsymbol{\mu}$. Denote by $\Gamarkov$ its square field operator, defined by
\begin{align*}
    \Gamarkov(X,Y) &:= \frac{1}{2} \bigl ( \Lmarkov(XY) - X \Lmarkov(Y)-Y\Lmarkov(X) \bigr )
\end{align*}
for $X$, $Y$ functions on $E$ in a suitable domain $\mathcal{A}$. When $X=Y$, we will denote $\Gamarkov(X) = \Gamarkov(X,X)$. From the definition of $\Gamarkov$, this integration by parts formula holds
\begin{align}\label{eq:IPP_markov}
    \int_E \Gamarkov(X,Y) \, \dint \boldsymbol{\mu} & =-\int_E X \Lmarkov Y \, \dint \boldsymbol{\mu} = -\int_E Y \Lmarkov X \, \dint \boldsymbol{\mu}
\end{align}
for $X,Y\in \mathcal{A}\subset L^2(E,\boldsymbol{\mu})$.

\begin{assumptions}\label{hyp:markov}
We make the following
    assumptions on the structure.
\begin{enumerate}[(a)]
    \item For any $\phi \in \Cpol{2}$, any $X \in \Dom \Lmarkov$, $\phi(X) \in \Dom \Lmarkov$ and 
    \begin{align}\label{eq:markov_diffusion}
    \Lmarkov \phi(X) = \phi''(X)\Gamarkov(X,X)+\phi'(X)\Lmarkov X.
    \end{align}
    \item The operator $-\Lmarkov$ is diagonalisable on $L^2(\boldsymbol{\mu})$ and has a discrete spectrum $\{\lambda_k\}_{k\geq 0}$ with order $0 = \lambda_0 < \lambda_1 < \lambda_2 < \cdots\ $. As a consequence, one has
    \begin{align*}
    L^2(\boldsymbol{\mu}) & = \underset{k=0}{\overset{+\infty}{\bigoplus}} \, \boldsymbol{\operatorname{Ker}}(\Lmarkov+\lambda_k\Idmarkov).
\end{align*}
    \item For any pair $(X,Y)$ of eigenfunctions of $-\Lmarkov$ associated with eigenvalues $(\lambda_p,\lambda_q)$,
    \begin{align}\label{eq:markov_chaos_stability}
    XY \in \underset{k=0}{\overset{p+q}{\bigoplus}} \,\Kermarkov(\Lmarkov+\lambda_k\Idmarkov).
    \end{align}
\end{enumerate}
\end{assumptions}

From Assumptions \ref{hyp:markov}, we get the following consequences.
    \begin{enumerate}[(i)]
        \item Assumption $\textit{(a)}$ is equivalent to 
        \begin{align}\label{eq:markov_diffusion2}
            \forall \phi\in \Cpol{1}, \forall X \in \Dom \Lmarkov, \quad \Gamarkov(\phi(X),X) = \phi'(X)\Gamarkov(X,X).
        \end{align}
        \item A consequence of Assumption $\textit{(a)}$ is the following expression: for all $p\in\N$, and $X \in \boldsymbol{\operatorname{Ker}}(\Lmarkov+\lambda_p\Idmarkov)$,
        \begin{align}\label{eq:markov_formule_Gamma}
            \Gamarkov(X) & = \frac{1}{2} (\Lmarkov+2\lambda_p\Idmarkov)X^2,
        \end{align}
        which combined with Assumption $\textit{(c)}$ yields 
        \begin{align}
        \forall X \in \boldsymbol{\operatorname{Ker}}(\Lmarkov+\lambda_p\Idmarkov), \quad \Gamarkov(X) \in \underset{k=0}{\overset{2p-1}{\bigoplus}} \, \boldsymbol{\operatorname{Ker}}(\Lmarkov+\lambda_k\Idmarkov).
        \end{align}
        \item Under all three assumptions \textit{(a)},\textit{(b)} and  \textit{(c)}, the eigenspaces of $\Lmarkov$ are \emph{hypercontractive} (see \cite{Bak1994} for sufficient conditions), that is, for all integer $q\geq0$,
        \begin{align}\label{ineq:mapping}
            \underset{k=0}{\overset{q}{\bigoplus}} \, \boldsymbol{\operatorname{Ker}}(\Lmarkov+\lambda_k\Idmarkov) \subseteq \underset{p\geq 1}{\bigcap} \, L^p(\boldsymbol{\mu}).
        \end{align}
        Then, using the open mapping theorem, we can show that the embedding \eqref{ineq:mapping} is continuous, that is, for all integer $k\geq1$, there exist $M_{k,q}>0$ such that 
        \begin{align}\label{ineq:markov_hyperontractivity}
            \forall X \in  \underset{k=0}{\overset{q}{\bigoplus}}\, \boldsymbol{\operatorname{Ker}}(\Lmarkov+\lambda_k\Idmarkov), \quad \E[X^{2k}] \leq M_{k,q}\, \E[X^2]^{k}.
        \end{align}
    \end{enumerate}
In the end, we have the same spectral properties as Propositions \ref{prop:LGamma} and \ref{prop:hypercontractivity}, which allow us to bound the quantities in that context as in Subsection \ref{ss:laststep}.
\subsection{Adapted Theorem \ref{maintheorem} for Markov chaoses}
Let $p$ be a positive integer and let $X$ be an element of $\Kermarkov(\Lmarkov+\lambda_p\Idmarkov)$. Then Stein's method yields
\begin{eqnarray*}
    && \dTV(\boldsymbol{\mu}_X,\gamma) = \frac{1}{2}\underset{h \in \CC_1}{\sup} \, \left | \int_\R h(x) \, \boldsymbol{\mu}_X(\dint x) - \int_\R h(x) \, \gamma(\dint x) \right |\\
    && = \frac{1}{2}\underset{h \in \CC_1}{\sup} \, \left | \int_E \phi'_h(X)-X\phi_h(X) \, \dint \boldsymbol{\mu} \right | = \frac{1}{2\lambda_p}\underset{h \in \CC_1}{\sup} \, \left | \int_E \phi_h'(X) \bigl (\lambda_p-\Gamarkov(X) \bigr ) \, \dint \boldsymbol{\mu} \right | 
\end{eqnarray*}
where one uses the diffusive properties \eqref{eq:markov_diffusion} and \eqref{eq:markov_diffusion2} to get 
\begin{eqnarray*}
    && \int_E \phi'_h(X)-X\phi_h(X) \, \dint \boldsymbol{\mu}   = \int_E  \frac{1}{\lambda_p} \bigl ( \lambda_p \phi'_h(X)+\phi_h(X)\Lmarkov X \bigr ) \, \dint \boldsymbol{\mu} \\
    && = \int_E  \frac{1}{\lambda_p} \bigl ( \lambda_p \phi'_h(X)-\Gamarkov(\phi_h(X),X) \bigr ) \, \dint \boldsymbol{\mu} =\frac{1}{\lambda_p} \int_E   \phi'_h(X)\bigl (\lambda_p-\Gamarkov(X) \bigr ) \, \dint \boldsymbol{\mu}.
\end{eqnarray*}
From there, we can reprove all the previous results, with some modifications, based solely on the structure defined in Subsection \ref{ss:def_Markov_framework}. Indeed, the proofs in the sections above are based on two things: the real analysis for functions $\phi$ in $L^2(\gamma)$, which only depends on the target distribution $\gamma$, and another part being the structure on the Wiener space, which we replace by the structure defined in Subsection \ref{ss:def_Markov_framework}. Fix $m$ a positive integer. Namely, we have:

\begin{theorem}\label{maintheorem_markov}
    There exists $C^{(\lambda)}_{p,m}>0$ such that, for all $X$ eigenfunction of order $p$ with respect to $-\Lmarkov$ with $\E[X^2]=1$,
    \begin{align*}
        \dTV(\boldsymbol{\mu}_X,\boldsymbol{\gamma}^{(\lambda)}_{X,m}) \leq C^{(\lambda)}_{p,m} \bigl ( \Var \Gamarkov(X) \bigr)^{\frac{m+1}{2}},
    \end{align*}
    where $\boldsymbol{\gamma}^{(\lambda)}_{X,m} \coloneq \gaus_{X,\tilde{r}_{\lambda,m}}$, with $\tilde{r}_{\lambda,m}$ defined in Subsection \ref{ss:markov_extension}.
\end{theorem}

Furthermore, we are still able to prove the optimality of the bounds from Theorem \ref{maintheorem_markov} in the case $m=1$, see  Appendix \ref{ss:optimal} for a proof.
\begin{proposition}\label{OptimalMarkov}
    Then there exists $c^{(\lambda)}_p >0$ such that for all eigenfunction $X$ of order $p$ with respect to $-\Lmarkov$ such that $\E[X^2]=1$, 
    \begin{align*}
        c^{(\lambda)}_p\Var\Gamarkov(X) & \leq \dTV(\boldsymbol{\mu}_{X},\boldsymbol{\gamma}^{(\lambda)}_{X,1})) \leq C^{(\lambda)}_{p,1}\Var\Gamarkov(X).
    \end{align*}
\end{proposition}

\subsection{The adapted integration by parts formula}
In this subsection, we adapt Lemma \ref{dev:polgen} in the framework described in Subsection \ref{ss:def_Markov_framework}. Given Assumptions \ref{hyp:markov}, the algebraic computations are completely analogous. We point out where it differs.
Let $m$ be a positive integer. Denote $d_{\lambda,m} \coloneq (2p-1)m$ and, for every $d \in \N$, denote $\EE{d} \coloneq \underset{k = 0}{\overset{d}{\bigoplus}} \Kermarkov(\boldsymbol{\L}+\lambda_k\boldsymbol{\operatorname{Id}})$.

We start with the revisited \eqref{eq:trick1}, which only follows from the diffusive property \eqref{eq:markov_diffusion}. 

\begin{lemma}\label{prop:markov_formula1}
    Let $X$ be an element of $\Kermarkov ( \Lmarkov +\lambda_p\Idmarkov)$. Then, for all $\phi \in \Cpol{2}$,
    \begin{align}\label{eq:markov_formula1}
        (\LL \phi)(X) & = -\frac{1}{\lambda_p}\,\phi''(X)\,(\Gamarkov(X)-\lambda_p)+ \frac{1}{\lambda_p}\Lmarkov \phi(X) 
    \end{align}
\end{lemma}
\begin{proof}
    We use \eqref{eq:markov_diffusion} to write
    \begin{align}\label{aux:markov_formula1}
        \Lmarkov \phi(X) & = \phi''(X)\Gamarkov(X)+\phi'(X)\Lmarkov X = \phi''(X)(\Gamarkov(X)-\lambda_p)+\lambda_p\phi''(X)-\lambda_p X\phi'(X) \nonumber \\
        & = \phi''(X)(\Gamarkov(X)-\lambda_p) + \lambda_p (\LL \phi)(X).
    \end{align}
    Isolating $(\LL \phi)(X)$ in \eqref{aux:markov_formula1} leads to the result.
\end{proof}

Then, we have the linear algebra lemma, as before.

\begin{lemma}\label{prop:polynome_annulateurMarkov}
Denote by $P_{\lambda, m}$ the polynomial $Z(Z+\lambda_1)\cdots (Z+\lambda_{d_{\lambda,m}})$. Then for all $W \in \EE{d_{\lambda,m}}$, one has $P_{\lambda,m}(\Lmarkov)(W) = 0$.
\end{lemma}

We will write $P_{\lambda, m}(Z)\coloneq \sum_{\alpha=0}^{d_{\lambda,m}+1} a_{\lambda, m, \alpha}Z^\alpha$ the expansion of $P_{\lambda,m}$ into the canonical basis of $\R[Z]$. Since $P_{\lambda,m}(0) = 0$, one has $a_{\lambda,m,0}=0$. Furthermore, for every $j \in \llbracket 0, d_{\lambda,m} \rrbracket$, we denote $P_{\lambda,m,j} \coloneq  \underset{\alpha=j+1}{\overset{d_{\lambda,m}+1}{\sum}} a_{\lambda,m,\alpha}Z^{\alpha-j-1}$. It follows the integration by parts formula. 

\begin{lemma}\label{markov:dev:polgen}
    Let $\phi$ be an element of $\Cpol{2d_{\lambda,m}+2}$ and let $W$ be an element of $\EE{d_{\lambda,m}}$. Then one has
\begin{align}\label{markov:eq:dev:polgen}
   \E[P_{\lambda, m}(\lambda_p\LL)(\phi)(X)W] & = -\underset{j=0}{\overset{d_{\lambda,m}}{\sum}}\, \E \bigl [\bigl (P_{\lambda,m,j}(\lambda_p\LL)(\phi) \bigr )''(X)(\Gamarkov(X)-\lambda_p)\Lmarkov^jW \bigr ]
\end{align}
where $P_{\lambda,m,j} := \underset{\alpha=j+1}{\overset{d_{\lambda,m}+1}{\sum}} a_{\lambda,m,\alpha}Z^{\alpha-j-1}$. 
\end{lemma}
Denote by $r_{\lambda,m}$ the integer part of the quotient $\lambda_{d_{\lambda,m}}/\lambda_p$.
\begin{proposition}\label{prop:markovIPP_W}
    Let $W$ be an element of $\EE{d_{\lambda,m}}$ and let $\varphi$ be an element of $\Cpol{2d_{\lambda,m}+2}$ with Hermite rank $r\geq r_{\lambda,m}+1$. Then
    \begin{align}\label{eq:markovIPP_W}
   \E \bigl [\varphi(X)W \bigr ] & =  \underset{j=0}{\overset{d_{\lambda,m}}{\sum}} \, \E \bigl [ \Tj^{(\lambda)}_{m,j}(\varphi)(X)(\Gamarkov(X)-\lambda_p)\Lmarkov^j W \bigr ]
\end{align}
where $\Tj^{(\lambda)}_{m,j} \varphi:= -P_{\lambda,m,j}\bigl ( \lambda_p(\LL-2) \bigr )P_{\lambda,m} \bigl ( \lambda_p(\LL-2) \bigr )^{-1} \varphi''$.
\end{proposition}

\begin{proof}
\noindent \underline{Step 1}: Since the computations are the same, using Lemma \ref{prop:polynome_annulateurMarkov} and replacing $p$ by $\lambda_p$ in the context of the proof of Lemma \ref{dev:polgen}, we get
\begin{align}\label{aux:markovIPP_W}
   \E[P_{\lambda, m}(\lambda_p\LL)(\phi)(X)W] & = -\underset{j=0}{\overset{d_{\lambda,m}}{\sum}}\, \E \bigl [\bigl (P_{\lambda,m,j}(\lambda_p\LL)(\phi) \bigr )''(X)(\Gamarkov(X)-\lambda_p)\Lmarkov^jW \bigr ],
\end{align}
for all $\phi \in \Cpol{2d_{\lambda,m}+2}$.

\medskip 
\noindent \underline{Step 2}: Now, as done in the proof of Lemma \ref{form:W}, we proceed to invert $P_{\lambda, m}(\lambda_p\LL)$. First, let us observe that
\begin{align*}
    P_{\lambda, m}(\lambda_pZ) & = \lambda_p^{d_{\lambda,m}+1} \underset{k=0}{\overset{d_{\lambda,m}}{\prod}} \left ( Z+\frac{\lambda_k}{\lambda_p} \right ),
\end{align*}
where, for all $k \in \llbracket 0, d_{\lambda,m} \rrbracket$, $\frac{\lambda_k}{\lambda_p} \leq \frac{\lambda_{d_{\lambda,m}}}{\lambda_p}$. As in the proof of Lemma \ref{prop:invertibility}, we deduce that $ P_{\lambda, m}(\lambda_p\LL)^{-1}\varphi$ is well-defined, since $\varphi$ has Hermite rank $r>r_{\lambda,m}$. 

\medskip 
\noindent \underline{Step 3}:  We apply \eqref{aux:markovIPP_W} to $\phi \coloneq P_{\lambda, m}(\lambda_p\LL)^{-1}\varphi$ to get 
\begin{align*}
   \E[\varphi(X)W] & = -\underset{j=0}{\overset{d_{\lambda,m}}{\sum}}\, \E \bigl [\bigl (P_{\lambda,m,j}(\lambda_p\LL)P_{\lambda, m}(\lambda_p\LL)^{-1}(\varphi) \bigr )''(X)(\Gamarkov(X)-\lambda_p)\Lmarkov^jW \bigr ].
\end{align*}
Similarly as in the Step 2 of the proof of Proposition \ref{rel1}, the use of the commutation properties \eqref{eq:commutationLOU} and \eqref{eq:resolvent_commutation} leads to the result.
\end{proof}
The operators $\bigl ( \Tj_{m,j}^{(\lambda)} \bigr )$ have the same form and enjoy analogous continuity and stability properties as the operators $(\Tj_{j,m})$ in Section \ref{s:higherorder}. Namely, one has the following. 

\begin{lemma}\label{prop:markov_rewrite_of_Tj}
    Let $j$ be an element of $\llbracket 0, d_{\lambda,m} \rrbracket$. For all $\varphi \in \Cpol{2d_{\lambda,m}+2}$ with Hermite rank $r \geq r_{\lambda,m}+1$, one has 
    \begin{align}\label{eq:markov_rewrite_of_Tj}
        \Tj^{(\lambda)}_{m,j}(\varphi) & = \Q^{(\lambda)}_{m,j}(\LL) \S(\varphi)
    \end{align}
    where 
    \begin{align*}
   \Q^{(\lambda)}_{m,j}(Z) & \coloneq \lambda_p^{-j-1}(Z-2)^{-j} - \underset{\alpha=1}{\overset{j}{\sum}} \, a_{\lambda, m,\alpha} \lambda_p^{\alpha-j-1} (Z-2)^{\alpha-j}P_{\lambda,m} \bigl ( \lambda_p(Z-2) \bigl )^{-1}.
\end{align*}
\end{lemma}
\begin{proposition}
     Let $k$ be a positive integer and let $j_k$ be an element of $\llbracket 0,d_{\lambda,k} \rrbracket$. Then there exist $B_{k,j_k} >0$ such that for every $\phi$ element of $\Cpol{0}$, with Hermite rank $r \geq r_{\lambda,k}+1$, verifying 
    \[\forall x \in \R, \quad |\phi(x)|\leq C(1+|x|^d),\]
    where $C>0$ and $d\in \N$, one has 
    \begin{align*}
        \forall x \in \R, \quad |\Tj^{(\lambda)}_{k,j_k}(\phi)(x)| \leq CB_{k,j_k}(1+|x|^{d+1+\beta_{\lambda,k}}),
    \end{align*}
    where $\beta_{\lambda,k} \coloneq 1 + \sum^{d_{\lambda,k}}_{i=p+1} \max \bigl (0,\lceil \lambda_i/\lambda_p-2 \rceil \bigr )$.
\end{proposition}

As done in Proposition \ref{form:Hrankcont}, one can then extend \eqref{eq:markovIPP_W} to continuous functions $\varphi$ with Hermite rank $r \geq r_{\lambda,m}+1$.
\subsection{Extension, general expansion and bounds on the remainder}\label{ss:markov_extension}

As before, we proceed to apply several times \eqref{eq:markovIPP_W} in order to get an expansion of $\E[\varphi(X)(\Gamarkov(X)-\lambda_p)]$. We first set some notations to make the exposition of the expansion simpler.
Denote by $\operatorname{I}_m$ the set $\llbracket 0, d_{\lambda,1} \rrbracket \times \cdots \times  \llbracket 0, d_{\lambda,m} \rrbracket$, and for $\j = (j_1, \ldots, j_m) \in \operatorname{I}_m$, denote by $W_{\j}=W_{j_1,\ldots,j_m}$ the random variable \[(\Gamarkov(X)-\lambda_p)\, \Lmarkov^{j_m}(\Gamarkov(X)-\lambda_p) \cdots \Lmarkov^{j_1}(\Gamarkov(X)-\lambda_p).\]

\begin{proposition}\label{prop:markov_Wfinal}
    Let $\varphi$ be an element of $\Cpol{0}$ with Hermite rank $r \geq r_{\lambda,m}+2m-1$. Then one has
    \begin{align}\label{eq:markov_Wfinal}
        \E[\varphi(X)(\Gamarkov(X)-\lambda_p)] & = \underset{\mathbf{j} \in \operatorname{I}_m}{\sum} \, \E \bigl [ \boldsymbol{\Tj}^{(\lambda, m)}_{\mathbf{j}}(\varphi)(X)W_{\mathbf{j}} \bigr ],
    \end{align}
    where $\boldsymbol{\Tj}_{\mathbf{j}}^{(\lambda, m)}:=  \Tj^{(\lambda)}_{m,j_m} \cdots \Tj^{(\lambda)}_{1,j_1}$.
\end{proposition}
\begin{proof}
    The proof stays the same as the proof of Proposition \ref{form:Wfinal}, with the adapted changes, that is the lower bound on the Hermite rank of $\varphi$. The operators $(\Tj_{k,j_k}^{(\lambda)})$ still send functions with Hermite rank $r$ to functions with Hermite rank $r-2$. Imposing $r \geq r_{\lambda,m}+2m-1$ allows us to apply repetitively \eqref{eq:markovIPP_W} and get \eqref{eq:markov_Wfinal}.
\end{proof}
\begin{remark}
    Note that, if we suppose the spectrum $\{\lambda_k\}_{k}$ to be sub-additive, that is $\lambda_{k+n} \leq \lambda_k+\lambda_n$ for all integers $k,n \geq 0$, then one has $\lambda_{(2p-1)m} \leq 2m\lambda_p$, and $r_{\lambda,m} < 2m$, i.e. $r_{\lambda,m} \leq 2m-1$.
\end{remark}
Then, one has the following estimate.
\begin{proposition}
    Let $\mathbf{j}$ be an element of $\operatorname{I}_m$. There exists $B^{(\lambda)}_{\mathbf{j},m}>0$ such that, for every bounded continuous function $\phi$, 
    \begin{align*}
        \forall x \in \R, \quad | \boldsymbol{\Tj}^{(\lambda,m)}_{\j} \bigl ( \projgeq{\tilde{r}_{\lambda,m}}(\phi) \bigr )(x) | \leq B^{(\lambda)}_{\mathbf{j},m}(1+|x|^{l_{\lambda,m}})\|\phi\|_\infty,
    \end{align*}
    where $l_{\lambda,m} \coloneq m + \sum_{k=1}^{m} \beta_{\lambda,k}$ and $\tilde{r}_{\lambda,m} \coloneq r_{\lambda,m}+2m-1$.
\end{proposition}

Given Assumptions \ref{hyp:markov} and the results above, we can now replicate the bounding done in Subsection \ref{ss:laststep} with the same arguments, and prove Theorem \ref{maintheorem_markov}.

\section*{Acknowledgments}
 Paul Mansanarez is supported by PhD FRIA grant from the FNRS (Fonds National pour la Recherche Scientifique), Communaut\'e fran{\c c}aise de Belgique. Yvik Swan is supported in part by FNRS grant J.0200.24F. This work is supported by the ANR grant LESSBIG.

\begin{appendix}
\section{Proof of Propositions \ref{prop:optimalm1} and \ref{OptimalMarkov}}\label{ss:optimal}

In this subsection, we prove the optimality of the bound from Theorem \ref{maintheorem} in the case $m=1$. The argument is adapted from \cite{NP2015}. Since we use only the algebraic framework from Subsection \ref{ss:def_Markov_framework}, the same proof (with the obvious changes, such as replacing $p$ by $\lambda_p$) also yields Proposition \ref{OptimalMarkov}. We begin with two lemmas.

\begin{lemma}\label{aux:optimal1}
    Let $F \in \WW_p$ satisfy $\E[F]=0$ and $\E[F^2]=1$. Then
    \[
        \left( \frac{9}{80\sqrt{2}\,p^2} - C_{p,2}\bigl( \Var \Ga(F) \bigr)^{1/2} \right)\Var \Ga(F)
        \;\leq\; \dTV(\prob_F,\boldsymbol{\gamma}_{F,1}).
    \]
\end{lemma}

\begin{proof}
    Let $h(x)\coloneq \frac{2}{5}\!\left(x^2+\frac{5}{2}\right)e^{-x^2/2}$. Then $h\in \mathscr{C}_b^0(\R)$, $\|h\|_\infty=1$, and
    \[
      \langle h,H_4\rangle_{L^2(\gamma)}=\frac{3}{10\sqrt{2}},\qquad
      \langle h,H_k\rangle_{L^2(\gamma)}=0\quad\text{for }k\in\{3,5,6,7\},
    \]
    where $\langle f,g\rangle_{L^2(\gamma)}$ denotes the $L^2(\gamma)$ inner product. By Theorem \ref{maintheorem} with $m=2$,
    \[
        \left| \E\!\left[h(F)-h(N)-\frac{1}{80\sqrt{2}}\kappa_4(F) \right] \right|
        \leq C_{p,2}\bigl(\Var \Ga(F)\bigr)^{3/2}.
    \]
    Indeed,
    \begin{align*}
         \E[h(N)f_2(N)]
        & = \E[h(N)] + \sum_{k=3}^{6}\frac{\langle h,H_k\rangle_{L^2(\gamma)}}{k!}\,\E[H_k(F)] \\
        & = \E[h(N)] + \frac{\langle h,H_4\rangle_{L^2(\gamma)}}{24}\,\E[H_4(F)]
        = \E[h(N)] + \frac{1}{80\sqrt{2}}\kappa_4(F),
    \end{align*}
    since, under $\E[F]=0$ and $\E[F^2]=1$, one has $\E[H_4(F)]=\kappa_4(F)$. Hence,
    \[
        \frac{2}{80\sqrt{2}}\bigl|\kappa_4(F)\bigr| - C_{p,2}\bigl(\Var \Ga(F)\bigr)^{3/2}
        \;\leq\; \bigl| \E[h(F)-h(N)] \bigr| \;\leq\; \dTV(\prob_F,\boldsymbol{\gamma}_{F,1}).
    \]
    Using the standard inequality (see, e.g., \cite{Led2012})
    \begin{equation}\label{gamma_cum}
        \frac{3}{p^2}\Var \Ga(F) \;\leq\; \kappa_4(F),
    \end{equation}
    we deduce
    \[
        \frac{9}{80\sqrt{2}\,p^2}\Var \Ga(F) - C_{p,2}\bigl(\Var \Ga(F)\bigr)^{3/2}
        \;\leq\; \dTV(\prob_F,\boldsymbol{\gamma}_{F,1}),
    \]
    which yields the claim.
\end{proof}

\begin{remark}\label{existence_h}
 In the proof of Lemma \ref{aux:optimal1}, we used an explicit $h$ lying in
    \[
      \operatorname{Ker} \phi_3 \,\cap\, \operatorname{Ker} \phi_5 \,\cap\, \operatorname{Ker} \phi_6 \,\cap\, \operatorname{Ker} \phi_7
      \,\cap\, \bigl(\mathscr{C}_b^0(\R)\setminus \operatorname{Ker} \phi_4\bigr),
    \]
    where $\phi_i:f \mapsto \langle f,H_i \rangle_{L^2(\gamma)}$ is a continuous linear functional on $\mathscr{C}_b^0(\R)$. Since $(\phi_i)_i$ is linearly independent, the intersection is nonempty, which guarantees the existence of such an $h$. This argument extends to other sets of constraints on the coefficients.
\end{remark}
\begin{remark}\label{general_gamma_cum}
    Let us note that, in setting of the Wiener space, there exists a better constant than $3/p^2$ (e.g. in \cite[Lemma 5.2.4]{NP2012}) for inequality \eqref{gamma_cum}. However, the latter inequality holds in the setting of Markov chaoses (see e.g. \cite{Led2012}).
\end{remark}

\begin{lemma}\label{prop:dTVminVar}
    There exists $C>0$ such that, for all $F \in \WW_p$ with $\E[F^2]=1$,
    \begin{equation}\label{eq:dTVminVar}
        \frac{3}{2p^2 C^{4/5}} \Var \Ga(F) \;\leq\; \dTV(\prob_F,\boldsymbol{\gamma}_{F,1})^{1/5}.
    \end{equation}
\end{lemma}

\begin{proof}
    Since $\int_\R x^4 \,\gamma(\dint x)=3$ and $\int_\R x^4 H_3(x)\,\gamma(\dint x)=0$, we have
    \[
      \E[F^4]-3
      = \E[F^4] - \int_\R x^4 \,\gamma(\dint x) - \frac{\E[H_3(F)]}{6}\int_\R x^4 H_3(x)\,\gamma(\dint x)
      = \int_\R x^4\,(\prob_F-\boldsymbol{\gamma}_{F,1})(\dint x).
    \]
    Fix $M>0$. Then
    \[
      \E[F^4]-3
      = \int_{-M}^{M} x^4\,(\prob_F-\boldsymbol{\gamma}_{F,1})(\dint x)
        + \int_{|x|>M} x^4\,(\prob_F-\boldsymbol{\gamma}_{F,1})(\dint x).
    \]
    On the one hand,
    \[
      \left|\int_{-M}^{M} x^4\,(\prob_F-\boldsymbol{\gamma}_{F,1})(\dint x)\right|
      \le M^4\,\dTV(\prob_F,\boldsymbol{\gamma}_{F,1}).
    \]
    On the other hand, by Markov's inequality and hypercontractivity \eqref{eq:hypercontractivity},
    \begin{eqnarray*}
      && \left|\int_{|x|>M} x^4\,(\prob_F-\boldsymbol{\gamma}_{F,1})(\dint x)\right|
      \\
      &&\le \E[F^8]^{1/2}\,\prob(|F|>M)^{1/2} + \E[N^8]^{1/2}\,\prob(|N|>M)^{1/2}
      \le \frac{C}{M}.
    \end{eqnarray*}
    Hence,
    \begin{equation}\label{eq:cum4M}
        \E[F^4]-3 \;\le\; M^4\,\dTV(\prob_F,\boldsymbol{\gamma}_{F,1}) + \frac{C}{M}.
    \end{equation}
    Optimizing the right-hand side in $M$ yields
    \[
      \E[F^4]-3 \;\le\; 2\,C^{4/5}\,\dTV(\prob_F,\boldsymbol{\gamma}_{F,1})^{1/5}.
    \]
    Finally, by the standard bound $\kappa_4(F)\ge \frac{3}{p^2}\Var\Ga(F)$ and the identity $\kappa_4(F)=\E[F^4]-3$ for centered, variance-one $F\in\WW_p$,
    \[
      \frac{3}{p^2}\Var\Ga(F)
      \;\le\; \E[F^4]-3
      \;\le\; 2\,C^{4/5}\,\dTV(\prob_F,\boldsymbol{\gamma}_{F,1})^{1/5},
    \]
    which is \eqref{eq:dTVminVar}.
\end{proof}

\begin{proof}[Proof of Proposition \ref{prop:optimalm1}]
    Set $\alpha \coloneq \dfrac{9}{160\sqrt{2}\,p^2\,C_{p,2}}$. If $\Var\Ga(F) \le \alpha^2$, then
    \[
      \frac{9}{160\sqrt{2}\,p^2}
      \;\le\; \frac{9}{80\sqrt{2}\,p^2} - C_{p,2}\bigl(\Var\Ga(F)\bigr)^{1/2},
    \]
    and Lemma \ref{aux:optimal1} gives
    \[
      \frac{9}{160\sqrt{2}\,p^2}\,\Var\Ga(F)
      \;\le\; \dTV(\prob_F,\boldsymbol{\gamma}_{F,1}).
    \]
    If instead $\Var\Ga(F) > \alpha^2$, then \eqref{eq:dTVminVar} yields
    \begin{equation}\label{eq:dTVminC}
        \frac{3^5\,\alpha^{10}}{32\,p^{10}\,C^{4}}
        \;\le\; \dTV(\prob_F,\boldsymbol{\gamma}_{F,1}).
    \end{equation}
    Moreover, by hypercontractivity \eqref{eq:hypercontractivity}, there exists $K_p>0$ such that
    \begin{equation}\label{eq:Varbound}
        \Var\Ga(F) \;\le\; K_p\,\E\!\bigl[|\Ga(F)-p|\bigr] \;\le\; 2pK_p,
    \end{equation}
    since $\Ga(F)\ge 0$ and $\E[\Ga(F)]=p\,\E[F^2]=p$. Combining \eqref{eq:dTVminC} with \eqref{eq:Varbound}, we obtain
    \[
      \frac{3^5\,\alpha^{10}}{64\,p^{11}\,C^{4}\,K_p}\,\Var\Ga(F)
      \;\le\; \dTV(\prob_F,\boldsymbol{\gamma}_{F,1}).
    \]
    Setting
    \[
      c_p \coloneq \min\!\left\{ \frac{3^5\,\alpha^{10}}{64\,p^{11}\,C^{4}\,K_p},\ \frac{9}{160\sqrt{2}\,p^2} \right\},
    \]
    we conclude that
    \[
      c_p\,\Var\Ga(F) \;\le\; \dTV(\prob_F,\boldsymbol{\gamma}_{F,1})
    \]
    for all $F\in\WW_p$ with $\E[F^2]=1$.
\end{proof}

Since the proof of Proposition \ref{prop:optimalm1} uses only the algebraic framework from Subsection \ref{ss:def_Markov_framework}, following the same arguments with the natural adaptations, such as replacing $p$ by $\lambda_p$, yield Proposition \ref{OptimalMarkov} (see also the discussions in Remarks \ref{general_gamma_cum} and \ref{existence_h}).
\section{Edgeworth expansion on general chaoses with matching moment condition}\label{ss:matching}
In this subsection, we fix a sequence $\mathbf{X} = (X_n)_{n}$ of i.i.d. random variables with $\E[X_1] = 0$ and $\E[X_1^2]=1$. We also fix a sequence $\mathbf{G} = (G_n)_n$ of i.i.d random variables with a standard Gaussian distribution, independent of $\mathbf{X}$. This is an adapted work from \cite[Chapter 11]{NP2012} and \cite{NPR2010}. 

\begin{definition}
    Let $1 \leq d \leq M$ be integers. A polynomial $Q \in \R[x_1,\ldots,x_M]$ is called a homogeneous polynomial of degree $d$ if it has the form
    \begin{align*}
        Q(x_1,\ldots, x_M) = \sum_{1\leq i_1 < \cdots < i_d\leq M} a(i_1, \ldots, i_d)\, x_{i_1} \cdots x_{i_d},
    \end{align*}
    where  $a(i_1, \ldots, i_d)$ is a real number.
\end{definition}

We next fix some notation for later use. For a homogeneous polynomial $Q \in \R[x_1,\ldots, x_M]$ of degree $d$ and $r \in \llbracket 1, d \rrbracket$, we write
\begin{align*}
    Q_{r}(x_1,\ldots, x_M) & \coloneq \sum_{\underset{i_1 \neq r, \ldots, i_d \neq r}{1\leq i_1 < \cdots < i_d\leq M}} a(i_1, \ldots, i_d)\, x_{i_1}\cdots x_{i_p} \; ,\\
    \widehat{Q}_{r}(x_1,\ldots, x_M) &\coloneq \sum_{\underset{r \in \{i_1,\ldots,i_d\}}{1\leq i_1 < \cdots < i_d\leq M}} a(i_1, \ldots, i_d)\, \prod_{\underset{i_j \neq r}{j=1}}^d x_{i_j}.
\end{align*}
In particular, $\widehat{Q}_r$ is a homogeneous polynomial of degree $d-1$, and 
\begin{align}\label{eq:swapping}
   \forall r \in \llbracket 1,M \rrbracket, \quad Q(x_1, \ldots, x_M) = x_r\widehat{Q}_{r}(x_1,\ldots, x_M) + Q_{r}(x_1,\ldots, x_M).
\end{align}
Finally, we will denote by $\mathbf{Q}$ the random variable $Q(X_1, \ldots, X_M)$, that is, the polynomial $Q$ evaluated in the sequence $\mathbf{X}$. Moreover, we will denote by $G(\mathbf{Q})$ its Gaussian counterpart, that is the random variable $Q(G_1, \ldots,G_M)$. Note that in this case $G(\mathbf{Q}) \in \WW_d$.

The space of random variables of such form with $Q$ a homogeneous polynomial of degree $d$ will be denoted by $\mathcal{P}_d(\mathbf{X},M)$ and we will denote by $\mathcal{P}_d(\mathbf{X})$ the union $\underset{M \geq d}{\bigcup} \mathcal{P}_d(\mathbf{X},M)$.

An important quantity for central convergence of such polynomial is the following.

\begin{definition}
    Let $\mathbf{Q}$ be an element of $\mathcal{P}_d(\mathbf{X},M)$ for some positive integer $M\geq d$, and let $r$ be an element of $\llbracket 1,M\rrbracket$. The \emph{influence} of the variable $r$ is the quantity
    \begin{align}
        \tau_r(\mathbf{Q}) \coloneq \sum_{\underset{r \in \{i_1, \ldots, i_d\}}{1\leq i_1 <\cdots <i_d \leq M}} a(i_1, \ldots,i_d)^2.
    \end{align}
    We will denote by $\tau(\mathbf{Q})$ the total influence, that is $\underset{1\leq r\leq M}{\max} \tau_r(\mathbf{Q})$.
\end{definition}

The next result is a property of hypercontractivity on $\mathcal{P}_p(\mathbf{X})$, the vector space spanned by the homogeneous polynomials of degree $d \leq p$ evaluated in the sequence $\mathbf{X}$. We refer to \cite[Subsection 1.3.2]{HMP2025} and the references therein \cite[Propositions 3.16, 3.11 \& 3.12]{MOO2010} for a proof.

\begin{proposition}\label{prop:generalchaos_hypercontractivity}
    The $L^p$-norms are equivalent on the space $\mathcal{P}_p(\mathbf{X})$, that is for all $1\leq r<q <+\infty$, there exists $D_{r,q,p}>0$ such that
    \begin{align}
        \forall \mathbf{W} \in \mathcal{P}_p(\mathbf{X}), \quad \|\mathbf{W}\|_r \leq D_{r,q,p}\|\mathbf{W}\|_q .
    \end{align}
\end{proposition}

\begin{definition}
    The distance $\operatorname{d}_{k}$ is the distance $\operatorname{d}_{\HH_{k}}$ defined in \eqref{eq:def_dHH}, where $\HH_k$ is the set of $k$-differentiable functions $h:\R \rightarrow \R$ such that $\forall l \in \llbracket 0,k \rrbracket, \; \|h^{(l)}\|_{\infty} \leq 1$.
\end{definition}

\begin{proposition}\label{prop:edgeworth_generalchaos}
 Let $m$ be a positive integer. Suppose that $\E\bigl [|X_1|^{k} \bigr ]<+\infty$ where $k>2$ is an integer and that $X_1$ has the same first $k-1$ moments as $G_1$. Then, there exists $C_{X_1,k,p}>0$ such that, for all $\mathbf{Q} \in \mathcal{P}_p(\mathbf{X})$ with $\E[\mathbf{Q}]=0$ and $\E[\mathbf{Q}^2]=1$,
    \begin{align}
        \operatorname{d}_{k}(\mathbf{Q}, \boldsymbol{\gamma}_{\mathbf{F},m}) \leq C_{X_1,k,p}\,\tau(\mathbf{Q})^{\frac{k-2}{2}} + C_{p,m} \bigl (\Var \Gamma(\mathbf{F})  \bigr )^{\frac{m+1}{2}},
    \end{align}
    where $\mathbf{F} \coloneq G(\mathbf{Q})$.
\end{proposition}
\begin{proof}
    Write $\mathbf{Q} = \mathbf{Q}_M = Q_p(X_1, \ldots, X_M)$ and $\mathbf{F} = G(\mathbf{Q})$, where 
    \begin{align*}
        Q_p= \sum_{1\leq i_1 < \cdots < i_p\leq M} a(i_1, \ldots, i_p)\, x_{i_1} \cdots x_{i_p}.
    \end{align*}
    Let $h$ be an element of $\HH_{k}$. Since $\mathbf{F} \in \WW_p$ and $\E[\mathbf{F}^2] = \E[\mathbf{Q}^2] = 1$, by Theorem \ref{maintheorem}, one has
    \begin{align*}
       | h(\mathbf{Q})-h(N)g_m(N)| &= |h(\mathbf{Q})-h(\mathbf{F})+h(\mathbf{F})-h(N)g_m(N)| \\
       & \leq |h(\mathbf{Q})-h(\mathbf{F})|+|h(\mathbf{F})-h(N)g_m(N)| \\
       & \leq |h(\mathbf{Q})-h(\mathbf{F})| + \dTV(\mathbf{F}, \boldsymbol{\gamma}_{\mathbf{F},m}) \\
       & \leq |h(\mathbf{Q})-h(\mathbf{F})| + C_{p,m} \bigl (\Var \Gamma(\mathbf{F})  \bigr )^{\frac{m+1}{2}},
    \end{align*}
    where $C_{p,m}$ is the constant given by Theorem \ref{maintheorem}. We use the Lindeberg invariance principle. One has
    \begin{align*}
        & h(\mathbf{Q})-h(\mathbf{F})\\
        &= \sum_{r=1}^M \Bigl\{h\bigl (Q_p(X_1,\ldots,X_r,G_{r+1},\dots,G_M) \bigr )-h\bigl (Q_p(X_1,\ldots,X_{r-1},G_{r},\dots,G_M) \bigr )\Bigr\}.
    \end{align*}
Denoting
\begin{align*}
    \mathbf{Q}_{M,r}(\mathbf{X},\mathbf{G})
    &\coloneq Q_{p,r}(X_1,\ldots,X_r,G_{r+1}, \ldots,G_M) = Q_{p,r}(X_1,\ldots,X_{r-1},G_{r}, \ldots,G_M),
\end{align*}
and 
\begin{align*}
    \widehat{\mathbf{Q}}_{M,r}(\mathbf{X},\mathbf{G})
    &\coloneq \widehat{Q}_{p,r}(X_1,\ldots,X_r,G_{r+1}, \ldots,G_M) =\widehat{Q}_{p,r}(X_1,\ldots,X_{r-1},G_{r}, \ldots,G_M),
\end{align*}
we use \eqref{eq:swapping} to write
\begin{align*}
    Q_p(X_1,\ldots,X_r,G_{r+1},\dots,G_M) = X_r \widehat{\mathbf{Q}}_{M,r}(\mathbf{X},\mathbf{G})+\mathbf{Q}_{M,r}(\mathbf{X},\mathbf{G}),
\end{align*}
for all $r \in \llbracket 1, M\rrbracket$. Then, using a Taylor expansion, there exists $c_{X,r} \in \R$ such that
\begin{align}\label{eq:taylorX}
    &  h \bigl (Q_p(X_1,\ldots,X_r,G_{r+1},\dots,G_M) \bigr ) \\
     &=  h\bigl ( X_r \widehat{\mathbf{Q}}_{M,r}(\mathbf{X},\mathbf{G})+\mathbf{Q}_{M,r}(\mathbf{X},\mathbf{G})\bigr ) \nonumber \\
     &=  \sum_{j=0}^{k-1} \frac{X_r^j}{j!}\widehat{\mathbf{Q}}_{M,r}(\mathbf{X},\mathbf{G})^j\,h^{(j)} \bigl (\mathbf{Q}_{M,r}(\mathbf{X},\mathbf{G}) \bigr ) + \frac{h^{(k)}(c_{X,r})}{k!}X_r^{k}.
\end{align}
Similarly, there exists $c_{G,r} \in \R$ such that
\begin{align}\label{eq:taylorG}
   &   h \bigl (Q_p(X_1,\ldots,X_{r-1},G_{r},\dots,G_M) \bigr ) \\ 
     &=  h\bigl ( G_r \widehat{\mathbf{Q}}_{M,r}(\mathbf{X},\mathbf{G})+\mathbf{Q}_{M,r}(\mathbf{X},\mathbf{G})\bigr ) \nonumber \\
     &=  \sum_{j=0}^{k-1} \frac{G_r^j}{j!}\widehat{\mathbf{Q}}_{M,r}(\mathbf{X},\mathbf{G})^j\,h^{(j)} \bigl (\mathbf{Q}_{M,r}(\mathbf{X},\mathbf{G}) \bigr ) + \frac{h^{(k)}(c_{G,r})}{k!}G_r^{k}.
\end{align}
Combining \eqref{eq:taylorX} and \eqref{eq:taylorG} together, one gets
\begin{align}\label{eq:taylorC}
     & h\bigl (Q_p(X_1,\ldots,X_r,G_{r+1},\dots,G_M) \bigr )-h\bigl (Q_p(X_1,\ldots,X_{r-1},G_{r},\dots,G_M) \bigr ) \nonumber\\
     &\quad = \sum_{j=0}^{k-1} \frac{X_r^j-G_r^j}{j!}\widehat{\mathbf{Q}}_{M,r}(\mathbf{X},\mathbf{G})^j\,h^{(j)} \bigl (\mathbf{Q}_{M,r}(\mathbf{X},\mathbf{G}) \bigr )
      +\frac{h^{(k)}(c_{X,r})}{k!}X_r^{k} - \frac{h^{(k)}(c_{G,r})}{k!}G_r^{k}.
\end{align}
Taking the expectation of \eqref{eq:taylorC} and using the independence between the elements of $\mathbf{X}$ and $\mathbf{G}$, one gets
\begin{align*}
     & \E \Bigl[ h\bigl (Q_p(X_1,\ldots,X_r,G_{r+1},\dots,G_M) \bigr )-h\bigl (Q_p(X_1,\ldots,X_{r-1},G_{r},\dots,G_M) \bigr ) \Bigr]\\
     &\quad = \sum_{j=0}^{k-1} \frac{\E[X_r^j-G_r^j]}{j!}\E \bigl [ \widehat{\mathbf{Q}}_{M,r}(\mathbf{X},\mathbf{G})^j\,h^{(j)} \bigl (\mathbf{Q}_{M,r}(\mathbf{X},\mathbf{G}) \bigr ) \bigr ]\\
     &\qquad +\E \left [\frac{h^{(k)}(c_{X,r})}{k!}X_r^{k}\widehat{\mathbf{Q}}_{M,r}(\mathbf{X},\mathbf{G})^{k} - \frac{h^{(k)}(c_{G,r})}{k!}G_r^{k}\widehat{\mathbf{Q}}_{M,r}(\mathbf{X},\mathbf{G})^{k} \right ] \\
     &\quad = \E \left [ \left (\frac{h^{(k)}(c_{X,r})}{k!}X_r^{k} - \frac{h^{(k)}(c_{G,r})}{k!}G_r^{k} \right ) \widehat{\mathbf{Q}}_{M,r}(\mathbf{X},\mathbf{G})^{k}\right ],
\end{align*}
since the moments of $X_r$ and $G_r$ are equal up to order $k-1$.
Then, 
\begin{align*}
     & \bigl | \E \bigl [ h\bigl (Q_p(X_1,\ldots,X_r,G_{r+1},\dots,G_M) \bigr )-h\bigl (Q_p(X_1,\ldots,X_{r-1},G_{r},\dots,G_M) \bigr ) \bigr ] \bigr |\\
     &\quad \leq   \E \Biggl [ \left |\frac{h^{(k)}(c_{X,r})}{k!}X_r^{k} - \frac{h^{(k)}(c_{G,r})}{k!}G_r^{k} \right | \bigl |\widehat{\mathbf{Q}}_{M,r}(\mathbf{X},\mathbf{G})^{k}\bigr |\Biggr ]\\
     &\quad \leq \frac{2\|h^{(k)}\|_{\infty}}{k!} \bigl ( \E[X_1^{k}]+\E[G_1^{k}] \bigr )\E \Bigl [ \bigl |\widehat{\mathbf{Q}}_{M,r}(\mathbf{X},\mathbf{G})\bigr |^{k} \Bigr ].
\end{align*}
Since $\widehat{\mathbf{Q}}_{M,r}(\mathbf{X},\mathbf{G})$ belongs to $\mathcal{P}_{p-1}(\mathbf{Y})$, where $\mathbf{Y} = (X_1, \ldots,X_{r-1},G_{r+1},G_{r+2}, \ldots)$ is a sequence of independent random variables, by Proposition \ref{prop:generalchaos_hypercontractivity}, there exists $D_{k,p} >0$ such that
\begin{align*}
    \E \Bigl [ \bigl |\widehat{\mathbf{Q}}_{M,r}(\mathbf{X},\mathbf{G})\bigr |^{k} \Bigr ] = \bigl \|\widehat{\mathbf{Q}}_{M,r}(\mathbf{X},\mathbf{G}) \bigr \|_{k}^{k} \leq D_{k,p}\,\|\widehat{\mathbf{Q}}_{M,r}(\mathbf{X},\mathbf{G}) \bigr \|_{2}^{k}.
\end{align*}
Since $\|\widehat{\mathbf{Q}}_{M,r}(\mathbf{X},\mathbf{G}) \bigr \|_{2}=\tau_r(\mathbf{Q})^{1/2}$, one gets
\begin{align*}
    &  |h(\mathbf{Q})-h(\mathbf{F}) | \\
     &\leq  \sum_{r=1}^M \bigl |h\bigl (Q_p(X_1,\ldots,X_r,G_{r+1},\dots,G_M) \bigr )-h\bigl (Q_p(X_1,\ldots,X_{r-1},G_{r},\dots,G_M) \bigr )\bigr | \\
     &\leq \frac{2\|h^{(k)}\|_{\infty}}{k!} \bigl ( \E[X_1^{k}]+\E[G_1^{k}] \bigr ) D_{k,p}\sum_{r=1}^M \tau_r(\mathbf{Q})^{\frac{k}{2}}.
\end{align*}
Furthermore, since $k>2$ and $\sum_{r=1}^M \tau_r(\mathbf{Q}) = \E[\mathbf{Q}^2]=1$, one has
\begin{align*}
   \sum_{r=1}^M \tau_r(\mathbf{Q})^{\frac{k}{2}} =  \sum_{r=1}^M\tau_r(\mathbf{Q})^{\frac{k-2}{2}}\tau_r(\mathbf{Q}) \leq \bigl (\underset{1\leq r \leq M}{\max} \tau_r(\mathbf{Q}) \bigr )^{\frac{k-2}{2}},
\end{align*}
hence 
\begin{align*}
     |h(\mathbf{Q})-h(\mathbf{F}) |
     \leq \frac{2\|h^{(k)}\|_{\infty}}{k!} \bigl ( \E[X_1^{k}]+\E[G_1^{k}] \bigr ) D_{k,p}\bigl (\underset{1\leq r \leq M}{\max} \tau_r(\mathbf{Q}) \bigr )^{\frac{k-2}{2}}.
\end{align*}
\end{proof}

\begin{remark}
    Suppose $X_1$ has the same first $4m-1$ moments as $G_1$. Then one has $\E[H_l(\mathbf{F})] = \E[H_l(\mathbf{Q})]$, for all $l \in \llbracket 3, 4m-1 \rrbracket$. It follows that the density of $\boldsymbol{\gamma}_{\mathbf{F},m}$ is equal to
    \begin{align*}
        x \longmapsto \left (1+\sum_{l=3}^{4m-1} \frac{\E[H_l(\mathbf{Q})]}{l!}H_l(x) \right) \frac{e^{-x^2/2}}{\sqrt{2\pi}}.
    \end{align*}
\end{remark}

\begin{remark}
    Assuming that $X_1$ has the same first $k-1$ moments as an element of a Wiener chaos, one could adapt the proof of Proposition \ref{prop:edgeworth_generalchaos} and get a similar conclusion.
\end{remark}
\end{appendix}

 \bibliographystyle{plain}

\begin{thebibliography}{10}
	
	\bibitem{AZ2006}
	Greg~W. Anderson and Ofer Zeitouni.
	\newblock A {CLT} for a band matrix model.
	\newblock {\em Probability Theory and Related Fields}, 134(2):283--338, 2006.
	
	\bibitem{AP2017}
	J{\"u}rgen Angst and Guillaume Poly.
	\newblock A weak {C}ram{\'e}r condition and application to {E}dgeworth
	expansions.
	\newblock {\em Electronic Journal of Probability}, 22:1--24, 2017.
	
	\bibitem{ACP2014}
	Ehsan Azmoodeh, Simon Campese, and Guillaume Poly.
	\newblock Fourth {M}oment {T}heorems for {M}arkov diffusion generators.
	\newblock {\em Journal of Functional Analysis}, 266(4):2341--2359, 2014.
	
	\bibitem{AMMP2016}
	Ehsan Azmoodeh, Dominique Malicet, Guillaume Mijoule, and Guillaume Poly.
	\newblock Generalization of the {N}ualart--{P}eccati criterion.
	\newblock {\em The Annals of Probability}, 44(2):924--954, 2016.
	
	\bibitem{Bak1994}
	Dominique Bakry.
	\newblock L'hypercontractivit{\'e} et son utilisation en th{\'e}orie des
	semi-groupes.
	\newblock In {\em Lectures on Probability Theory ({Saint--Flour}, 1992)},
	volume 1581 of {\em Lecture Notes in Mathematics}, pages 1--114. Springer,
	1994.
	
	\bibitem{BGL2013}
	Dominique Bakry, Ivan Gentil, and Michel Ledoux.
	\newblock {\em Analysis and {G}eometry of {M}arkov {D}iffusion {O}perators}.
	\newblock Springer, 2013.
	
	\bibitem{BC2016}
	Vlad Bally and Lucia Caramellino.
	\newblock Asymptotic development for the {CLT} in total variation distance.
	\newblock {\em Bernoulli}, 22(4):2442--2485, 2016.
	
	\bibitem{Bar1986}
	Andrew~D. Barbour.
	\newblock Asymptotic expansions based on smooth functions in the central limit
	theorem.
	\newblock {\em Probability Theory and Related Fields}, 72(2):289--303, 1986.
	
	\bibitem{BGZ1997}
	Vidmantas Bentkus, Friedrich G{\"o}tze, and Willem~R. van Zwet.
	\newblock An {E}dgeworth expansion for symmetric statistics.
	\newblock {\em The Annals of Statistics}, 25(2):851--896, 1997.
	
	\bibitem{BR1976}
	Rabi Bhattacharya and Ranga Rao.
	\newblock {\em Normal Approximation and Asymptotic Expansions}.
	\newblock SIAM, 1976.
	
	\bibitem{BGZ1986}
	Peter~J. Bickel, Friedrich G{\"o}tze, and Willem~R. van Zwet.
	\newblock The {E}dgeworth expansion for {U}-statistics of degree two.
	\newblock {\em The Annals of Statistics}, 14(4):1463--1484, 1986.
	
	\bibitem{BBNP2012}
	Hermine Bierm{\'e}, Aline Bonami, Ivan Nourdin, and Giovanni Peccati.
	\newblock Optimal {B}erry--{E}sseen rates on the {W}iener space: {T}he barrier
	of third and fourth cumulants.
	\newblock {\em ALEA}, 9(2):473--500, 2012.
	
	\bibitem{BG2022}
	Mindaugas Bloznelis and Friedrich G{\"o}tze.
	\newblock Edgeworth approximations for distributions of symmetric statistics.
	\newblock {\em Probability Theory and Related Fields}, 183(3):1153--1235, 2022.
	
	\bibitem{BH1991}
	Nicolas Bouleau and Francis Hirsch.
	\newblock {\em Dirichlet Forms and Analysis on the {W}iener Space}, volume~14
	of {\em De Gruyter Studies in Mathematics}.
	\newblock De Gruyter, 1991.
	
	\bibitem{BM1983}
	Peter Breuer and P{\'e}ter Major.
	\newblock Central limit theorems for non-linear functionals of {G}aussian
	fields.
	\newblock {\em Journal of Multivariate Analysis}, 13(3):425--441, 1983.
	
	\bibitem{Cra1946}
	Harald Cram{\'e}r.
	\newblock {\em Mathematical Methods of Statistics}.
	\newblock Princeton University Press, 1946.
	
	\bibitem{DH2023}
	Dmitry Dolgopyat and Yeor Hafouta.
	\newblock A {B}erry--{E}sseen theorem and {E}dgeworth expansions for uniformly
	elliptic inhomogeneous {M}arkov chains.
	\newblock {\em Probability Theory and Related Fields}, 186(1):439--476, 2023.
	
	\bibitem{Dud2002}
	Richard~M. Dudley.
	\newblock {\em Real Analysis and Probability}.
	\newblock Cambridge University Press, 2002.
	
	\bibitem{EN1999}
	Klaus-Jochen Engel and Rainer Nagel.
	\newblock {\em One-Parameter Semigroups for Linear Evolution Equations}.
	\newblock Springer, 1999.
	
	\bibitem{FL2025}
	Xiao Fang and Song-Hao Liu.
	\newblock {E}dgeworth expansion by {S}tein's method.
	\newblock {\em Bernoulli}, 31(3):2018--2041, 2025.
	
	\bibitem{Fat2021}
	Max Fathi.
	\newblock Higher-order {S}tein kernels for {G}aussian approximation.
	\newblock {\em Studia Mathematica}, 256(3):241--258, 2021.
	
	\bibitem{GH1978}
	Friedrich G{\"o}tze and Christian Hipp.
	\newblock Asymptotic expansions in the central limit theorem under moment
	conditions.
	\newblock {\em Zeitschrift f{\"u}r Wahrscheinlichkeitstheorie und Verwandte
		Gebiete}, 42(1):67--87, 1978.
	
	\bibitem{Hall1992}
	Peter Hall.
	\newblock {\em The {B}ootstrap and {E}dgeworth {E}xpansion}.
	\newblock Springer, New York, NY, 1992.
	
	\bibitem{HMP2024}
	Ronan Herry, Dominique Malicet, and Guillaume Poly.
	\newblock Superconvergence phenomenon in {W}iener chaoses.
	\newblock {\em The Annals of Probability}, 52(3):1162--1200, 2024.
	
	\bibitem{HMP2025}
	Ronan Herry, Dominique Malicet, and Guillaume Poly.
	\newblock Regularity of laws via {D}irichlet forms: application to quadratic
	forms in independent and identically distributed random variables.
	\newblock {\em Probability Theory and Related Fields}, 191(1):523--567, 2025.
	
	\bibitem{Hipp1977}
	Christian Hipp.
	\newblock {E}dgeworth expansions for integrals of smooth functions.
	\newblock {\em The Annals of Probability}, 5(6):1004--1011, 1977.
	
	\bibitem{HNTX2015}
	Yaozhong Hu, David Nualart, Samy Tindel, and Fangjun Xu.
	\newblock Density convergence in the {B}reuer--{M}ajor theorem for {G}aussian
	stationary sequences.
	\newblock {\em Bernoulli}, 21(4):2336--2350, 2015.
	
	\bibitem{JW2003}
	Bing-Yi Jing and Qiying Wang.
	\newblock {E}dgeworth expansion for {U}-statistics under minimal conditions.
	\newblock {\em The Annals of Statistics}, 31(4):1376--1391, 2003.
	
	\bibitem{KP2018}
	Yoon~Tae Kim and Hyun~Suk Park.
	\newblock An {E}dgeworth expansion for functionals of {G}aussian fields and its
	applications.
	\newblock {\em Stochastic Processes and their Applications},
	128(12):3967--3999, 2018.
	
	\bibitem{Led2012}
	Michel Ledoux.
	\newblock Chaos of a {M}arkov {O}perator and the {F}ourth {M}oment {C}riterion.
	\newblock {\em The Annals of Probability}, 40(6):2439--2459, 2012.
	
	\bibitem{MOO2010}
	Elchanan Mossel, Ryan O'Donnell, and Krzysztof Oleszkiewicz.
	\newblock Noise stability of functions with low influences: {I}nvariance and
	optimality.
	\newblock {\em Annals of Mathematics}, 171(1):295--341, 2010.
	
	\bibitem{NP2009b}
	Ivan Nourdin and Giovanni Peccati.
	\newblock {S}tein's method and exact {B}erry--{E}sseen asymptotics for
	functionals of {G}aussian fields.
	\newblock {\em The Annals of Probability}, 37(6):2231--2261, 2009.
	
	\bibitem{NP2009a}
	Ivan Nourdin and Giovanni Peccati.
	\newblock {S}tein's method on {W}iener chaos.
	\newblock {\em Probability Theory and Related Fields}, 145(1):75--118, 2009.
	
	\bibitem{NP2010}
	Ivan Nourdin and Giovanni Peccati.
	\newblock Universal {G}aussian fluctuations of non-{H}ermitian matrix
	ensembles: from weak convergence to almost sure {CLT}s.
	\newblock {\em ALEA}, 7:341--375, 2010.
	
	\bibitem{NP2012}
	Ivan Nourdin and Giovanni Peccati.
	\newblock {\em Normal {Approximations} with {Malliavin} {Calculus}: {From}
		{S}tein's {Method} to {Universality}}.
	\newblock Cambridge University Press, Cambridge, 2012.
	
	\bibitem{NP2015}
	Ivan Nourdin and Giovanni Peccati.
	\newblock The optimal fourth moment theorem.
	\newblock {\em Proceedings of the American Mathematical Society},
	143(7):3123--3133, 2015.
	
	\bibitem{NPR2010}
	Ivan Nourdin, Giovanni Peccati, and Gesine Reinert.
	\newblock Invariance principles for homogeneous sums: universality of
	{G}aussian {W}iener chaos.
	\newblock {\em The Annals of Probability}, 38(5):1947--1985, 2010.
	
	\bibitem{Nua2006}
	David Nualart.
	\newblock {\em The {M}alliavin {C}alculus and Related Topics}.
	\newblock Springer, 2006.
	
	\bibitem{NP2005}
	David Nualart and Giovanni Peccati.
	\newblock Central limit theorems for sequences of multiple stochastic
	integrals.
	\newblock {\em The Annals of Probability}, 33(1):177--193, 2005.
	
	\bibitem{PY2016}
	Mark Podolskij and Nakahiro Yoshida.
	\newblock {E}dgeworth expansion for functionals of continuous diffusion
	process.
	\newblock {\em The Annals of Applied Probability}, 26(6):3415--3455, 2016.
	
	\bibitem{RR2003}
	Yosef Rinott and Vladimir Rotar.
	\newblock On {E}dgeworth expansions for dependency-neighborhoods chain
	structures and {S}tein's method.
	\newblock {\em Probability Theory and Related Fields}, 126(4):528--570, 2003.
	
	\bibitem{Rot2005}
	Vladimir Rotar.
	\newblock {S}tein's method, {E}dgeworth's expansions and a formula of
	{B}arbour.
	\newblock In {\em {S}tein's Method and Applications}, pages 59--84. World
	Scientific, 2005.
	
	\bibitem{Sze1975}
	G{\'a}bor Szeg{\H{o}}.
	\newblock {\em Orthogonal {P}olynomials}.
	\newblock American Mathematical Society, 1975.
	
	\bibitem{TY2019}
	Ciprian~A. Tudor and Nakahiro Yoshida.
	\newblock Asymptotic expansion for vector-valued sequences of random variables
	with focus on {W}iener chaos.
	\newblock {\em Stochastic Processes and their Applications}, 129(9):3499--3526,
	2019.
	
	\bibitem{Yos2013}
	Nakahiro Yoshida.
	\newblock Martingale expansion in mixed normal limit.
	\newblock {\em Stochastic Processes and their Applications}, 123(3):887--933,
	2013.
	
\end{thebibliography}

\end{document}